\documentclass[10pt]{article}

\usepackage[english]{babel}

\usepackage{amsmath,amsfonts,amssymb,amsthm,amscd,latexsym}
\usepackage{graphicx,epsfig}
\usepackage{psfrag}
\usepackage{perpage}
\usepackage{url}
\usepackage{color}
\usepackage{bbm}
\usepackage{epstopdf}
\usepackage[utf8]{inputenc}
\usepackage[T1]{fontenc}
\usepackage{microtype}
\usepackage{hyperref}
\usepackage{mathabx}

\counterwithin{figure}{section}
\numberwithin{figure}{section}
\numberwithin{equation}{section}

\newtheorem{theorem}{Theorem}[section]
\newtheorem{lemma}[theorem]{Lemma}

\newtheorem{remark}[theorem]{Remark}

\newtheorem{conjecture}[theorem]{Conjecture}

\usepackage{tikz}
\usetikzlibrary{arrows,automata,positioning,calc,shapes,decorations.pathreplacing,decorations.markings,shapes.misc,petri,topaths}
\usepackage{pgfplots}
\pgfplotsset{compat=newest}
\usetikzlibrary{plotmarks}
\usepackage{grffile}
\newlength\figureheight
\newlength\figurewidth
\pgfplotsset{
tick label style={font=\scriptsize},
label style={font=\footnotesize},
legend style={font=\footnotesize},
every axis plot/.append style={very thick}
}

\usepackage{rotating}
\usepackage{amsbsy,enumerate}
\usepackage{graphicx}
\usepackage{ccaption}
\usepackage{comment}
\usepackage{mathrsfs} 

\newcommand{\dd}{{\rm d}}

\newcommand{\dy}{{\rm d}y}
\newcommand{\dx}{{\rm d}x}

\newcommand{\ee}{\mathbb{E}}

\newcommand{\cM}{{\mathcal{M}}}

\newcommand{\cW}{{\mathcal{W}}}

\newcommand{{\paa}[1]}{p_{00,#1}}
\newcommand{{\pab}[1]}{p_{01,#1}}
\newcommand{{\pba}[1]}{p_{10,#1}}
\newcommand{{\pbb}[1]}{p_{11,#1}}

\newcommand{\eee}{{\rm e}}

\makeatletter
\renewcommand{\fnum@figure}[1]{\textbf{\figurename~\thefigure}. }
\renewcommand{\fnum@table}[1]{\textbf{\tablename~\thetable}. }
\makeatother

\allowdisplaybreaks

\begin{document}


\title{Opinion dynamics on dense dynamic random graphs}
  
\author{S.\ Baldassarri$^1$, P.\ Braunsteins$^2$, F.\ den Hollander$^1$, M.\ Mandjes$^1$}

\date{}

\maketitle

\maketitle                              

\begin{abstract}
We consider two-opinion voter models on dense dynamic random graphs. Our goal is to understand and describe the occurrence of \emph{consensus} versus \emph{polarisation} over long periods of time. The former means that all vertices have the same opinion, the latter means that the vertices split into two communities with different opinions and few disagreeing edges. We consider three models for the joint dynamics of opinions and graphs: one with a one-way feedback and two which are \emph{co-evolutionary}, i.e., with a two-way feedback. In the first model only \emph{coexistence} is attainable, meaning that both opinions survive, but with the presence of \emph{many} disagreeing edges. In the second model only consensus prevails, while in the third model polarisation is possible. Our main results are functional laws of large numbers for the densities of the two opinions, functional laws of large numbers for the dynamic random graphs in the space of graphons, and a characterisation of the limiting densities in terms of Beta-distributions. Our results are supported by simulations. To prove our results we develop a novel method that involves coupling the co-evolutionary process to a \emph{mimicking process} with one-way feedback. We expect that this method can be extended to other dense co-evolutionary models. 

{\it Acknowledgment.} The work in this paper was supported by the European Union’s Horizon 2020 research and innovation programme under the Marie Skłodowska-Curie grant agreement no.\ 101034253, and by the NWO Gravitation project NETWORKS under grant no.\ 024.002.003.\\ 

\includegraphics[height=3em]{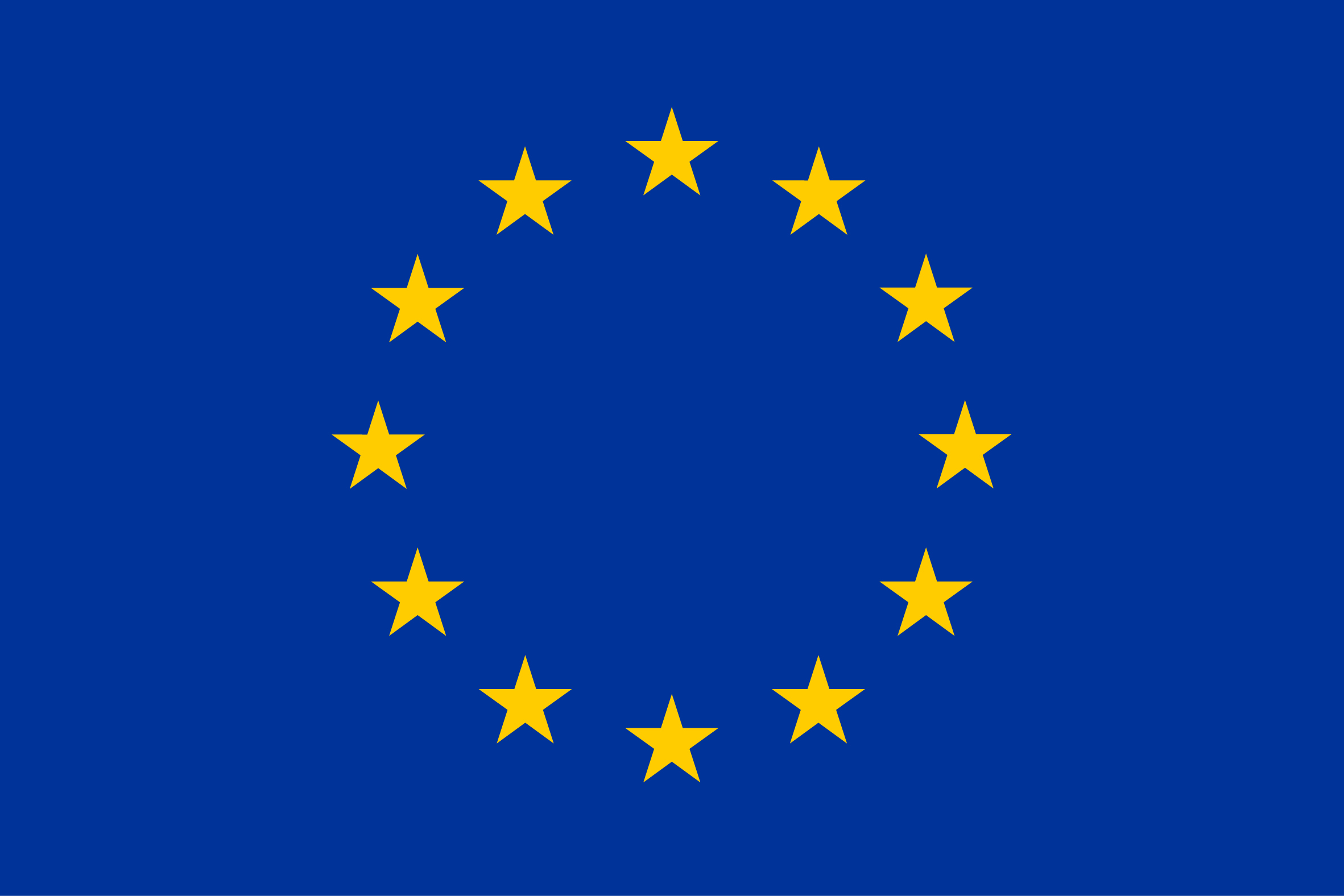} 
\end{abstract}

\bigskip

\footnoterule
\noindent
\hspace*{0.3cm} {\footnotesize $^{1}$ 
Mathematisch Instituut, Universiteit Leiden, Einsteinweg 55 1, 2333 CC  Leiden, The Netherlands,\\
{\tt \{s.baldassarri,denholla,m.r.h.mandjes\}@math.leidenuniv.nl}}\\
\hspace*{0.3cm} {\footnotesize $^{2}$
School of Mathematics and Statistics, Anita B.\ Lawrence Centre, UNSW Sydney, Sydney NSW 2052, Australia,
{\tt p.braunsteins@unsw.edu.au}}

\newpage

\footnotesize
\tableofcontents
\normalsize

\newpage


\section{Introduction}
\label{sec:Int}


\subsection{Background and motivation}

Describing the evolution of dynamic networks together with dynamic processes running on them constitutes a major challenge in network science. Despite considerable efforts in past years, and notable progress on an intuitive and approximative level, our \emph{mathematical} understanding of such systems is still in its infancy. The current mathematical literature consists of only a handful of scattered results. Examples are: random walks on dynamical percolation \cite{PSS15}, \cite{PSS20}, random walks on graphs with dynamic rewiring \cite{AGHdH18}, random walks on graphs with dynamic switching \cite{ST20}, \cite{AGHdH19}, the contact process on dynamic random graphs \cite{JM17}, \cite{CSW22},  the voter model on dynamic random graphs \cite{BS17}, \cite{CR18}, \cite{CDFR18}, \cite{BdHM22pr}, \cite{RZ23}, \cite{ABHdHQ24}, \cite{ACHQ23pr}, \cite{ABHdHQ24pr}, \cite{Cap2024}. Since double dynamics are prevalent in most real-world networks, there is an urgent need for breakthroughs. 

The \emph{voter model}, initially introduced as an interacting particle system on a lattice \cite{HL75}, \cite{L85}, can be used to model the evolution of opinions in a social network. The objective is to study the way in which individuals interact and change their opinion. Empirical studies have revealed the presence of evolution patterns in which the individuals split into communities with different opinions, and few disagreeing edges, that remain stable over long periods of time, a phenomenon commonly referred to as \emph{polarisation}. However, in voter models on static networks eventually all the individuals hold the same opinion, a situation referred to as \emph{consensus}. There is strong evidence that polarisation can only occur in a dynamic network, in the sense that the interplay between the opinion dynamics and the network dynamics plays a crucial role. This interplay, in the literature referred to as \emph{co-evolution}, leads to challenging mathematical questions, as a consequence of the intrinsically complex joint dynamics.

The last few years have witnessed a growing number of papers describing polarisation effects for co-evolutionary random networks via simulations and heuristic arguments, for classic voter models \cite{MKF03}, \cite{GZ06}, \cite{ZG06}, \cite{PTN06}, \cite{RG17}, \cite{RM20}, \cite{WW23pr}, \cite{BVP24}, nonlinear voter models with discrete opinions \cite{MM19}, \cite{LHAJS20}, \cite{M23}, \cite{JTSZH20}, as well as their extension to continuous opinions \cite{KB08a}, \cite{KB08b}, \cite{IKKB09}, \cite{HPZ11}. These polarisation effects can be mitigated or promoted by the rules according to which the opinions and the network evolve. Indeed, empirical studies \cite{DSCSQ17}, \cite{SCPCFM21}, \cite{ACAHFM22} show that social influence alone does not affect the network structure, so that opinions may become polarised only after a long time, while network rewiring alone makes polarisation inevitable (like-minded individuals cluster together). The proliferation of polarisation may be an inevitable outcome of the human tendency to be influenced by information and opinions which one is exposed to, and that of disliking disagreeable social ties. Social influence and rewiring appear to provide synergistic conditions for the rapid formation of completely segregated and polarised clusters.

In \cite{AdHR24}, a model is analysed where vertices decide to update at a rate proportional to the number of incident open edges, and do so by adopting the opinion of the vertex at the other end. Edges decide to update at a positive rate, and do so by switching between closed and open with a probability that depends on their status and on whether the vertices at their ends are concordant or discordant.
Coupled evolution equations are derived, in the dense graph limit, for the density of vertices with a given opinion and the density of open edges.

Motivated by the above developments, in the present paper the focus is on \emph{two-opinion voter models} on \emph{dynamic dense random graphs}. We consider three models, the first with a one-way feedback (Section~\ref{sec:oneway}), the second and third with a two-way feedback (Sections~\ref{sec:twoway1} and~\ref{sec:twoway2}). In all three models the dynamics of the edges depends on the dynamics of the vertices, but the reverse dependence is present only in the second and third model. It turns out that polarisation is possible only when there is a two-way interaction, but the two-way interaction does not guarantee polarisation. Indeed, the second model always reaches consensus unless we tweak it in a way that, unfortunately, makes it mathematically intractable. To overcome this hurdle, we introduce a third model that we are able to analyse mathematically and that goes through polarisation as well. 

To understand the dynamics of opinions and network for voter models with one-way and two-way feedback, we establish a functional law of large numbers. This functional law captures the evolution of the proportion of individuals holding each opinion, the firmness with which individuals hold their opinions (defined through ``vertex types''), and the network of connections. We use the functional law of large numbers to determine the parameter values for which the models move quickly toward consensus, and the parameter values for which both opinions co-exist for an extended period of time. In the latter case, we prove that the distribution of vertex types (which measure how firmly individuals hold their opinion) converges to a Beta-distribution when the time and the number of individuals in the network are large. While the three models we consider are hard to analyse in detail, and some questions remain open, with the help of supporting simulations we manage to develop a \emph{full picture} of their behaviour. 

A key novelty of the paper is that we establish the functional law of large numbers for the network of connections in the space of \emph{graphons} (see Appendix \ref{appA}). Indeed, while several authors have established limits of graph-valued processes in the space of graphons (see, for instance, \cite{CK20}, \cite{C16}, \cite{G22}, \cite{R}, \cite{RZ23}), to the best of our knowledge the present paper and \cite{AdHR24} are the first to do so for a \emph{co-evolutionary network}. In proving our results we construct a framework that is robust, in the sense that we believe that it can be applied to several other co-evolutionary networks of interest. 


\subsection{General co--evolutionary models}
\label{sec:genstrategy}

In this section we introduce a co-evolutionary network as a graph-valued process in order to provide a general framework for understanding how the graph and the process running on the graph evolve together. For a given time horizon $T>0$, let $(G_n(t))_{t \in [0,T]}$ denote the graph-valued process and $x_i(t) \in S$ denote the (discrete) quantity of interest (e.g.\, opinion, spin, number of particles) at each vertex $i \in [n]$ at time $t \in [0,T]$.

In particular, our focus is on providing a framework for establishing convergence in the space of graphons (see Appendix \ref{appA}).
Steps 1--3 below, which largely follows \cite{BdHM22pr}, describe how to establish convergence in the space of graphons for a model with one-way dependence (i.e., where the vertices evolve independently of the edges, but the edges depend on the vertices), while Step 4 below describes how this framework may be repurposed to co-evolutionary networks.

\medskip
\noindent
\underline{\emph{Step 1: Define vertex types and establish convergence of the empirical type process.}}
We assign to vertex $i$ at time $t$ a random variable $y_i(t) \in [0,1]$ that may fluctuate over time. The definition of $y_i(t)$ must be chosen carefully (see Step 2 below) and must be a functional of the path of the process of interest $(x_i(s))_{0 \leq s \leq t}$ (see \eqref{eqn:ydef} for an example). We refer to $X_i(t)=(x_i(t),y_i(t))$ as the \emph{generalised type} of vertex $i$ at time $t$, and we let $(F_n(t;\cdot))_{t\in[0,T]}$ denote the {\it empirical type process} characterised by
\[
F_n(t;\cdot) = \dfrac{1}{n} \sum_{i=1}^n \mathbf{1}\{X_i(t) \in \cdot\},
\]
which lives on $D(\cM(S \times [0,1]),[0,T])$, the space of $\cM(S \times [0,1])$-valued c\`adl\`ag paths over the time interval $[0,T]$, where $\cM(S \times [0,1])$ is the space of measures on $S \times [0,1]$ endowed with the topology of weak convergence. In order to show convergence of the empirical graphon process, we require the following assumption.

\medskip\noindent
{\bf Assumption 1.} Let $\Rightarrow$ denote convergence in distribution. We assume that $F_n\Rightarrow F$ as $n\to\infty$ on $D(\cM(S \times [0,1]),[0,T])$, that is, the empirical type process satisfies a stochastic-process limit. \hfill$\spadesuit$

\medskip
Suppose that, for $i \in [n]$, $(x_i(t),y_i(t))_{t \in [0,T]}$ are \emph{independent and identically distributed Markov processes}. Then the limiting path $F$, as $n$ grows large, can be determined by using the \emph{Kolmogorov forward equations} of the Markov process. Indeed, at any time $t$, the empirical distribution function $F_n(t;\cdot)$ converges to the distribution function $F(t;\cdot)$ of $(x_i(t),y_i(t))$, which can be computed by using the Kolmogorov forward equations. Under minor regularity conditions tightness can be established, thereby verifying Assumption 1. 

The infinitesimal generator $\cal L$ of the continuous-time Markov process $(X_i(t))_{t \geq 0} = (x_i(t),y_i(t))_{t\geq0}$ with state space $S\times[0,1]$ is defined through its action on a set of sufficiently smooth test functions $f\colon\, S \times [0,1] \mapsto \mathbb{R}$ as
\begin{equation}
\label{eq:generator}
({\cal L} f)(x,y) = \lim_{t \to 0} \dfrac{\ee[f(X_t)| X_0=(x,y)]-f(x,y)}{t}, \qquad x\in S, \quad y\in[0,1].
\end{equation}
When the process $(y_i(t))_{t\geq0}$ is a deterministic functional of the path $(x_i(t))_{t\geq0}$, which is the specific situation that we consider below, the latter can be written as \cite{EK86}
\begin{equation}
\label{eq:gendecomp}
({\cal L} f)(x,y) = ({\cal L}^S f)(x,y) + ({\cal L}^{[0,1]} f)(x,y),
\end{equation}
where ${\cal L}^S$, ${\cal L}^{[0,1]}$ are the generators associated to the processes $(x_i(t))_{t \geq 0}$, $(y_i(t))_{t \geq 0}$, respectively, which take the general form
\begin{equation}
\label{eq:genu}
\begin{array}{lll}
({\cal L}^{S} f)(x,y) 
&=& \displaystyle\sum_{x'\in S\setminus\{x\}} q_{x,x'} [f(x', y)-f(x,y)], \\[0.3cm]
({\cal L}^{[0,1]} f)(x,y) 
&=& b(x,y) \dfrac{\partial}{\partial y} f(x,y) + \dfrac{1}{2} \sigma^2(x,y) \dfrac{\partial^2}{\partial y^2} f(x,y),
\end{array}
\end{equation}
where $q_{x,x'}$ denotes the rate at which the process $(x_i(t))_{t\geq0}$ transitions from $x$ to $x'$, and $b(\cdot, \cdot)$ and $\sigma^2(\cdot, \cdot)$ denote the drift and the diffusion term, respectively. Once the generator of the process is given, then we follow the standard procedure  to derive the Kolmogorov forward equations \cite{EK86}. This procedure amounts to first expressing the densities $\mathbb{P}(X_i(t)\in(\cdot,\dd u))$ in terms of the generator ${\cal L}$, and then writing down a difference equation for the densities $\mathbb{P}(X_i(t+\Delta t)\in(\cdot,\dd u))$ for small $\Delta t$ by using the Markov property of the process and a Taylor series expansion, which eventually enables us to identify the associated differential equation by letting the time increment $\Delta t$ tend to zero.
In Sections~\ref{sec:oneway}, \ref{sec:twoway1} and \ref{sec:twoway2}, we will follow this procedure to derive a system of PDEs describing the time evolution of the densities of interests for the models that we consider.

\medskip
\noindent
\underline{\emph{Step 2: Express the edge connection probability in terms of the types:}} 
When the definition of $y_i(t)$ is chosen carefully, we may be able to express the probability that edge $ij$ is active in terms of $y_i(t)$, $y_j(t)$, and the path of the empirical distribution up to time $t$, $(F_n(s;\cdot))_{s \in [0,t]}$. This motivates the following assumption.

\medskip\noindent
{\bf Assumption 2.} At any time $t$, edge $ij$ is active with probability
\begin{equation}
\label{eq:genH}
H(t; y_i(t), y_j(t), (F_n(t;\cdot))_{t\in[0,T]}),
\end{equation}
conditionally independently of all the other edges given $y_i(t)$, $y_j(t)$ and $(F_n(s;\cdot))_{s\in[0,t]}$, where
\[
H\colon\, [0,T] \times [0,1]^2 \times D(\cM(S \times [0,1]),[0,T]) \mapsto [0,1]. 
\]
\hfill $\spadesuit$

\medskip
Assumption 2 implies that, while the probability that edge $ij$ is active at time $t$ can depend on the \emph{path} of the empirical type distribution, it depends on the types of the vertices $i$ and $j$ \emph{at time $t$ only}. In practice this means that, when computing the edge probability at time $t$, we need to be able to summarise the \emph{paths} of the quantities of interest $(x_i(t))_{t \in [0,T]}$ and $(x_j(t))_{t \in [0,T]}$ as numbers $y_i(t)$ and $y_j(t)$ through a single functional (see \eqref{eqn:typedef} for an example). We remark that in Section \ref{sec:nonlin} we will encounter a graph-valued process for which Assumption 2 fails.

\medskip
\noindent
\underline{\emph{Step 3: Establish convergence of the graphon valued process:}} The function $H$ gives rise to an {\it induced reference graphon process} $g^{[F]}$ characterised by
\begin{equation}
\label{eq:genrefgraphon}
g^{[F]}(t; x, y) = H\big(t; \bar{F}(t;x), \bar{F}(t; y), (F(t;\cdot))_{t\in[0,T]}\big),
\end{equation}
where $\bar{F}(t; \cdot)$ denotes the right-continuous generalised inverse of the distribution function $F$ with support $[0,1]$, i.e.,
\[
\bar F(t;u) = \inf\{x\in[0,1]\colon\, F(t;x)>u\}.
\]
If we let $F$ be the limit described in Assumption 1, then $g^{[F]}$ is our candidate limit for the graph-valued process $G_n(t)$ in the space of graphons. To show that this is indeed the correct limit, we use the continuous mapping theorem, for which we need the following assumption. 

\medskip\noindent
{\bf Assumption 3.} The map $F\mapsto g^{[F]}$ from $D(\cM(S \times [0,1]),[0,T])$ to $D((\mathcal{W},d_{\square}),[0,T])$ is continuous, where $\mathcal{W}$ and $d_{\square}$ are defined in Appendix \ref{appA}. \hfill $\spadesuit$

\medskip
Note that Assumption 3 is an assumption on the edge connection function $H$, which effectively requires that $H$ must be continuous in all its arguments. This means that two vertices $i$ and $j$ with similar types $y_i(t) \approx y_j(t)$ must have a similar edge connection probability with a third vertex $k$, regardless of the value of its type $y_k(t)$. With Assumptions 1-3 met we can then apply \cite[Theorem 3.10]{BdHM22pr} to establish convergence.

\medskip
\noindent
\underline{\emph{Step 4: Extension to co-evolutionary models:}} 
Recall that the goal of the present paper is to prove that $h^{G_n}\Rightarrow g^{[F]}$ as $n\to\infty$ in the space $D((\mathcal{W},d_{\square}),[0,T])$, where $h^{G_n}$ is the empirical graphon associated with the \emph{co-evolutionary process} $(G_n(t))_{t \in [0,T]}$ (see \eqref{eq:graphon}). While the framework above can be readily applied to models with a one-way feedback, there appears to be no direct way to apply it to models with a two-way feedback. Indeed, Assumption 1 becomes restrictive because we may expect that $(x_i(t),y_i(t))_{i \in [n]}$ are dependent. Also,  Assumption 2 becomes restrictive because we may expect that edges are not  conditionally independent given the types of the vertices.

One of the primary contributions of this paper is to demonstrate that the above framework can still be applied after approximating the co-evolutionary process by an intermediary \emph{mimicking process} with a one-way-dependence for which Assumptions 1-3 can be verified. 
We then apply the framework above to the mimicking process, and couple the mimicking process with the co-evolutionary process in such a way that discrepancies during the time interval $[0,T]$ have a sufficiently small probability.

To understand why such a mimicking process often exists, recall that in the present paper we consider \emph{dense} graph-valued processes. This means that there are $n$ vertices and ${n \choose 2} = O({n^2})$ edges. Hence, informally, each vertex is adjacent to order $n$ active edges. Consequently,  as $n \to \infty$ a law of large numbers kicks in at each vertex and, depending on the dynamics of the model, we may pretend the edges to behave effectively deterministically. In that case there is no longer a feedback loop between the randomness of the vertices and the randomness of the edges. It is this intuition that we rely on to construct a mimicking process that has only a one-way dependence, but that sufficiently closely approximates the co-evolutionary model with a two-way feedback. 

Note that the implementation of the above framework heavily depends on the specific model we analyse. In what follows we focus on three voter models on dynamic random graphs, introduced in Sections \ref{sec:oneway}, \ref{sec:twoway1}, and \ref{sec:twoway2}.


\subsection{Key definitions}
\label{sec:3models}

In this section we collect the key quantities that are needed to describe our three models. At any time $t\geqslant 0$, the set of vertices is $V = [n] = \{1,\ldots,n\}$ and any vertex holds opinion either $+$ or $-$. We assume that each edge is initially active with probability $p_0$, independently of everything else (i.e., we start from an Erd\H{o}s-R\'enyi random graph on $n$ vertices with `edge retention probability' $p_0$). Let 
\[
x_i(t) = \mbox{the \emph{opinion} of vertex $i$ at time $t$},
\]
and put
\begin{equation}
\label{eqn:ydef}
y_i(t) = \eee^{-t}y_i(0) + \int_{0}^t {\rm d}s\,{\rm e}^{-s}\, \mathbf{1}\{ x_i(t-s) = + \}.
\end{equation}
We let 
\begin{equation}
\label{eqn:typedef}
y_i(t) - \eee^{-t}y_i(0) = \int_{0}^t {\rm d}s\,{\rm e}^{-s}\, \mathbf{1}\{ x_i(t-s) = + \} = \mbox{ the \emph{type} of vertex $i$ at time $t$},
\end{equation}
which encodes how much time vertex $i$ has held opinion $+$ up to time $t$. Note that the type of each vertex at time $t=0$ equals 0.
We define 
\[
X_i(t)=(x_i(t),y_i(t)) = \mbox{the \emph{generalised type} of vertex $i$ at time $t$}.
\]
The generalised type belongs to the space $\{-,+\} \times [0,1]$ and keeps track of the evolution of the opinion. We assume that $(X_i(0))_{i \in [n]}$ are independent and identically distributed with a prescribed density (see the discussion in Section \ref{sec:outline}). In particular, for technical reasons we assume that the density of $X_i(0)$ is analytic for any $i\in[n]$.

The associated {\it empirical type process} is
\begin{equation}\label{eq:emptype}
F_n(t;+, u) = \frac{1}{n} \sum_{i=1}^n \mathbf{1}\{X_i(t) \in (+, [0,u])\},
\end{equation}
which lives on $D(\cM([0,1]),[0,T])$, the space of $\cM([0,1])$-valued c\`adl\`ag paths over the time interval $[0,T]$. To verify Assumption 1, we seek convergence to a limiting empirical type process $F(t;+, u)$ as $n\to\infty$.

\medskip\noindent
{\bf Assumption 4.} The labels of the vertices are updated dynamically so that, at any time $t$, they are ordered lexicographically: all opinion $+$ vertices have a lower label than opinion $-$ vertices, and between vertices with the same opinion the vertices are ordered by increasing type $y_i(t) - \eee^{-t}y_i(0)$.
\hfill $\spadesuit$

\medskip
In all models considered, we let the dynamic labelling of Assumption 4 be in force, meaning that we do not need to verify it. Under Assumption 4 we are able to write down a candidate limit $g^{[F]}$ for $h^{G_n}$ in the space of graphons, $\mathcal{W}$ defined in Appendix \ref{appA}. To do this, we let
\begin{equation}\label{eq:densities}
f_+(t,u)\,{\rm d}u = \mathbb{P}(X_i(t) \in (+,{\rm d}u) ), \qquad 
f_-(t,u)\,{\rm d}u = \mathbb{P}(X_i(t) \in (-,{\rm d}u) ),
\end{equation}
and define
\begin{equation}
\label{eqn:invdef}
\bar F(t;y) = \begin{cases}
\inf \{ s \in [0,1]\colon\, F(t; +, s) \geq y \}, \quad &\text{ if } 0 \leq y \leq F(t;+,1), \\
\inf\{s \in [0,1]\colon\, F(t; +, 1) + F(t; -, s) \geq y \}, \quad &\text{ otherwise},
\end{cases}
\end{equation}
where
$$
F(t; +, u) = \int_{0}^u f_+(t,s)\,{\rm d}s, \qquad F(t; -, u) = \int_{0}^u f_-(t,s)\,{\rm d}s,
$$
and define
\begin{equation}
\label{eqn:empgra}
g^{[F]}(t;x,y) = H(t; \bar F(t;x), \bar F(t; y)),
\end{equation}
where $H(t; u,v)$ is the probability that there is an active edge connecting two vertices with type $u$ and $v$ at time $t$. We refer to $g^{[F]}$ as the {\it induced reference graphon process} for $F\in D(\cM([0,1]),[0,T])$. 


\subsection{Outline}
\label{sec:outline}

We will consider three different models of increasing complexity for the joint dynamics of opinions and graphs: one with one-way feedback and two with a two-way feedback. 
\begin{itemize}
  \item   
In the first model, each edge is equipped with a Poisson clock having the same rate, and when it rings the selected edge is active with a probability depending on the opinion of the vertices at its ends. For this model we prove three theorems. The first theorem concerns the \emph{convergence} as $n\to\infty$ of the empirical graphon-process associated to the graph-valued process $(G_n(t))_{t \in [0,T]}$ to the induced reference graphon-valued process $g^{[F]}$. The second theorem concerns the \emph{existence and uniqueness} of the induced reference graphon process. For this model we in addition establish an abstract representation of $g^{[F]}$. The third theorem serves to provide insight into structure of the solution in the limit as $t \to \infty$.
\item In the second model, the edge dynamics is the same as that of the first model. The vertex dynamics is different, though: each vertex is equipped with a Poisson clock having the same rate, and when the clock rings the vertex selects one of its neighbours uniformly at random and copies its opinion. We prove three theorems the correspond to those for the first model, but now the second theorem concerns only the {\it local uniqueness} of the induced reference graphon process $g^{[F]}$ and without its representation.
\item In the third model, there are $2q$ copies of the second model, for some $q\in{\mathbb N}$. Of these, $q$ are referred to as $g$-graphs and $q$ as $r$-graphs. These graphs evolve independently in parallel, but $g$-graphs and $r$-graphs have different parameters. The opinions of the vertices are the same in all the copies, but the states of the edges are different. The true graph is a function of the $g$-graphs and the $r$-graphs. The main idea behind the third model is that it is having a lower probability of connecting voters holding different opinions. We prove similar results as for the second model.
\end{itemize}
We support our theorems with numerical simulations.

For technical reasons we will restrict ourselves to \emph{analytic initial conditions}, i.e., we assume that the initial densities $f_+(0,\cdot)$ and $f_-(0,\cdot)$ are analytic, leaving the possible occurrence of singularities aside. This is the reason why we assume that the density of $X_i(0)$, $i\in[n]$, is analytic. We will also refrain from carrying through a stability analysis of the equilibrium solution, which is beyond the scope of the present paper.  

The remainder of this paper is organised as follows. In Sections~\ref{sec:oneway}, \ref{sec:twoway1} and \ref{sec:twoway2} we give the definition of the first, second and third model, respectively, state the main theorems, offer numerical simulations that illustrate the theorems, and provide a proof of the theorems. Sections~\ref{sec:oneway}, \ref{sec:twoway1} and \ref{sec:twoway2} can be read essentially independently. In Section~\ref{sec:conclusions} we draw our conclusions. Appendix~\ref{appA} collects some basic facts about graphons and Hoeffding's inequality, while Appendices~\ref{appB} and~\ref{appC} contain some technicalities that arise in the proofs of Sections \ref{sec:oneway} and \ref{sec:twoway1}, respectively.


\section{First model: one-way feedback and coexistence}
\label{sec:oneway}

The first model, with one-way feedback, is characterised by the following dynamics.
\begin{itemize}
\item
{\it Vertex dynamics.} If a vertex holds opinion $-$ (respectively, $+$), then it switches its opinion to $+$ at rate $\gamma_{-+}$ (respectively, $\gamma_{+-}$). Both happen independently of everything else.
\item
{\it Edges dynamics.} Each edge is re-sampled at rate $1$, i.e., a rate-$1$ Poisson clock is attached to each edge and when the clock rings the edge is active with a probability that depends on the current opinion of the two connected vertices: with probability $\pi_+$ if the two vertices hold opinion $+$, with probability $\pi_-$ if the two vertices hold opinion $-$, and otherwise with probability $\tfrac12(\pi_++\pi_-)$.
\end{itemize}

Section~\ref{mod1:prep} steps through the framework in Section \ref{sec:genstrategy}, Section~\ref{mod1:mr} states the main results, Section~\ref{mod1:num} offers simulations, Sections~\ref{mod1:pr}--\ref{sec:nodepu} provide proofs and computations, respectively.


\subsection{Preparations}
\label{mod1:prep}

\medskip
\noindent
\emph{\underline{Step 1:}} Observe that, due to the one-way dependence of the model, the processes $((x_i(t),y_i(t))_{t \in [0,T]})_{i \in [n]}$ are independent and identically distributed. Consequently, for any $t \in [0,T]$
\begin{align*}
F_n(t; +,u) \to F(t,+,u) &= \int_{0}^u f_+(t,s)\,{\rm d}s, \qquad n \to \infty, \\
F_n(t; -,u) \to F(t,-,u) &= \int_{0}^u f_-(t,s)\,{\rm d}s, \qquad n \to \infty,
\end{align*}
where $F$ is the limit of the empirical type distribution in Assumption 1 and $f_+(\cdot,\cdot), f_-(\cdot,\cdot)$ are defined in \eqref{eq:densities}. To characterise $F$ we thus use the Kolmogorov forward equations of the Markov process $(x_i(t),y_i(t))_{t \in [0,T]}$. The infinitesimal generator of this process \eqref{eq:gendecomp} can be written as
\begin{equation}
	\label{eq:gen1}
	({\cal L} f) (x,y) = (\gamma_{+-}\mathbf{1}\{x=+\}+\gamma_{-+}\mathbf{1}\{x=-\})[f(x',y)-f(x,y)]
	+ b(x,y) \dfrac{\partial}{\partial y} f(x,y),
\end{equation}
where $x'$ is the opinion of the selected vertex after switching its initial opinion $x$ and $b(x,\cdot)$ is the drift term when starting from $x\in\{-,+\}$ that will be characterised later (see \eqref{eq:b+}-\eqref{eq:b-} below). This gives rise to the Kolmogorov forward equations in \eqref{eq:PDEs} below that characterise $F$ and verifies Assumption~1.

\medskip\noindent
\emph{\underline{Steps 2 and 3:}} Given the paths of $x_i(\cdot)$ and $x_j(\cdot)$, the probability that edge $ij$ is active at time $t$ is
$$
p_{ij}(t) = {\rm e}^{-t}p_{0} + \int^t_0 {\rm d}s \, {\rm e}^{-s} \tfrac{1}{2} \left[ \pi_{x_i(t-s)} +  \pi_{x_j(t-s)} \right]. 
$$
This equation can be understood by noting the following:
\begin{itemize}
\item 
The probability that the Poisson clock attached to edge $ij$ rings during the time interval $[t-s, t-s +{\rm d}s]$ is ${\rm d}s$.
\item 
The probability that this is the last time the Poisson clock attached to edge $ij$ rings during the time interval $[0,t]$ is ${\rm e}^{-s}$.
\item 
If the Poisson clock attached to edge $ij$ rings during the interval $[t-s, t-s + {\rm d}s]$, then the probability that edge $ij$ is active is $\frac{1}{2} \left[ \pi_{x_i(t-s)} +  \pi_{x_j(t-s)} \right]$.
\item 
Integrate over $s$ to obtain the formula for $p_{ij}(t)$.
\end{itemize}
Using that $\int_{0}^t {\rm d} s\,{\rm e}^{-s} = 1 - {\rm e}^{-t}$, we can write
\begin{align*}
p_{ij}(t) 
&= {\rm e}^{-t} p_{0} + \int^t_0 {\rm d}s\,{\rm e}^{-s} \tfrac{1}{2} \left[ \pi_{x_i(t-s)} +  \pi_{x_j(t-s)} \right] \\
&= {\rm e}^{-t} p_{0} + \int^t_0 {\rm d}s\,{\rm e}^{-s} \tfrac{1}{2} \big[ \pi_+ \mathbf{1}\{ x_i(t-s) = + \} + \pi_- \mathbf{1}\{ x_i(t-s) = - \} \\
&\qquad \qquad \qquad + \pi_+ \mathbf{1}\{ x_j(t-s) = + \} + \pi_- \mathbf{1}\{ x_j(t-s) = + \} \big] \\
&= {\rm e}^{-t}p_{0} + \int^t_0 {\rm d}s\,{\rm e}^{-s} \tfrac{1}{2} \big[ \pi_+ \mathbf{1}\{ x_i(t-s) = + \} + \pi_- (1- \mathbf{1}\{ x_i(t-s) = + \}) \\
&\qquad \qquad \qquad + \pi_+ \mathbf{1}\{ x_j(t-s) = + \} + \pi_- (1-\mathbf{1}\{ x_j(t-s) = + \}) \big]  \\[0.2cm]
&= {\rm e}^{-t}p_0 + \tfrac{1}{2} \big[ \pi_+  (y_i(t) - \eee^{-t}y_i(0)) + \pi_-(1-{\rm e}^{-t} - (y_i(t) - \eee^{-t}y_i(0))) \\[0.2cm] 
&\qquad \qquad \qquad +\pi_+ (y_j(t) - \eee^{-t}y_i(0)) + \pi_-(1-{\rm e}^{-t} - (y_j(t) - \eee^{-t}y_i(0))) \big].
\end{align*}
Consequently, the probability $H(t;u,v)$ of having an active edge between two vertices with type $u$ and $v$ at time $t$ is given by 
\begin{equation}
\label{eqn:Hdef}
H(t;u,v) = {\rm e}^{-t} p_0 + \tfrac{1}{2}\left[ \pi_+ u + \pi_-(1-{\rm e}^{-t} - u) +  \pi_+  v +  \pi_- (1-{\rm e}^{-t} - v) \right].
\end{equation}
This verifies Assumption 2. Note that $H$ is continuous, which implies that Assumption 3 also holds.

\begin{remark}
{\rm Note that when $\pi_+=\pi_-=\pi$, then the function $H$ takes the form
$$
H(t;u,v) = \eee^{-t} p_0 + \pi(1-\eee^{-t}),
$$
so that the probability of having an active edge between two vertices at any time is the same. This means that the graphon in \eqref{eqn:empgra} is constant, i.e., is the graphon associated to an Erd\H{o}s-R\'enyi random graph.}
\end{remark}


\subsection{Main results: Theorems~\ref{thm:graphonconv}--\ref{thm:systempdeasymp}}
\label{mod1:mr}

Our first theorem shows convergence of the empirical graphon dynamics associated with our graph-valued process $(G_n(t))_{t \in [0,T]}$ to the induced reference graphon process $g^{[F]}$.

\begin{theorem}
\label{thm:graphonconv}
$h^{G_n}\Rightarrow g^{[F]}$ as $n\to\infty$ in the space $D((\mathcal{W},d_{\square}),[0,T])$, where $h^{G_n}$ is the empirical graphon associated with $G_n$ defined in \eqref{eq:graphon}.
\end{theorem}

Our second theorem allows for an identification of the induced reference graphon process, albeit an implicit one, by providing expressions for the densities $f_+(t,u)$ and $f_-(t,u)$. Put
\begin{equation}
\label{eq.vecf}
\vec{v}(t,u) = \left(\begin{matrix} f_+(t,u) \\ f_-(t,u)\end{matrix}\right).
\end{equation}

\begin{theorem}
\label{thm:systempde}
For every $t\in[0,T]$ and $u\in[0,1]$,
$$
\vec{v}(t,u) = \eee^{t\left(N+M(u)\frac{\partial}{\partial u}\right)} \left(\begin{matrix} g_+(u) \\ g_-(u)\end{matrix}\right),
$$
where 
\begin{equation}
\label{eq:defM}
N = \left(\begin{matrix} 1-\gamma_{+-} & \gamma_{-+} \\ \gamma_{+-} & 1-\gamma_{-+} \end{matrix}\right), \qquad 
M(u) = \left(\begin{matrix} u-1 & 0 \\ 0 & u \end{matrix}\right),
\end{equation}
and $g_+(\cdot),g_-(\cdot)$ are the initial conditions given by $g_+(u)=f_+(0,u)$ and $g_-(u)=f_-(0,u)$.
\end{theorem}
Define the number of vertices having opinion $+$ and $-$ at time $t$ as
\[
N_+(t)=\sum_{i=1}^n \mathbf{1}\{ x_i(t) = + \}, \qquad N_-(t)=\sum_{i=1}^n \mathbf{1}\{ x_i(t) = - \},
\]
respectively. Our third theorem investigates the asymptotic behaviour of $N_+(t)$, $N_-(t)$ and $\vec{v}(t,u)$ when $n,t\to\infty$.

\begin{theorem}
\label{thm:systempdeasymp}
The following statements hold.
\begin{itemize}
	\item[(i)] For any $t\in[0,T]$,
	\[
	\mathbb{P}\left(\lim_{n\to\infty} \dfrac{N_+(t)}{n}= \mathbb{P}(x_1(t)=+)\right)=1, \qquad 
	\mathbb{P}\left(\lim_{n\to\infty} \dfrac{N_-(t)}{n}= \mathbb{P}(x_1(t)=-)\right)=1.
	\]
	In particular, as $n\to\infty$ followed by $t\to\infty$, almost surely $N_+(t)/n$ converges to $\gamma_{-+}/(\gamma_{+-} + \gamma_{-+})$ and $N_-(t)/n$ converges to $\gamma_{+-}/(\gamma_{+-} + \gamma_{-+})$, i.e., the limiting proportion of vertices holding opinion $+$ and $-$ are given by the stationary distribution of a Markov chain on the state space $\{+,-\}$ with switching rate $\gamma_{+-}$ and $\gamma_{-+}$, respectively.
	\item[(ii)] For any $u\in[0,1]$, 
	\begin{equation}
		\label{eq:limdensities}
	\lim_{t\to\infty}\vec{v}(t,u) := \left(\begin{matrix} f_+(\infty,u) \\ f_-(\infty,u)\end{matrix}\right)
	= \left(\begin{matrix} u  \\ 1-u \end{matrix}\right) f_{\gamma_{-+}, \gamma_{+-}} (u),
	\end{equation}
	where $f_{\gamma_{-+}, \gamma_{+-}} (\cdot)$ is the density corresponding to a ${\rm Beta}(\gamma_{-+},\gamma_{+-})$ random variable, i.e.,
	\[
	f_{\gamma_{-+}, \gamma_{+-}} (u) = u^{\gamma_{-+}-1} (1-u)^{\gamma_{+-}-1}
	\mathbf{1}\{u\in[0,1]\}.
	\]
\end{itemize}
\end{theorem}

\begin{remark}
{\rm Observe that, in the regime where $t$ is large, the quantity $y_i(t)$ is well approximated by the type of vertex $i$ at time $t$, i.e., $\int_0^t {\rm d}s \, e^{-s}\mathbf{1}\{x_i(t-s)=+\}$. The type is then close to zero if the vertex has held opinion $-$ for an extended period of time, close to one if the vertex has held opinion $+$ for an extended period of time, and close to $0.5$ if the vertex has been regularly flipping between opinions $+$ and $-$. As a consequence, if $\gamma_{+-}=\gamma_{-+}=\gamma$ and $\gamma$ is large, then we expect that most vertices have been regularly flipping between opinions $+$ and $-$, and therefore have type close to 0.5, whereas if $\gamma$ is small, then we expect that most vertices have held either opinion $+$ or $-$ for an extended period of time, and therefore have type close to 0 or 1. This intuition is confirmed by Theorem \ref{thm:systempdeasymp}\emph{(ii)}, and is illustrated in the top left and middle right panels of Figure \ref{fig:histmod1} below. }
\end{remark}

Theorems \ref{thm:graphonconv}--\ref{thm:systempdeasymp} fully accomplish our goal of characterising the functional law of large numbers in the space of graphons for this first model. In particular, we deduce that this model leads to {\it coexistence} only, meaning that both opinions survive, but there appear \emph{many} edges connecting different opinions in the resulting graph. Although there is only a one-way feedback between vertex and edge dynamics, there appears to be no explicit expression for the densities $f_+$ and $f_-$. Note that Theorem \ref{thm:systempdeasymp} {\it fully} characterises the limiting proportion of vertices having opinion $+$ and $-$, as well the limiting densities $f_+$ and $f_-$. This is possible because there is a one-way feedback only. As we will see later, it is still possible for the second model, even though it gives rise to a richer phenomenology (see Theorem \ref{thm:systempdeasymp2} below), while for the third model only weaker results are obtained (see Theorem \ref{thm:systempdeasymp3} and Conjecture \ref{conjecture} below).

\begin{figure}
\includegraphics[width=5cm]{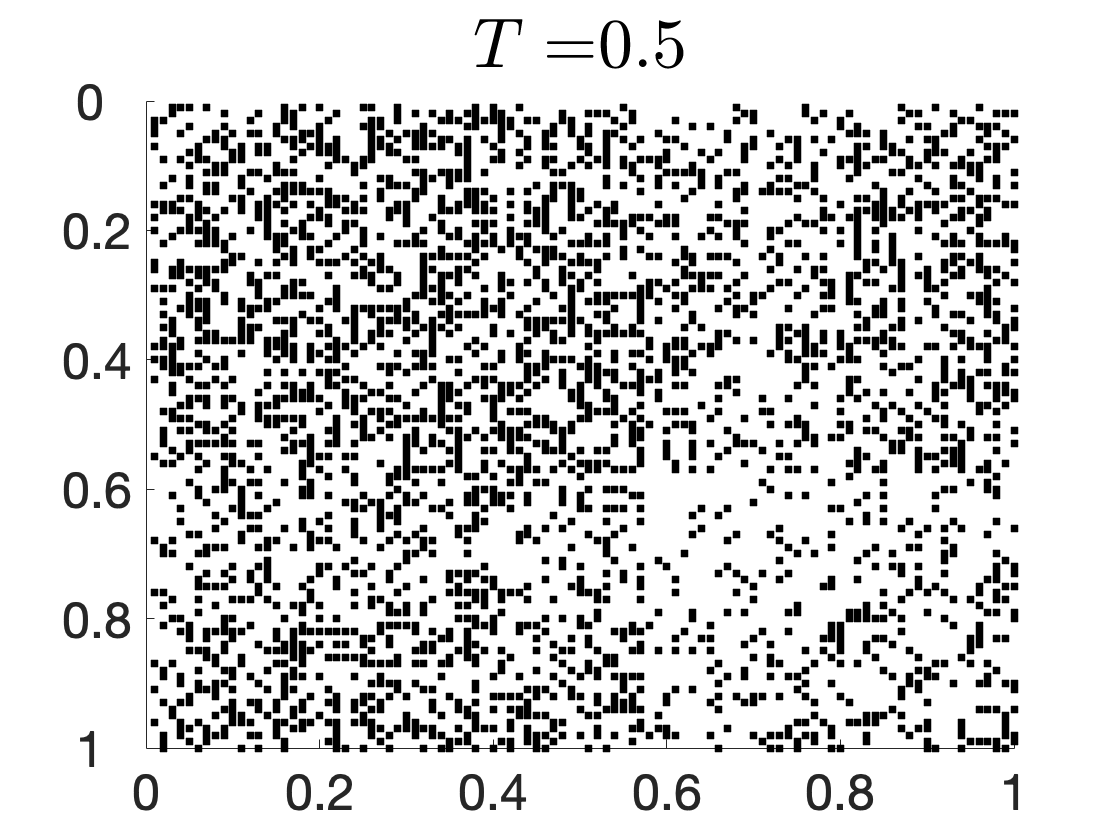}
\includegraphics[width=5cm]{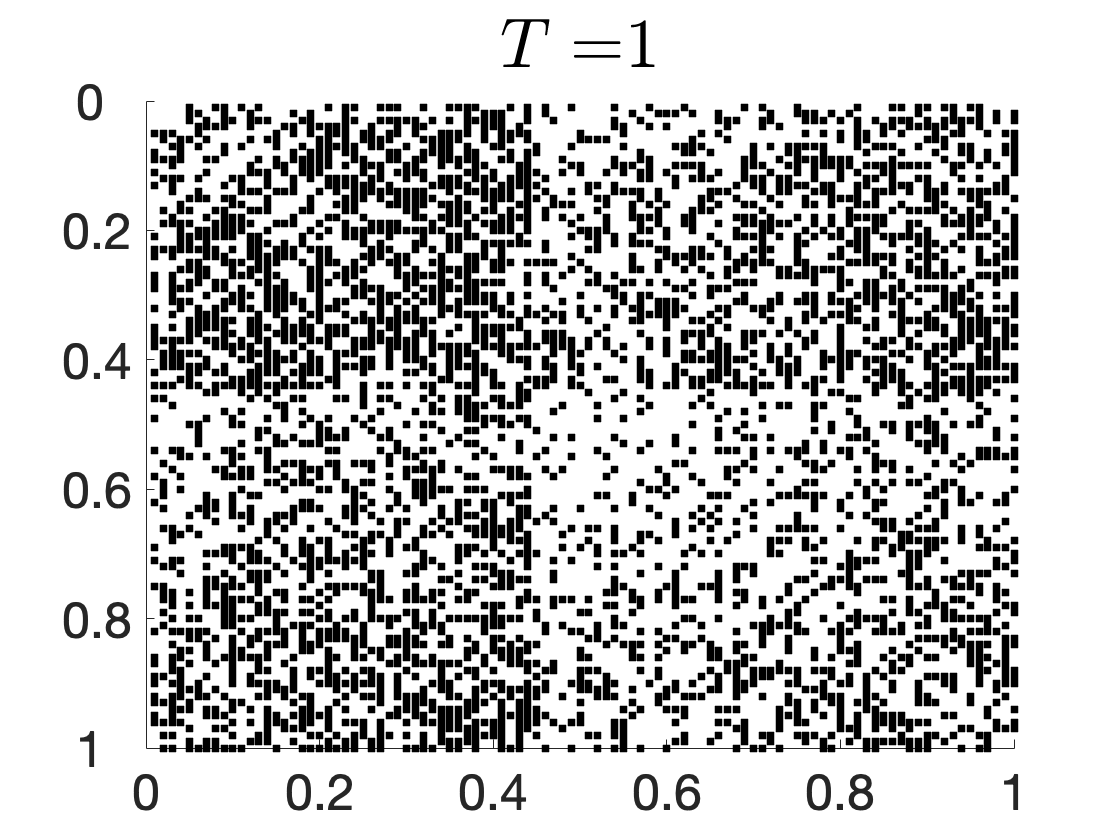}
\includegraphics[width=5cm]{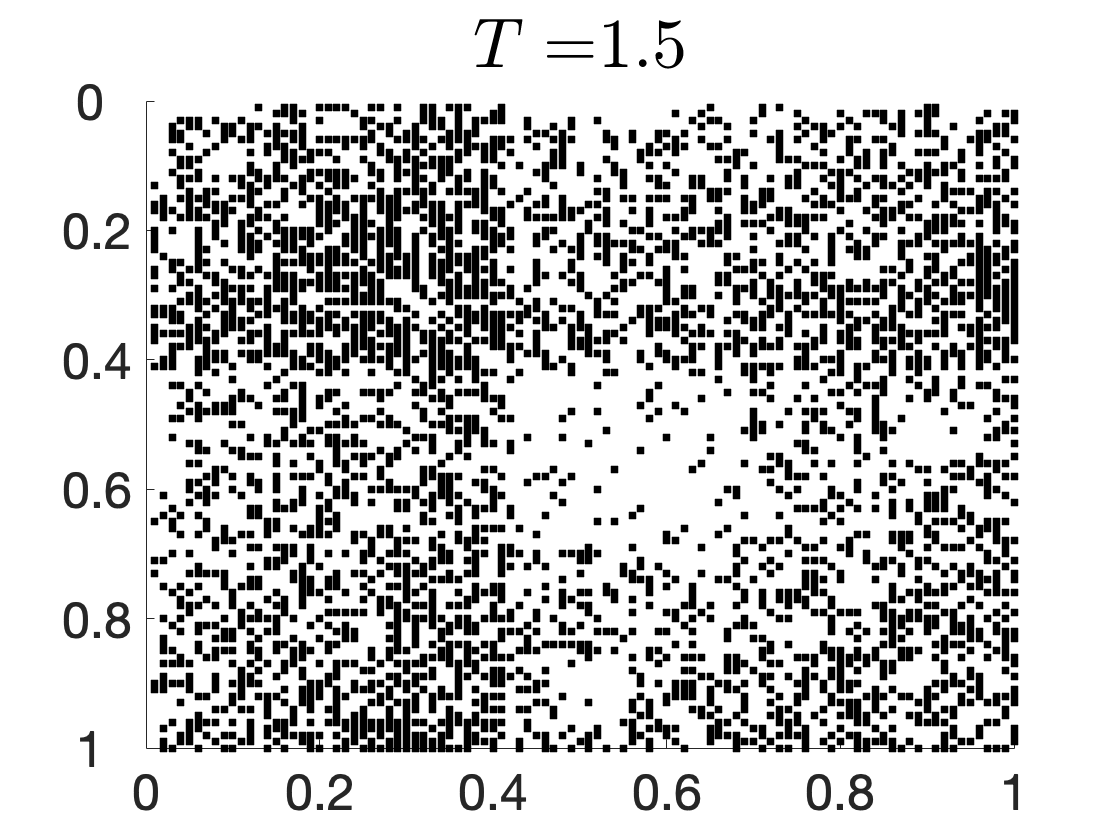}
	
\includegraphics[width=5cm]{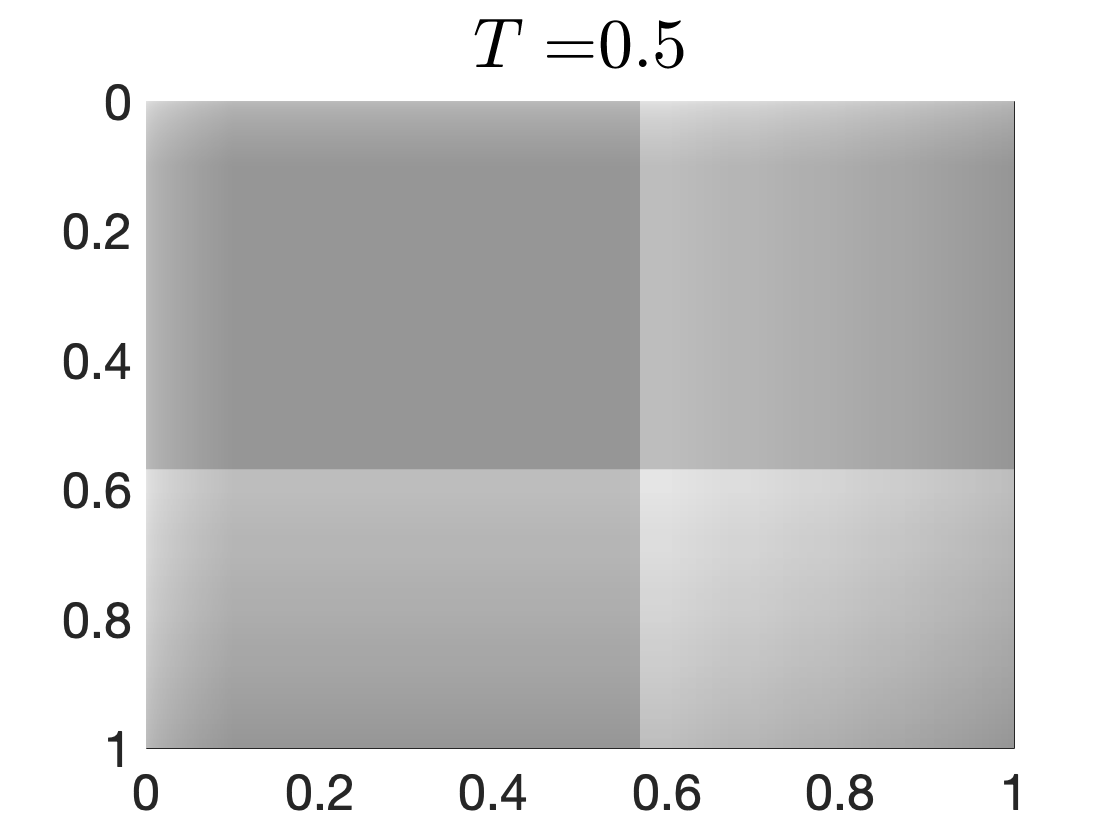}
\includegraphics[width=5cm]{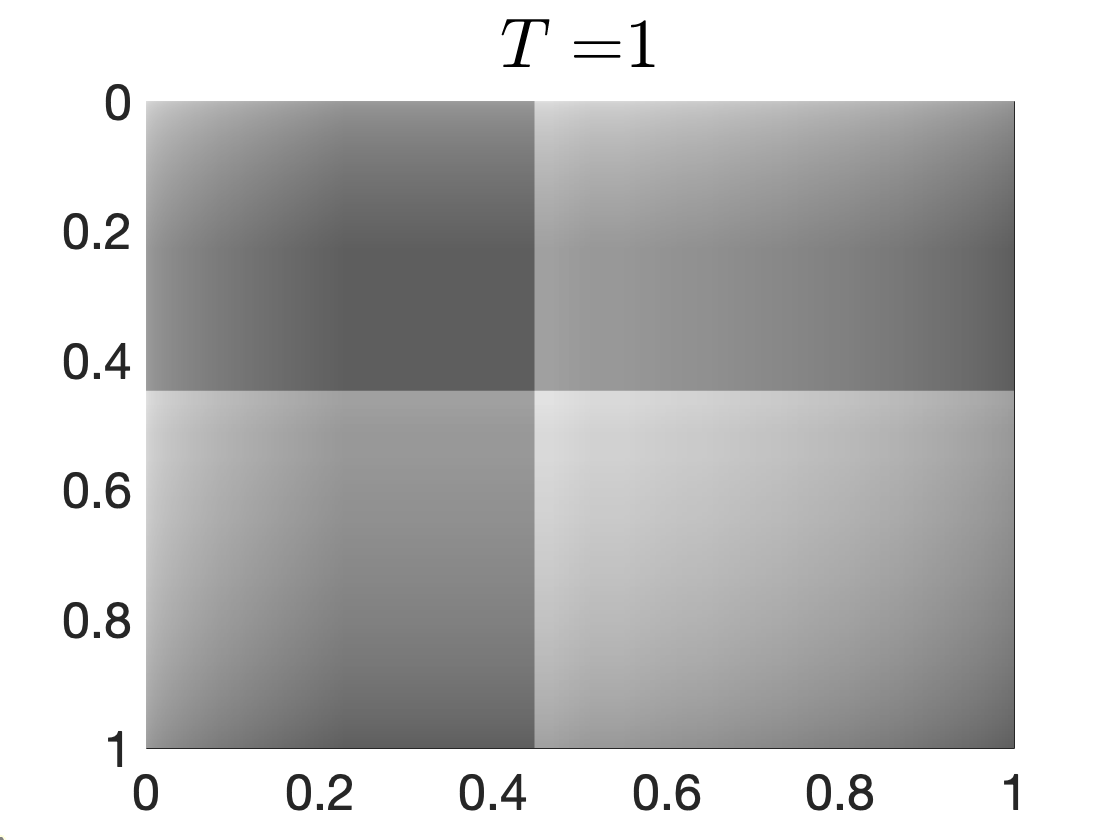}
\includegraphics[width=5cm]{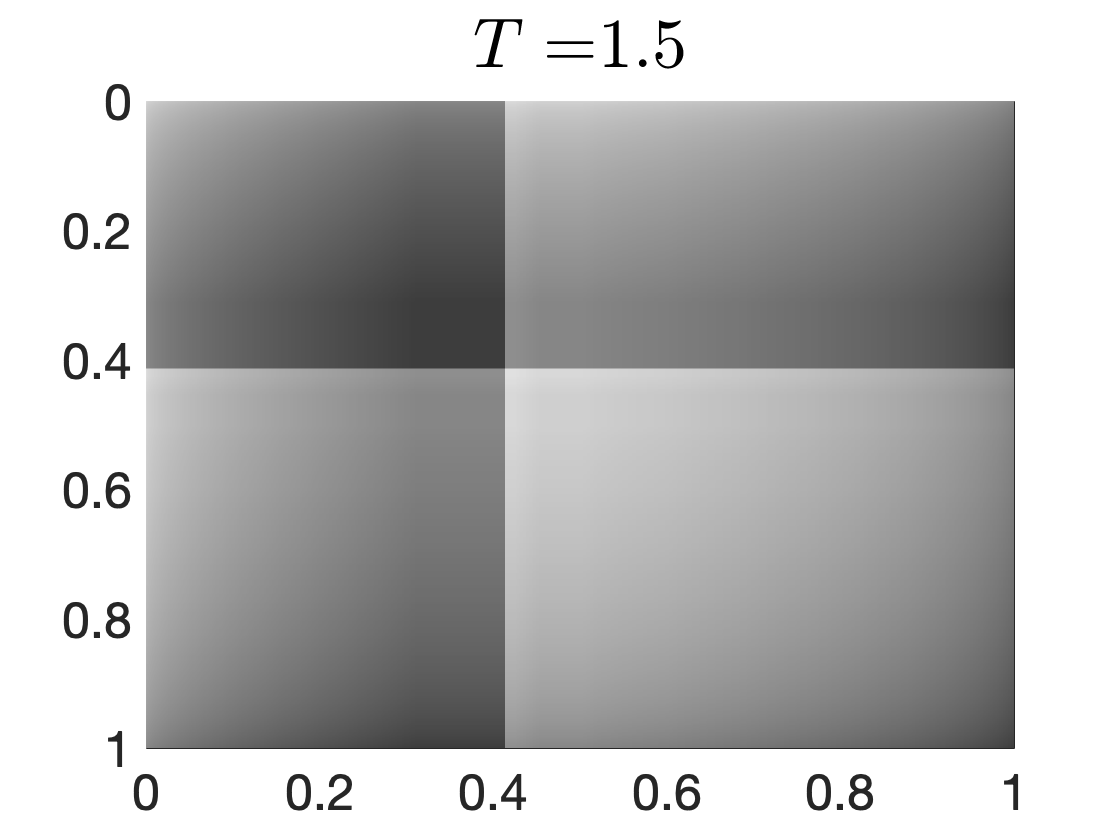}
\caption{\small The top row displays the empirical graphon when $n=100$ and $T=0.5,1,1.5$, the bottom row displays the corresponding functional law of large numbers. Simulations are based on a single run. A dot represents an edge. The labels of the vertices are updated dynamically so that they are ordered lexicographically, i.e., the vertices with opinion $+$ have lower labels than the vertices with opinion $-$, and then by increasing type.}
\label{fig-gFt}
\end{figure}

\subsection{Numerical examples}
\label{mod1:num}

We proceed by presenting a set of numerical examples that illustrate our model's qualitative behavior and show that our asymptotics sets in for modest values of time and graph size. In our experiments we take $p_0=0.05$ and suppose that $x_i(0)=+$ and $y_i(0)$ is uniformly distributed in [0,1] for any $i\in[n]$. We set $\gamma_{-+}=1$, $\gamma_{+-}=1.5$, $\pi_+ =0.9$, $\pi_-=0.1$. In Fig.~\ref{fig-gFt} we compare the empirical graphons obtained for $T=0.5,1,1.5$ with the corresponding functional law of large numbers (see Theorem \ref{thm:graphonconv}). Note that each block representing the limiting graphon has a low (high) density in the top left (bottom right) corner. This is a consequence of the dynamical labelling of the vertices depending on their types (see Assumption~4). Indeed, the top left corner of each block corresponds to those vertices having minimal type, which means that their opinion has been $-$, facilitating their edges to be inactive  (since $\pi_-<\pi_+$). Similarly, the bottom right corner of each block corresponds to those vertices having maximal type, which means that their opinion has been $+$, facilitating their edges to be active. Finally, in Fig.~\ref{fig:histmod1} we compare the empirical distribution of $(y_i(T))_{i\in [n]}$ when $n=6000$ and $T=700$, to its limiting distribution $\rm{Beta}(\gamma_{-+},\gamma_{+-})$, when $n,T \to \infty$, thus corroborating Theorem \ref{thm:systempdeasymp}. 

\begin{figure}[h!]
	\includegraphics[width=5.3cm]{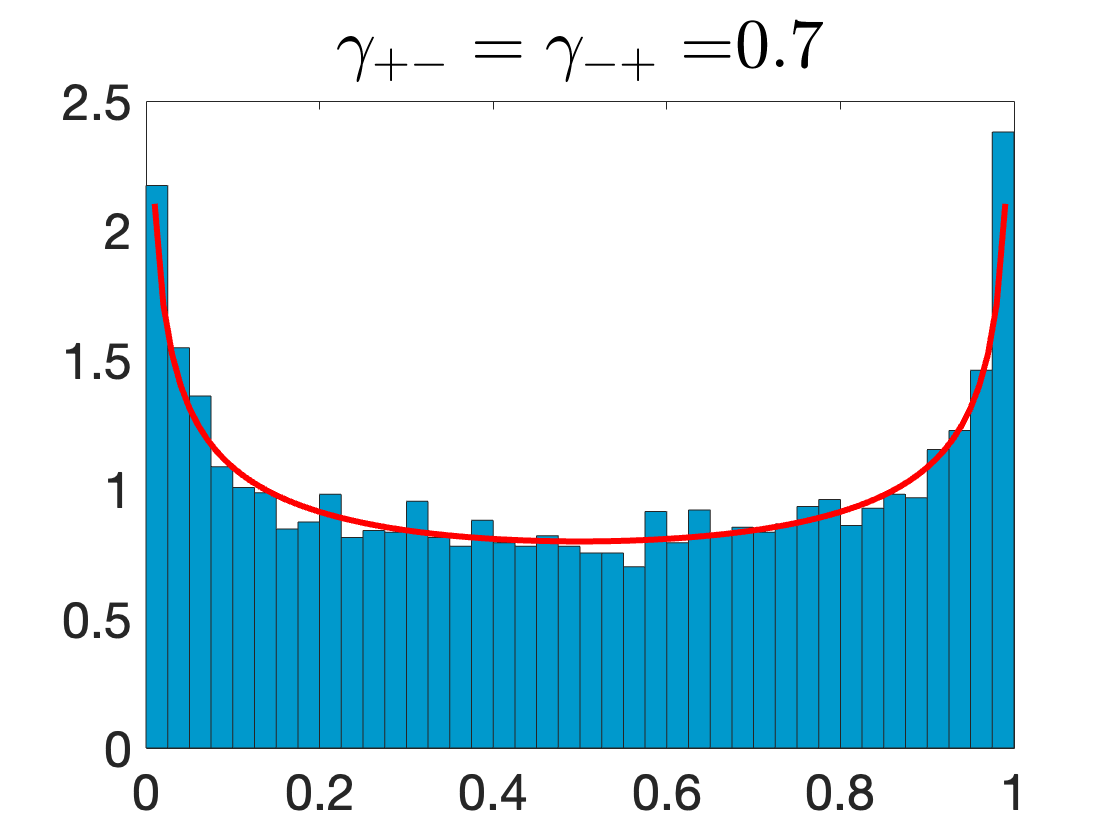}
	\includegraphics[width=5.3cm]{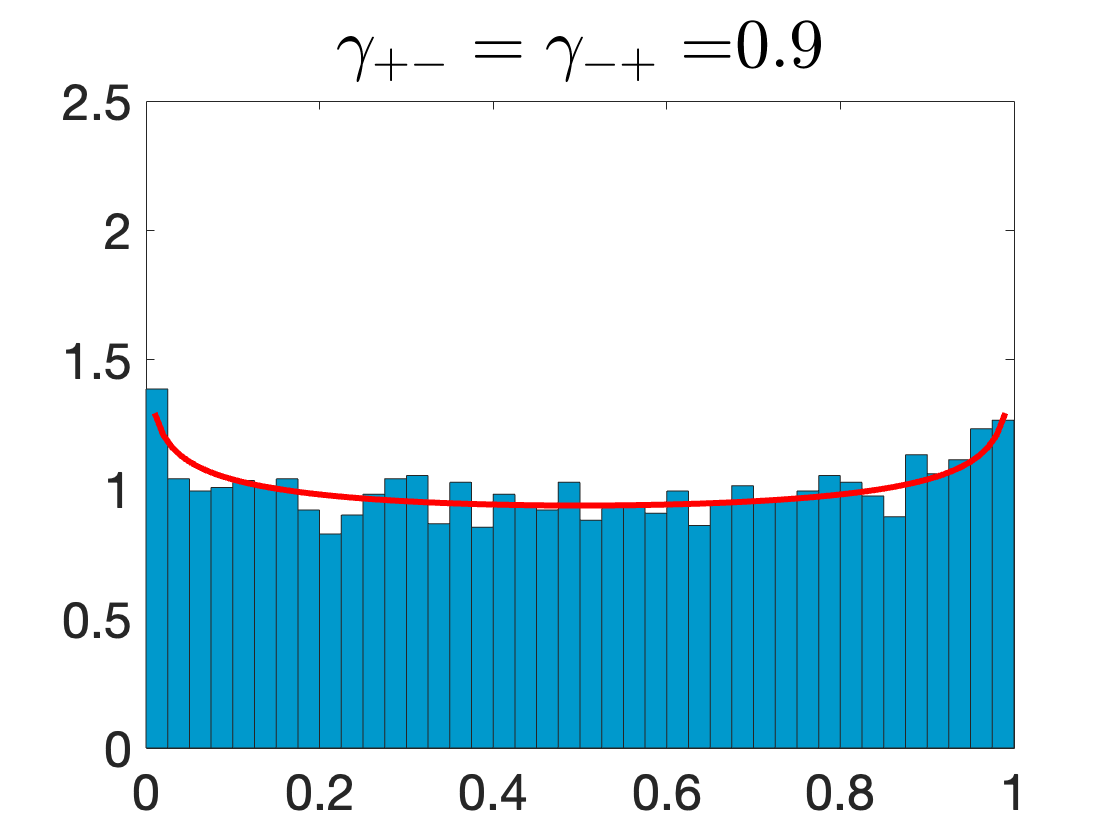}
	\includegraphics[width=5.3cm]{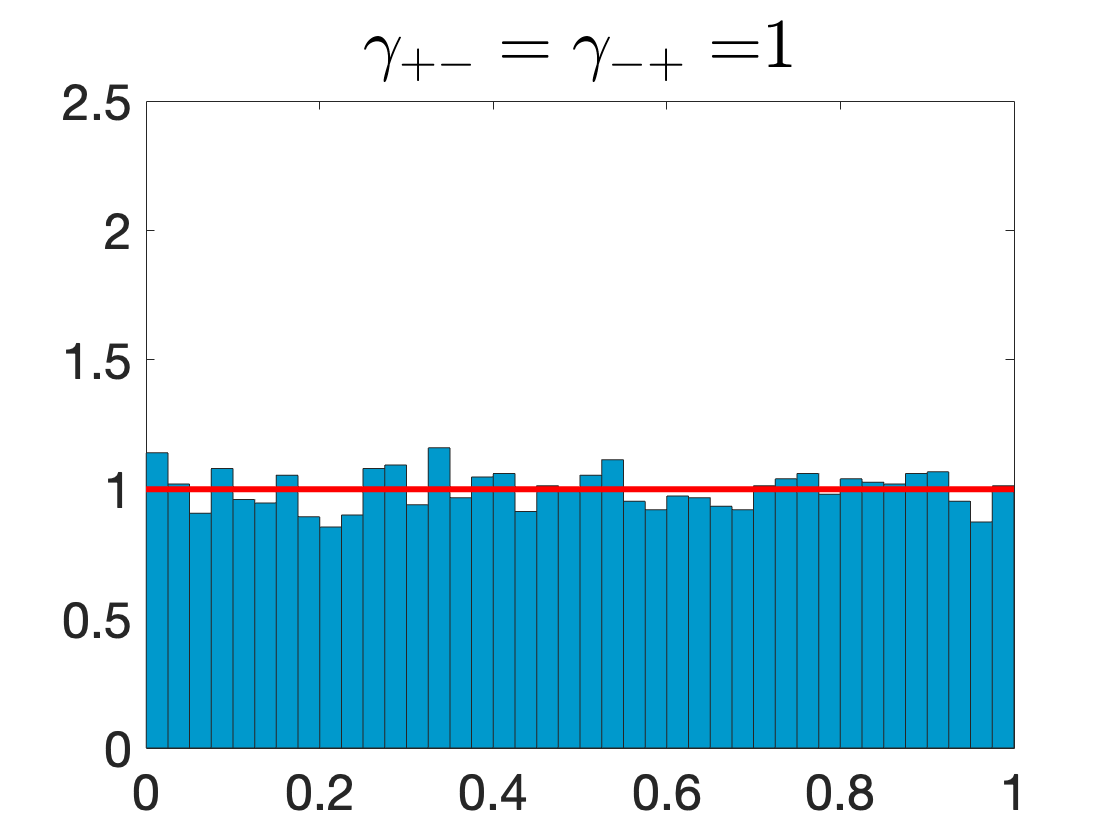}
	\includegraphics[width=5.3cm]{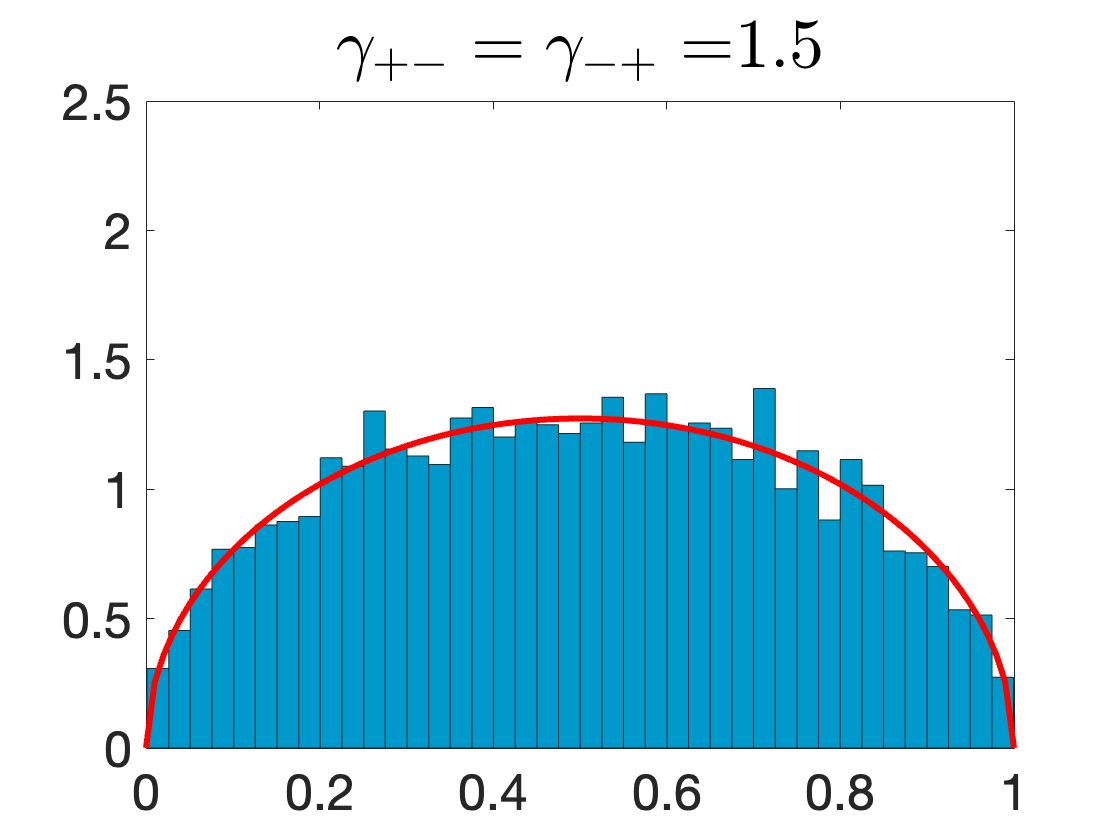}
	\includegraphics[width=5.3cm]{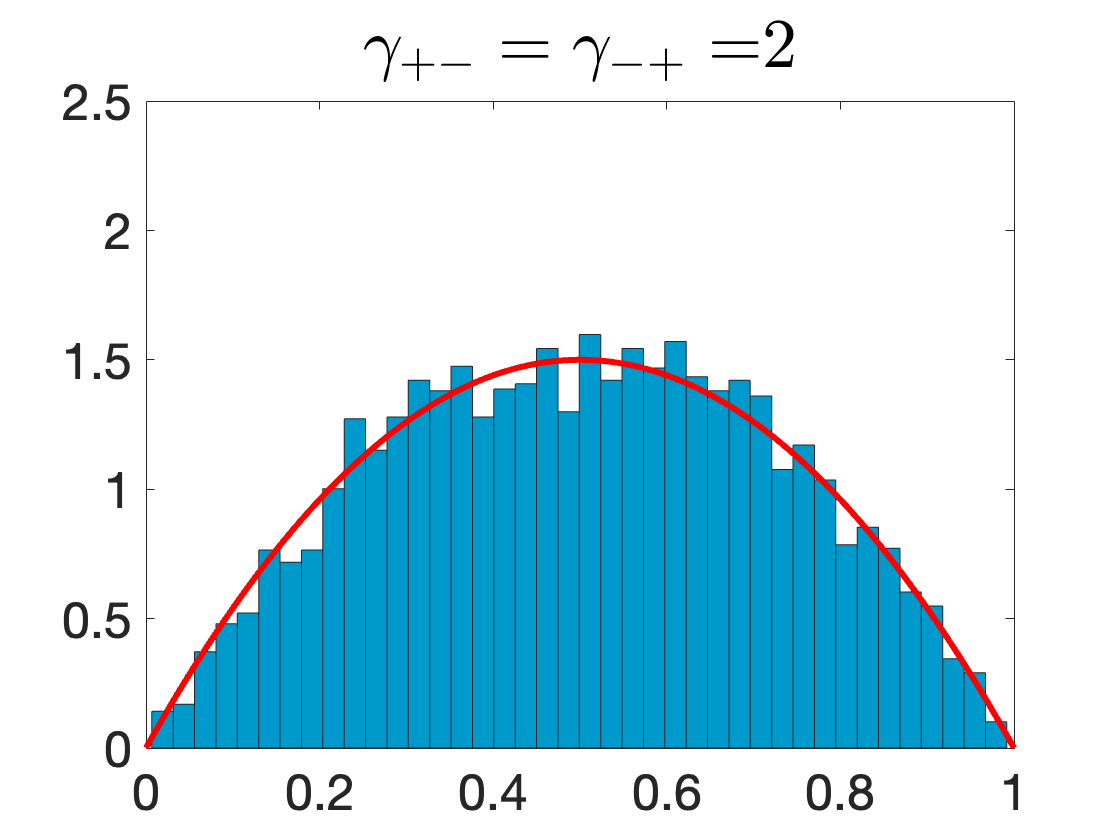}
	\includegraphics[width=5.3cm]{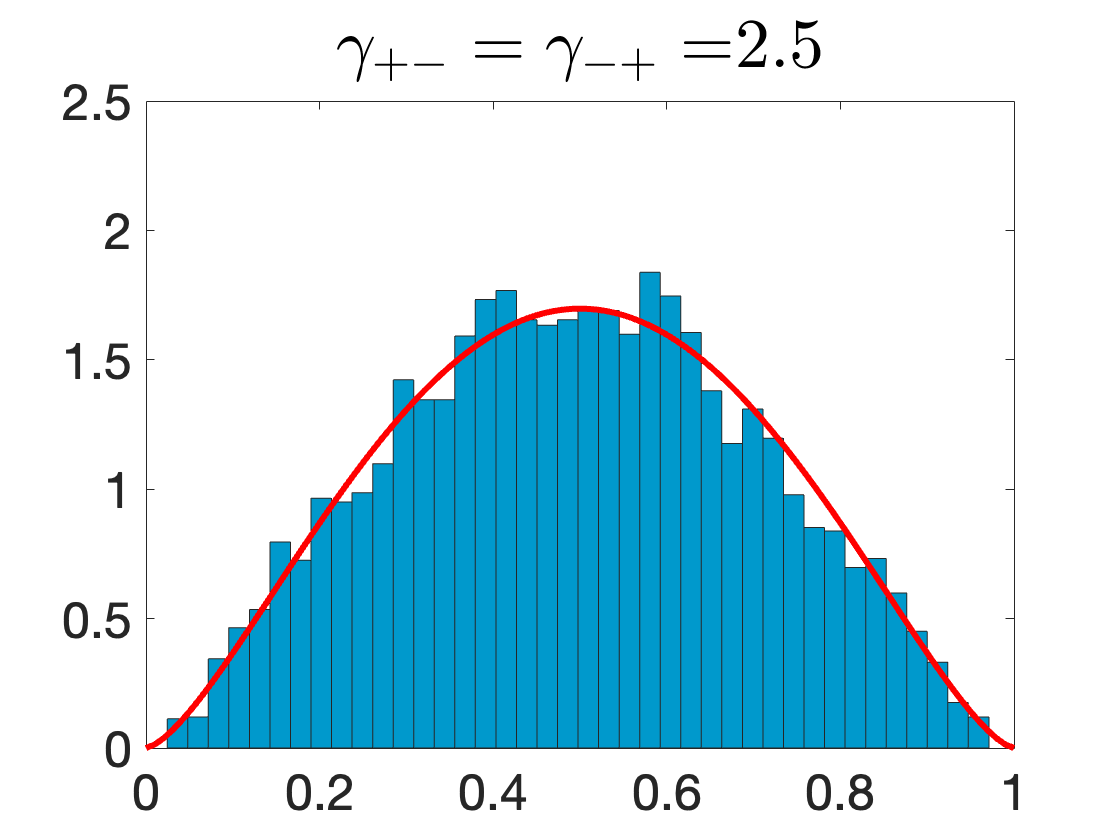}
	\includegraphics[width=5.3cm]{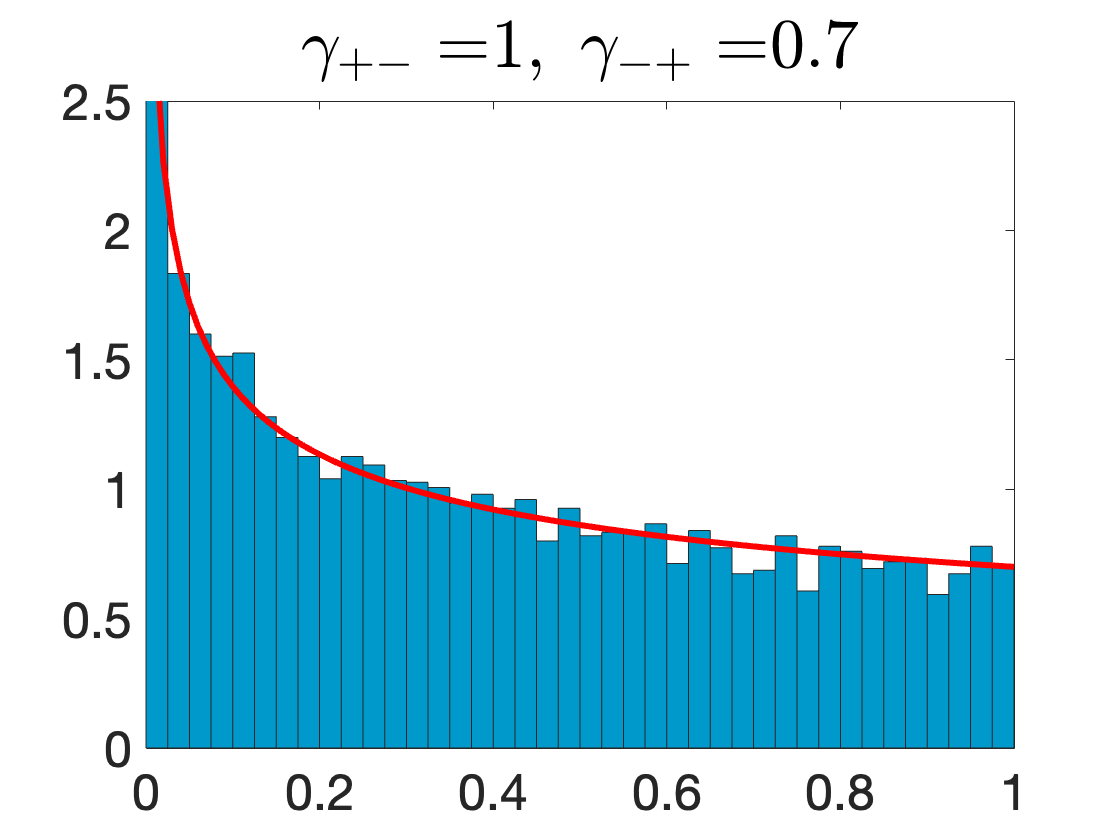}
	\includegraphics[width=5.3cm]{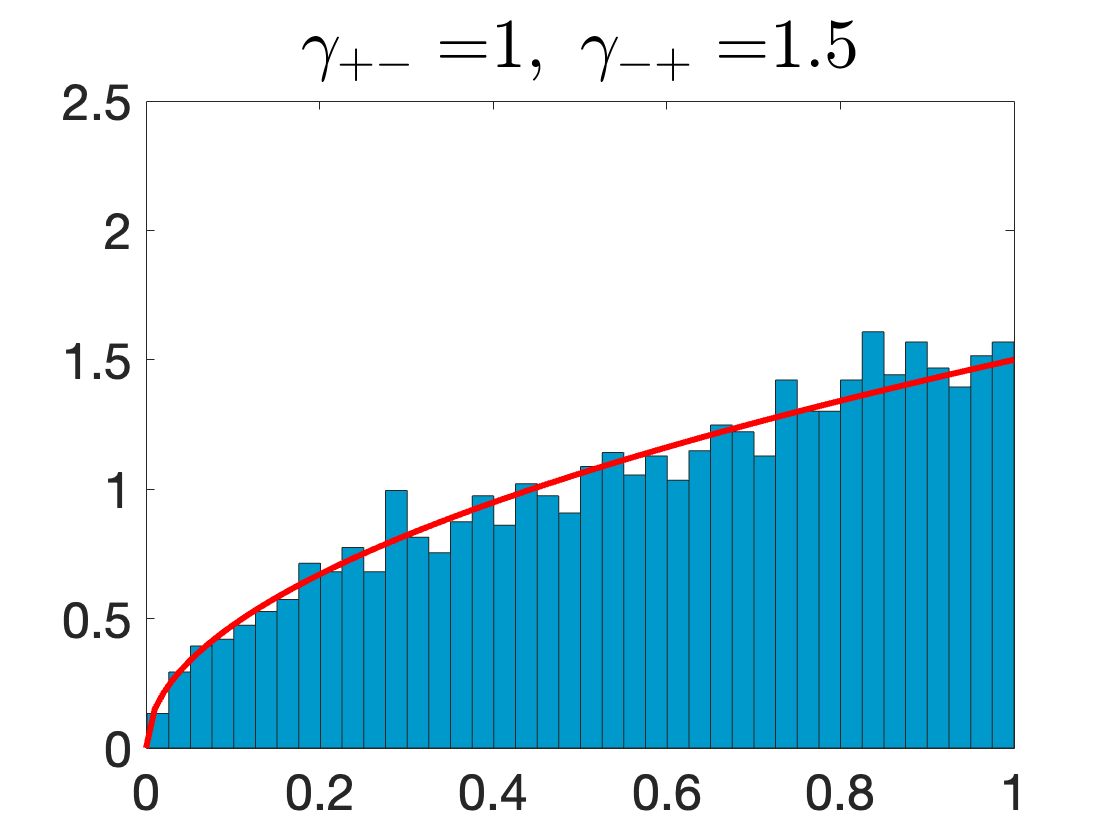}
	\includegraphics[width=5.3cm]{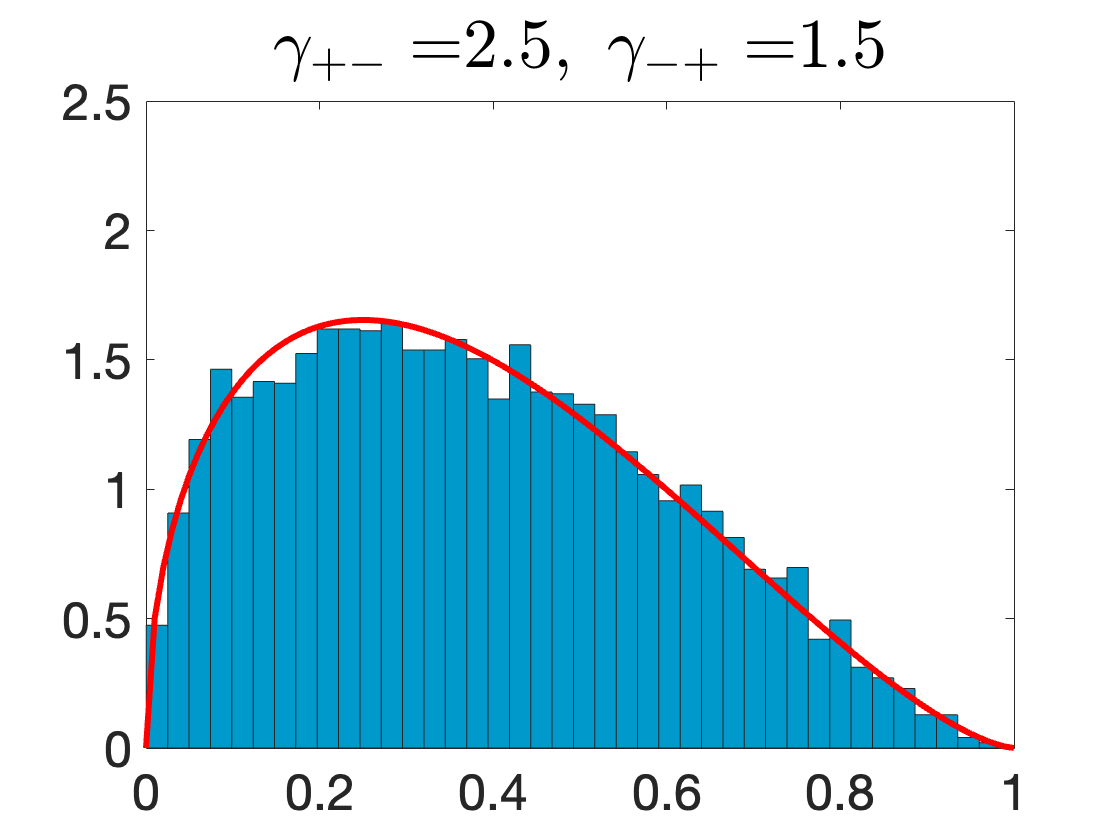}
	\caption{\small Empirical distribution of $(y_i(T))_{i\in [n]}$ when $n=6000$ and $T=700$ for varying $\gamma_{+-}$,$\gamma_{-+}$. For any $i\in[n]$, $y_i(0)$ is chosen uniformly at random in $[0,1]$, independently of the other vertices. The red line indicates the corresponding $\hbox{Beta}(\gamma_{-+},\gamma_{+-})$ density function. Simulations are based on a single run.}
	\label{fig:histmod1}
\end{figure}


\subsection{Proof of the main theorems}
\label{mod1:pr}

\subsubsection{Proof of Theorem \ref{thm:graphonconv}}

The claim follows from \cite[Theorem 3.10]{BdHM22pr} once we have proven that \cite[Assumptions 3.1, 3.6--3.7, 3.9]{BdHM22pr} are in force. Since the first three assumptions hold (see Assumption 4 in Section \ref{sec:3models} and Steps 1--3 in Section \ref{mod1:prep}), it remains to verify the last assumption only, which is needed to prove that $(h^{G_n})_{n\in\mathbb{N}}$ is tight. Let
$$
E_{ij}^{(n)}(t)=\left\{
\begin{array}{ll}
1, &\hbox{if edge } ij \hbox{ is active at time } t, \\
0, &\hbox{otherwise},
\end{array}
\right.
$$
and define
$$
C_n(t,\delta)=\sum_{1 \leq i<j\leq n} \sup_{t \leq u \leq v \leq t + \delta} |E_{ij}^{(n)}(u)-E_{ij}^{(n)}(v)|,
$$
which is the number of edges that change, both from active to inactive and vice versa, at some time between $t$ and $t+\delta$. We need to verify that, for any $\varepsilon>0$,
$$
\lim_{\delta \downarrow 0} \limsup_{n\to\infty} \sup_{t\in{[0,T]}} \dfrac{T}{\delta}\,
\mathbb{P}\left(C_n(t,\delta)>\varepsilon\binom{n}{2}\right) = 0.
$$
To this end, we first note that
$$
\lim_{\delta \downarrow 0} \limsup_{n\to\infty} \sup_{t\in{[0,T]}} \dfrac{T}{\delta}\,
\mathbb{P}\left(C_n(t,\delta)>\varepsilon\binom{n}{2}\right)
\leq \lim_{\delta \downarrow 0} \limsup_{n\to\infty} \max_{i\in{\{0,\ldots,\lfloor \frac{T}{\delta}\rfloor}\}} 
\dfrac{T}{\delta}\,\mathbb{P}\left(C_n(i\delta,2\delta)>\varepsilon\binom{n}{2}\right).
$$
Indeed, the states of the edges change in the time interval $[i\delta,(i+2)\delta]$ when the rate-$1$ Poisson clock rings during that time interval, which occurs with probability $1-{\rm e}^{-2\delta}$. Thus, $C_n(i\delta,2\delta)$ is stochastically dominated by $Z_n(\delta) \sim {\rm Bin}\left(\binom{n}{2},1-{\rm e}^{-2\delta}\right)$. The claim is now an immediate consequence of the Chernoff bound for binomial random variables.


\subsubsection{Proof of Theorem \ref{thm:systempde}}
\label{sec:pdemodel1}

The proof is built on the following lemma, whose proof follows from straightforward computations and therefore is deferred to Appendix \ref{appB}.

\begin{lemma}
\label{lmm:PDE1}
For any $t\in[0,T]$ and $u\in[0,1]$, the densities $f_+(t,u)$ and $f_-(t,u)$ satisfy the following system of PDEs:
\begin{equation}
\label{eq:PDEs}
\begin{aligned}
\frac{\partial}{\partial t} f_+(t,u) + (1-u) \frac{\partial}{\partial u} f_+(t,u) 
&= -(\gamma_{+-}-1) f_+(t,u) + \gamma_{-+} f_-(t,u),\\
\frac{\partial}{\partial t} f_-(t,u) - u \frac{\partial}{\partial u} f_-(t,u) 
&= -(\gamma_{-+}-1) f_-(t,u) + \gamma_{+-} f_+(t,u).
\end{aligned}
\end{equation}
\end{lemma}
Recalling \eqref{eq.vecf} and \eqref{eq:defM}, we can rewrite \eqref{eq:PDEs} in vector form:
\begin{equation}
\label{eq:system1}
\dfrac{\partial}{\partial t} \vec{v}(t,u) = N \vec{v}(t,u) + M(u) \dfrac{\partial}{\partial u} \vec{v}(t,u)
= \left(N + M(u) \dfrac{\partial}{\partial u} \right) \vec{v}(t,u).
\end{equation}
Since the operator $N + M(u) \frac{\partial}{\partial u}$ does not depend on $t$, for any $t\in[0,T]$ and $u\in[0,1]$ the solution of the equation is
$$
\vec{v}(t,u) = \eee^{t\left(N+M(u)\frac{\partial}{\partial u}\right)} \vec{v}(0,u)
= \eee^{t\left(N+M(u)\frac{\partial}{\partial u}\right)} \left(\begin{matrix} g_+(u) \\ g_-(u)\end{matrix}\right).
$$


\subsubsection{Proof of Theorem \ref{thm:systempdeasymp}}
\label{sec:proof2.5}

We start with (i). The first part follows from the law of large numbers as $n\to\infty$ applied to the proportion of vertices holding opinion $+$ and $-$ at time $t$, respectively. In addition, as $n\to\infty$ followed by $t\to\infty$, we know that the probability that a vertex holds opinion $+$ or $-$ corresponds to the first and second entrance, respectively, of the stationary distribution of a continuous--time Markov chain on the state space $\{+,-\}$, its transition rate matrix being given by
\[
P = \left(\begin{matrix} -\gamma_{+-} & \gamma_{+-} \\ \gamma_{-+} & -\gamma_{-+} \end{matrix}\right).
\]
This concludes the proof of (i).

Next consider (ii). The claimed Beta limit follows by checking that the densities $f_+(\infty,u)$ and $f_-(\infty,u)$ given in \eqref{eq:limdensities} satisfy 
\begin{equation}
	\label{eq:mod1asymp}
	M(u)\dfrac{\partial}{\partial u}\vec{v}(\infty,u) + N \vec{v}(\infty,u) = 0.
\end{equation}
This is a straightforward computation after noting that
\[
f'_{\gamma_{-+},\gamma_{+-}}(u) = u(1-u)\left[ \dfrac{\gamma_{-+}-1}{u} - \dfrac{\gamma_{+-}-1}{1-u} \right] f_{\gamma_{-+},\gamma_{+-}}(u).
\]
It remains to show that no other analytic densities satisfy \eqref{eq:mod1asymp}. For any $u\in[0,1]$, such analytic densities take the form
$$
f_+(\infty,u) = \sum_{k=0}^\infty f_{+,k} u^k, \qquad f_-(\infty,u) = \sum_{k=0}^\infty f_{-,k} u^k.
$$
Thus, \eqref{eq:mod1asymp} can be written as
$$
\left\{
\begin{aligned}
\sum_{k=0}^\infty [-(\gamma_{+-}-1) f_{+,k} + \gamma_{-+} f_{-,k}] u^k 
&= \sum_{k=0}^\infty (k+1) f_{+,k+1}u^k - \sum_{k=1}^\infty k f_{+,k} u^k, \\
\sum_{k=0}^\infty [-(\gamma_{-+}-1) f_{-,k} + \gamma_{+-} f_{+,k}] u^k 
&= -\sum_{k=1}^\infty k f_{-,k} u^k,
\end{aligned}
\right.
$$
which provides us with recursive relations between the coefficients. Concretely, equating the coefficients pertaining to both sides of the equations yields
\begin{equation}
\label{eq:systmod1zeroas}
f_{+,1} =-(\gamma_{+-}-1) f_{+,0} + \gamma_{-+} f_{-,0}, \qquad 0 = -(\gamma_{-+}-1) f_{-,0} + \gamma_{+-} f_{+,0},
\end{equation}
and, for $k\geq1$,
\begin{equation}
\label{eq:systmod1as}
\left\{
\begin{aligned}
-(\gamma_{+-}-1) f_{+,k} + \gamma_{-+} f_{-,k} 
&= (k+1) f_{+,k+1}- k f_{+,k}, \\
-(\gamma_{-+}-1) f_{-,k} + \gamma_{+-} f_{+,k} 
&= - k f_{-,k}.
\end{aligned}
\right.
\end{equation}

By the second equation in \eqref{eq:systmod1as}, for any $k\geq1$ we can express $f_{-,k}$ in terms of $f_{+,k}$ only, and after substituting this expression into the first equation in \eqref{eq:systmod1as}, we deduce that
\[
(k+1)f_{+,k+1} = \left(k-\gamma_{+-}+1+\dfrac{\gamma_{-+}\gamma_{+-}}{\gamma_{-+}-1-k}\right) f_{+,k},
\]
which leads to
\begin{equation}\label{eq:recsolution}
f_{+,k} = \left[ \displaystyle\prod_{i=0}^{k-1} \dfrac{i-\gamma_{+-}+1+\dfrac{\gamma_{-+}\gamma_{+-}}{\gamma_{-+}-1-i}}{i+1} \right] f_{+,0}.
\end{equation}
Thus, the coefficients $f_{+,k}$ and $f_{-,k}$ are uniquely determined once $f_{+,0}$ is known. Since we already proved that the Taylor coefficients of the densities in \eqref{eq:limdensities} satisfy these recursive relations, we deduce that no other solution is allowed. Note that the value of $f_{+,0}$ is uniquely determined by the fact that $f_+(\infty,u)+f_-(\infty,u)$ is a density. This concludes the proof.


\subsection{Expression for the vector of densities}
\label{sec:nodepu}

In this section we seek a solution $\vec{v}(t,u)$ of \eqref{eq:system1} of the form
\begin{equation}\label{eq:expl}
\vec{v}(t,u) = \eee^{Nt}\vec{w}(t,u)
\end{equation}
in order to make its expression more explicit. This leads to the following PDE for $\vec{w}(t,u)$:
\begin{equation}
\label{solution1}
\dfrac{\partial}{\partial t}\vec{w}(t) = \left( \eee^{-Nt}M(u)\eee^{Nt} \right) \dfrac{\partial}{\partial u} \vec{w}(t,u).
\end{equation}
In what follows we explicitly compute the exponential matrix $\eee^{Nt}$. We can decompose the matrix $N$ as
$$
N = \mathbb{I}_2 +\bar{N},  \qquad \bar{N} = \gamma_{+-} A +\gamma_{-+} B, 
\qquad A = \left(\begin{matrix} -1 & 0 \\ 1 & 0 \end{matrix}\right), 
\qquad B = \left(\begin{matrix} 0 & 1 \\ 0 & -1\end{matrix}\right),
$$
so that we are able to find the explicit PDE \eqref{solution1}, as well as to provide a more explicit solution of the form \eqref{eq:expl}. Let $T\colon\,\mathbb{R}^{2\times2} \to \mathbb{R}^{2\times2}$ be the operator that switches the two rows of a matrix. It is trivial to check by induction that, for $k \in \mathbb{N}$, 
\begin{equation}
\label{Apower}
A^k = \left\{
\begin{array}{ll}
TA, &\hbox{if } k \hbox{ is even}, \\
A, &\hbox{if } k \hbox{ is odd},
\end{array}
\right.
\end{equation}
and
\begin{equation}
\label{Bpower}
B^k = \left\{
\begin{array}{ll}
TB, &\hbox{if } k \hbox{ is even}, \\
B, &\hbox{if } k \hbox{ is odd}.
\end{array}
\right.
\end{equation}
The computation of the exponential matrix is built on the following lemmas, whose proofs follow from straightforward computations and are therefore deferred to Appendix \ref{appB}.

\begin{lemma}
\label{Npower}
For any $k\in{\mathbb N}$,
$$
\bar{N}^k=(\gamma_{+-} + \gamma_{-+})^{k-1}(\gamma_{+-} A^k + \gamma_{-+} B^k).
$$
\end{lemma}

\begin{lemma}
\label{Nexp}
For $t \geq 0$,
$$
{\rm e}^{\bar{N}t} = \dfrac{1}{\gamma_{+-} + \gamma_{-+}}\left(\gamma_{+-}{\rm e}^{t(\gamma_{+-} + \gamma_{-+})A}
+ \gamma_{-+} {\rm e}^{t(\gamma_{+-} + \gamma_{-+})B}\right).
$$
\end{lemma}

\begin{lemma}
\label{expansion}
For $t \geq 0$,
$$
{\rm e}^{t(\gamma_{+-}+\gamma_{-+})A}
= \left(\begin{matrix} {\rm e}^{-t(\gamma_{+-}+\gamma_{-+})} 
& 0 \\ 1- {\rm e}^{-t(\gamma_{+-} + \gamma_{-+})} 
& 1\end{matrix}\right),
\qquad
{\rm e}^{t(\gamma_{+-} + \gamma_{-+})B} = \left(\begin{matrix} 1 & 1-{\rm e}^{-t(\gamma_{+-} + \gamma_{-+})} \\ 0 
& {\rm e}^{-t(\gamma_{+-} + \gamma_{-+})} \end{matrix}\right).
$$
\end{lemma}

From Lemmas \ref{Nexp}--\ref{expansion} we deduce that 
$$
\begin{array}{ll}
\eee^{Nt} = \eee^{\mathbb{I}_2t}\eee^{\bar{N}t} 
&= \dfrac{1}{\gamma_{+-} + \gamma_{-+}} \eee^t\mathbb{I}_2\left(\gamma_{+-}{\rm e}^{t(\gamma_{+-} + \gamma_{-+})A}
+ \gamma_{-+}{\rm e}^{t(\gamma_{+-} + \gamma_{-+})B}\right) \\ [0.5cm]
&= \eee^{t}\left(\begin{matrix} \dfrac{\gamma_{+-}{\rm e}^{-t(\gamma_{+-}+\gamma_{-+})}+\gamma_{-+}}{\gamma_{+-}+\gamma_{-+}} 
&\dfrac{\gamma_{-+}(1- {\rm e}^{-t(\gamma_{+-} + \gamma_{-+})})}{\gamma_{+-}+\gamma_{-+}} \\ 
\dfrac{\gamma_{+-}(1- {\rm e}^{-t(\gamma_{+-} + \gamma_{-+})})}{\gamma_{+-}+\gamma_{-+}} 
&\dfrac{\gamma_{-+}{\rm e}^{-t(\gamma_{+-} + \gamma_{-+})} + \gamma_{+-}}{\gamma_{+-} + \gamma_{-+}}\end{matrix}\right).
\end{array}
$$


\section{Second model: two-way feedback and consensus}
\label{sec:twoway1}

In the second model, treated in this section, the dynamics of the edges is the same as in the first model, but the dynamics of the vertices is different. The model can be described as follows, where for completeness we repeat the description of the dynamics of the edges:
\begin{itemize}
\item 
{\it Vertex dynamics.} Each vertex holds opinion $+$ or $-$, and is assigned an independent rate-$\beta$ Poisson clock. Each time the clock rings the vertex selects one of its neighbours uniformly at random and copies the opinion of that vertex.
\item 
{\it Edge dynamics.} Each edge is re-sampled at rate $1$, i.e., a rate-$1$ Poisson clock is attached to each edge and when the clock rings the edge is active with a probability that depends on the current opinion of the two connected vertices: with probability $\pi_+$ if the two vertices hold opinion $+$, with probability $\pi_-$ if the two vertices hold opinion $-$, and otherwise with probability $\tfrac12(\pi_++\pi_-)$.
\end{itemize}
Note that this network is co-evolutionary, in that the opinions of the adjacent vertices affect the probability that an edge is active (through $\pi_+ $ and $\pi_-$), and the state of adjacent edges affect the opinions of the vertices (because the vertices copy the opinions of their neighbours).

The structure of this section follows the one of Section \ref{sec:oneway}: Section~\ref{mod2:prep} steps through the framework in Section \ref{sec:genstrategy}, Section~\ref{mod2:mr} states the main results, Section~\ref{mod2:num} offers simulations, Section~\ref{mod2:pr} provides proofs and Section~\ref{sec:nonlin} defines a generalisation of the model defined above that shows polarisation.


\subsection{Preparations}
\label{mod2:prep}

To help set up the notation and describe Steps 1--4 introduced in Section \ref{sec:genstrategy}, consider a related model with the same edge dynamics as the co-evolutionary model, but with a \emph{different vertex dynamics} that we refer to as the \emph{mimicking process} $(G_n^*(t))_{t\in[0,T]}$. For this process, we define the generalised type $X_i^*(t)$ of vertex $i$ at time t as in Section \ref{sec:3models} for the original model, but now referring to as $X_i^*(t) = (x_i^*(t), y_i^*(t))$. In the mimicking process we suppose that each vertex is assigned an independent rate-$\beta$ Poisson clock, and when the clock rings the vertex chooses opinion + with probability  
\begin{equation}
\label{eq:alpha}
\alpha(t;u,\vec{v}(t,\cdot)) = \frac{\int_0^1 {\rm d}y\,f_+(t,y) H(t;y,u)}{\int_0^1 {\rm d}y\, [f_+(t,y) + f_-(t,y)]\,H(t;y,u)},
\end{equation}
and opinion $-$ with probability $1-\alpha(t;u,\vec{v}(t,\cdot))$, where $u$ is the type of the vertex defined in \eqref{eqn:typedef}, $\vec{v}(t,u) = (f_+(t,u),f_-(t,u))^\top$ (as in \eqref{eq.vecf}), $H$ is defined in \eqref{eqn:Hdef}, and $f_+(t,\cdot)$ and $f_-(t,\cdot)$ are the (unique) solutions of the system \eqref{eq:PDEsalt} with values $f_+(0,\cdot)$ and $f_-(0,\cdot)$ at time zero. In what follows we simply refer to $\alpha(t;u,\vec{v}(t,\cdot))$ as $\alpha(t;u,\vec{v})$ to lighten the notation. 

Observe that the mimicking process has \emph{one-way-dependence} and vertex types are independent and identically distributed. Consequently, we can apply Steps 1--3 for the mimicking process using similar arguments as in Section \ref{sec:oneway} (and later we have to establish that our actual process and its mimicking counterpart are sufficiently close). 

\medskip \noindent
\emph{\underline{Step 1:}} We characterise $F$ by writing down the generator of the process $(x^*_i(t),y^*_i(t))_{t \in [0,T]}$ (i.e., the counterpart of \eqref{eq:gen1}). It takes the form
\begin{equation}
\label{eq:gen2}
	({\cal L}_t f)(x,y) = \beta((1-\alpha(t;y,\vec{v}))\mathbf{1}\{x=+\}+\alpha(t;y,\vec{v})\mathbf{1}\{x=-\})[f(x',y)-f(x,y)]
+ b(x,y) \dfrac{\partial}{\partial y} f(x,y),
\end{equation}
where $x'$ is the opinion of the selected vertex after switching its opinion $x$ at time $t$ and $b(x,\cdot)$ is the drift term when starting from $x\in\{-,+\}$ and has the same shape as for model 1 (see \eqref{eq:b+}-\eqref{eq:b-} below).
This then gives the corresponding Kolmogorov forward equations (i.e., the counterpart of \eqref{eq:PDEs}) 
\begin{equation}
\label{eq:PDEsalt}
\begin{aligned}
\frac{\partial}{\partial t} f_+(t,u) + (1-u) \frac{\partial}{\partial u} f_+(t,u)  
&= \beta f_-(t,u)\,\alpha(t;u,\vec{v}) -(\beta(1-\alpha(t;u,\vec{v}))-1) f_+(t,u),\\
\frac{\partial}{\partial t} f_-(t,u) - u \frac{\partial}{\partial u} f_-(t,u) 
&= \beta f_+(t,u)\,(1-\alpha(t;u,\vec{v})) - (\beta \alpha(t;u,\vec{v})-1)f_-(t,u).
\end{aligned}
\end{equation}
As before, these differential equations characterise $F$ and can be used to verify Assumption 1. 

\medskip \noindent
\emph{\underline{Steps 2 and 3:}} Because the edge dynamics of the model is equivalent to that of the model in Section \ref{sec:oneway}, these steps are the same with $H$ defined in \eqref{eqn:Hdef}.

\medskip \noindent
\emph{\underline{Step 4:}} We need to show that the dynamics of the mimicking process and the co-evolutionary model converge as $n \to \infty$. While the arguments required are technically involved, in essence the reasoning boils down to noting that if $h^{G^*_n} \Rightarrow g^{[F]}$, then the numerator of $\alpha(t;u,\vec{v})$ can be interpreted as the number of vertices having opinion $+$ connected to a vertex with type $v\in [u-{\rm d} u,u+{\rm d}u]$ divided by $n$, and the denominator can be interpreted as the total number of vertices connected to a vertex with type $v\in [u-{\rm d} u,u+{\rm d}u]$ divided by $n$. The probability that the vertex takes opinion $+$ when its clock rings is then asymptotically equivalent to the probability that it takes opinion $+$ by copying one of its neighbours (i.e., the same vertex dynamics as in the co-evolutionary model). This is possible because we consider dense graph-valued processes, meaning that each vertex is adjacent to order $n$ active edges, so that, as $n \to \infty$, a law of large numbers holds at each vertex.


\subsection{Main results: Theorems~\ref{thm:graphonconvalt}--\ref{thm:systempdeasymp2}}
\label{mod2:mr}

The analogues of Theorems~\ref{thm:graphonconv}--\ref{thm:systempdeasymp} read as follows.

\begin{theorem}
\label{thm:graphonconvalt}
$h^{G_n} \Rightarrow g^{[F]}$ as $n\to\infty$ in the space $D((\mathcal{W},d_{\square}),[0,T])$, where $h^{G_n}$ is the empirical graphon associated with $G_n$ defined in \eqref{eq:graphon}.
\end{theorem}

Recall the definition of the vector of densities $\vec{v}(t,u)$ given in \eqref{eq.vecf}. The next theorem claims the (local) existence of the densities $f_+(t,u)$ and $f_-(t,u)$, which satisfy the Kolmogorov forward equations for a simpler process, referred to {\it mimicking process}, defined in Section \ref{sec:coupling} below. 

\begin{theorem}
\label{thm:systempdealt}
For every $t\in[0,T]$ and $u\in(0,1)$, there exists a unique local analytic vector of densities $\vec{v}(t,u)$, i.e., for any $u\in(0,1)$ there exists an open set $O_u$ containing $u$ such that there is a unique analytic vector $\vec{v}(t,u)$ on $O_u$ satisfying $\vec{v}(0,u)=(f_+(0,u),f_-(0,u))^\top$ for any $u\in O_u$.
\end{theorem}
Let $p_+$ be the limiting proportion of vertices holding opinion $+$, i.e., 
\begin{equation}\label{eq:limprop}
	p_+ = \lim_{t\to\infty}\lim_{n\to\infty} \dfrac{N_+(t)}{n}.
\end{equation}
Note that this quantity is well defined and the order of the two limits {\it does} matter. Indeed, the process $Z(t)=((x_i(t))_{i=1,...,n},G_n(t))_{t\geq0}$ is an irreducible and aperiodic Markov chain with finite state space. Thus, by using the independence of vertices in the mimicking process (see Section \ref{sec:coupling}), together with Lemma \ref{lem:coupling} below, we know that the real process is close to this simpler one, so that also the real process converges to its stationary distribution in the limit as $n\to\infty$ followed by $t\to\infty$. This implies that the proportion $p_+$ of vertices having opinion $+$ is well defined, and it is determined by the parameters of the model. 
\begin{theorem}
\label{thm:systempdeasymp2}
For any $u\in[0,1]$, the following statements hold.
\begin{itemize}
	\item[(i)]	If $\pi_+=\pi_-$, then
	\begin{equation}\label{eq:limdensities2}
	\lim_{t\to\infty}\vec{v}(t,u) := \left(\begin{matrix} f_+(\infty,u) \\ f_-(\infty,u)\end{matrix}\right)
	= \left(\begin{matrix} u  \\ 1-u \end{matrix}\right) f_{\beta p_+, \beta (1-p_+)} (u),
	\end{equation}
	where $f_{\beta p_+, \beta (1-p_+)}$ is the density corresponding to a ${\rm Beta}(\beta p_+,\beta (1-p_+))$ random variable, i.e.,
	\[
	f_{\beta p_+, \beta (1-p_+)} (u) = u^{\beta p_+-1} (1-u)^{\beta (1-p_+)-1}\mathbf{1}\{u\in[0,1]\}.
	\]
	\item[(ii)] If $\pi_+\neq\pi_-$, then only consensus is admissible, i.e., either
	\begin{equation}\label{eq:limdensities3}
	\lim_{t\to\infty}\vec{v}(t,u) := \left(\begin{matrix} f_+(\infty,u) \\ f_-(\infty,u)\end{matrix}\right)
	= \left(\begin{matrix} \delta_1(u)  \\ 0 \end{matrix}\right),
	\end{equation}
	or
	\begin{equation}\label{eq:limdensities4}
	\lim_{t\to\infty}\vec{v}(t,u) := \left(\begin{matrix} f_+(\infty,u) \\ f_-(\infty,u)\end{matrix}\right)
	= \left(\begin{matrix} 0  \\ \delta_0(u) \end{matrix}\right),
	\end{equation}
	where $\delta_{x_0}(\cdot)$ is the point--mass distribution at $x_0$. 
	\end{itemize}
\end{theorem}

\begin{remark}
	{\rm Note that Theorem \ref{thm:systempdeasymp2}\emph{(i)} does not exclude the possibility that the system reaches consensus. Indeed, consensus corresponds to the limiting cases $p_+ \to 1$ (consensus on opinion $+$) and $p_+ \to 0$ (consensus on opinion $-$). Indeed, if $X\sim{\rm Beta}(\beta p_+,\beta (1-p_+))$, then for any $\beta>0$ we have $X \stackrel{d}{\to} 1$ when $p_+ \to 1$ and $X \stackrel{d}{\to} 0$ when $p_+ \to 0$.} \hfill$\spadesuit$
    \end{remark}

 \begin{remark}\label{rem:diff}
 	{\rm Due to the presence of two-way feedback, model 2 displays far richer behaviour than model 1. Indeed, Theorem \ref{thm:systempdeasymp2} suggests that there are two distinct scenarios. First, Theorem \ref{thm:systempdeasymp2}\emph{(ii)} indicates that if $\pi_-\neq \pi_+$, then the model moves toward consensus quickly with the proportion of vertices holding the minority opinion falling below any $\varepsilon >0$ in a finite time that does not increase with $n$ (see Figure \ref{fig-later} for an example). Since consensus is absorbing, this (effectively) ends the dynamics of opinions. Second, if $\pi_+=\pi_-$, then Theorem \ref{thm:systempdeasymp2}\emph{(i)} suggests that the model moves towards one of the fixed points in described by \eqref{eq:limdensities2} quickly, i.e., in a finite time that does not increase with $n$. However, unless this fixed point is consensus, the dynamics of the opinions does not end. Moreover, \eqref{eq:limdensities2} describes a continuum (i.e., a manifold) of fixed point solutions. It may therefore be insightful to observe the process on a different timescale, i.e., $(h^{G_n(nt)})_{t \in [0,T]}$ rather than $(h^{G_n(t)})_{t \in [0,T]}$. We conjecture that if $\pi_+=\pi_-$, then as $n\to\infty$ the process $(h^{G_n(nt)})_{t \in [0,T]}$ converges to a diffusion on the manifold defined by the fixed point solutions in \eqref{eq:limdensities2}.} \hfill$\spadesuit$
\end{remark}


\subsection{Numerical examples}
\label{mod2:num}

Suppose that $n=100$, $p_0=0.05$, $\beta=0.66$, $\pi_+=0.9$, $\pi_-=0.1$, and vertices initially hold opinion $+$ with probability 0.5. Figure \ref{fig-plot} displays an outcome of the empirical graphon (top row) and its corresponding functional law of large numbers from Theorem \ref{thm:graphonconvalt} (bottom row) when $T=1,2,3$. Recalling Assumption 4 and that $\pi_+>\pi_-$, we note that the darker top left corner of each graphon represents edges between pairs of opinion $+$ vertices, whereas the lighter bottom right corner represents edges between pairs of opinion $-$ vertices. Observe that as $T$ increases the proportion of vertices holding opinion $+$ increases. Figure \ref{fig-later} displays the functional law of large numbers when $T=6,12,18$, and illustrates convergence of the model to consensus on opinion $+$. This agrees with Theorem \ref{thm:systempdeasymp2}\emph{(ii)} which states that when $\pi_-\neq \pi_+$ the only fixed point solutions of the model correspond to consensus. Finally, for different parameters, Fig.~\ref{fig:histmod2} is in line with the claim that the density function of the types be a Beta distribution, as was stated in Theorem \ref{thm:systempdeasymp2}\emph{(i)}.

\begin{figure}
\includegraphics[width=5cm]{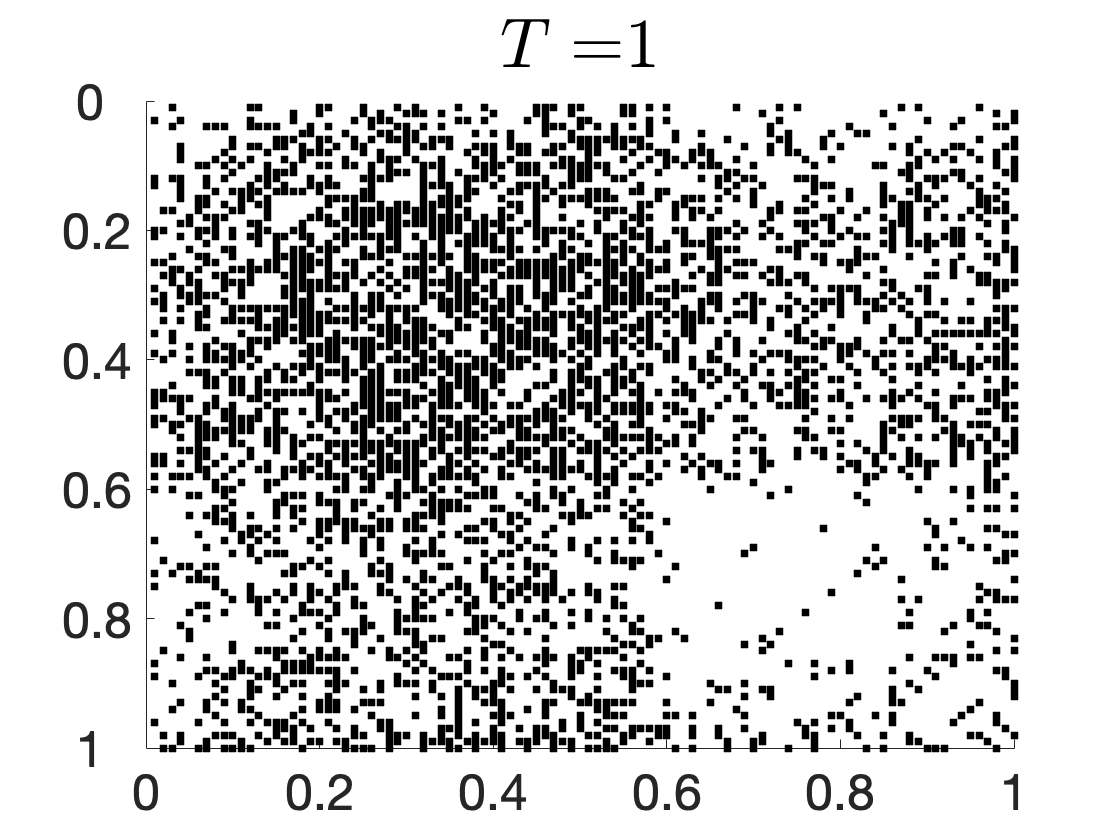}
\includegraphics[width=5cm]{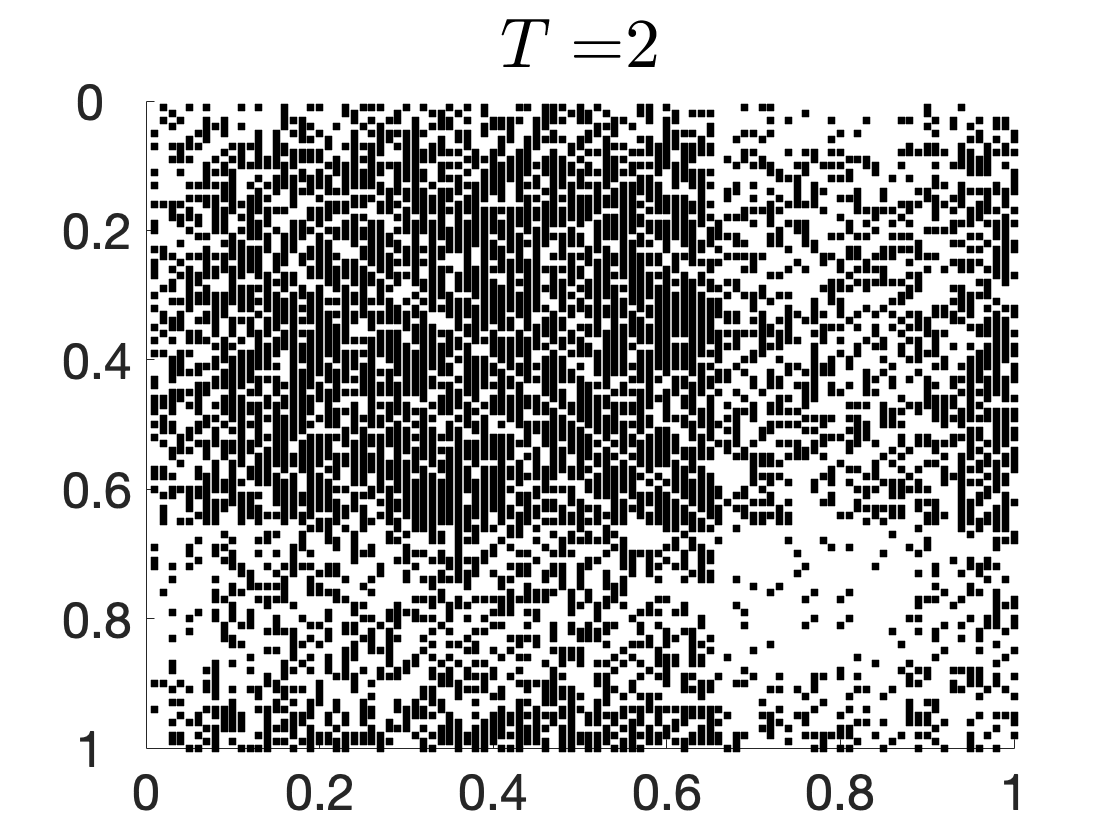}
\includegraphics[width=5cm]{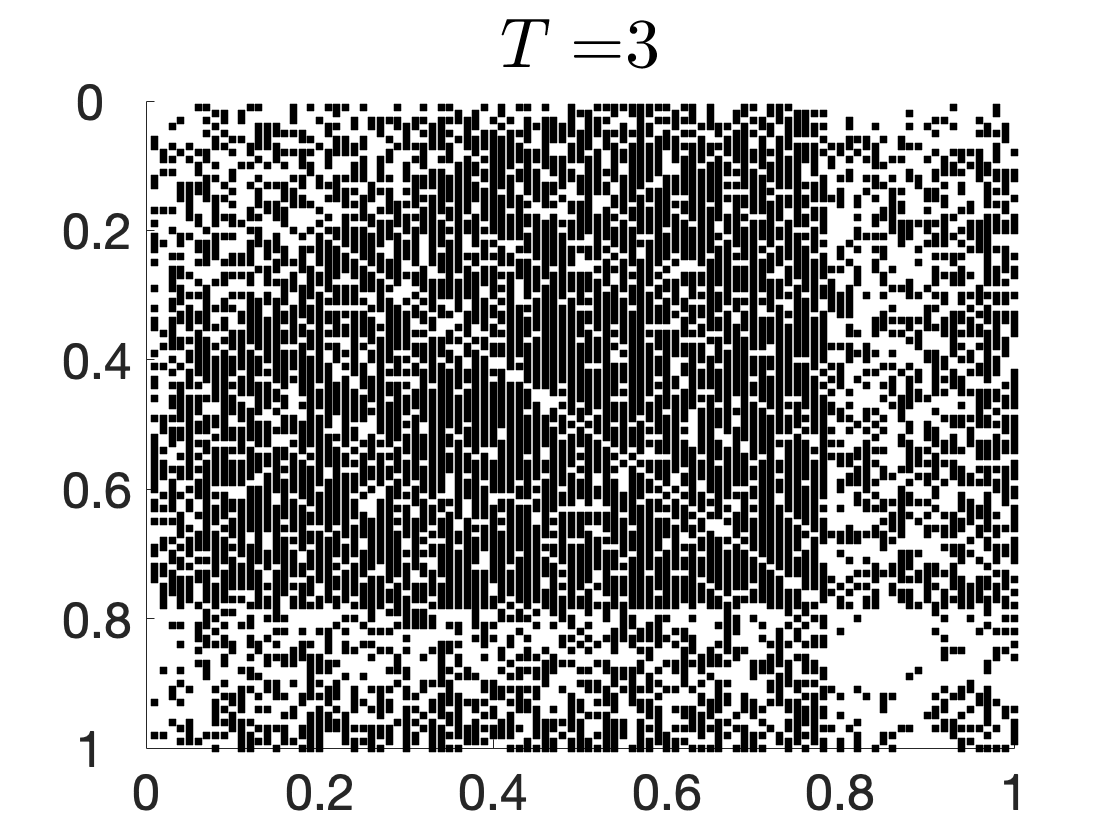}
	
\includegraphics[width=5cm]{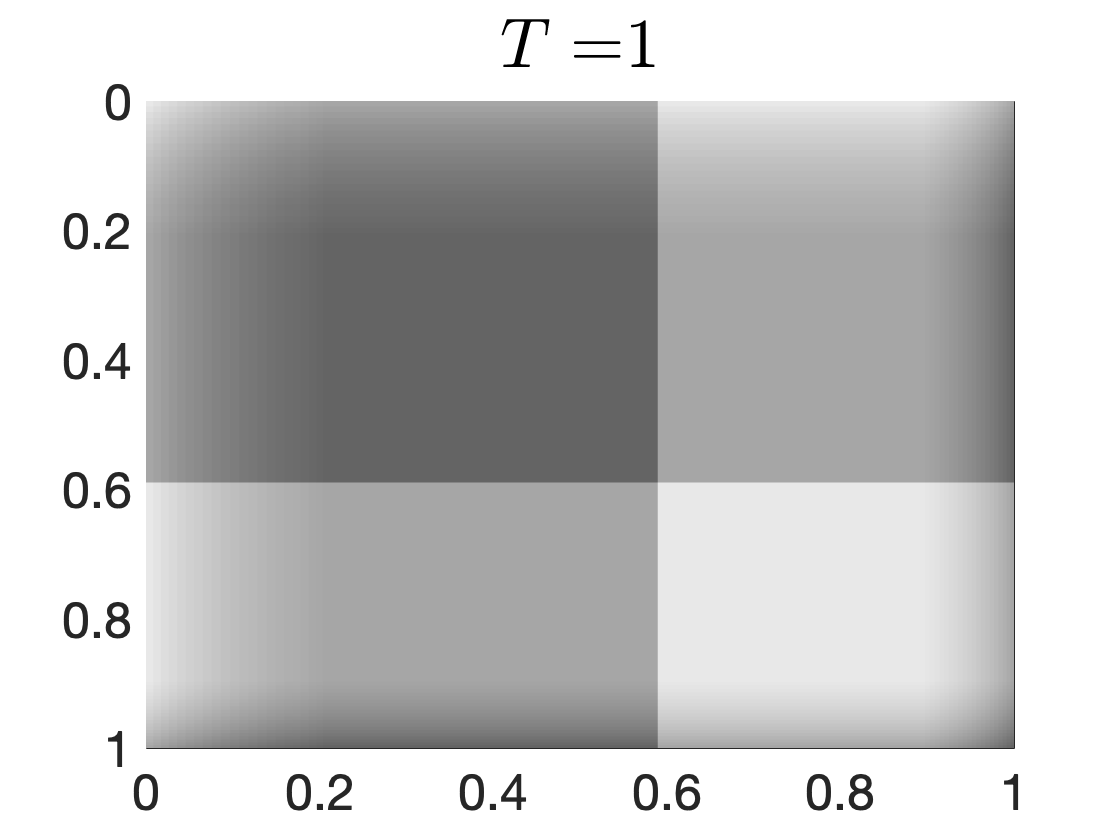}
\includegraphics[width=5cm]{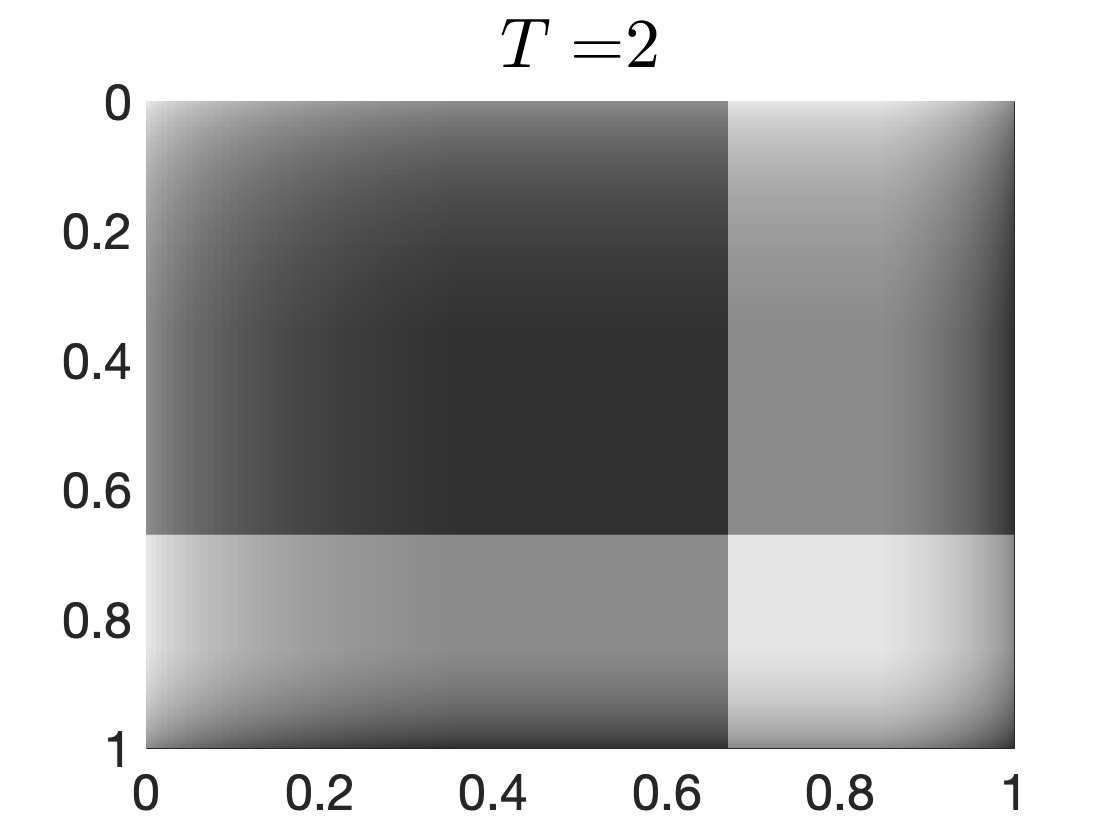}
\includegraphics[width=5cm]{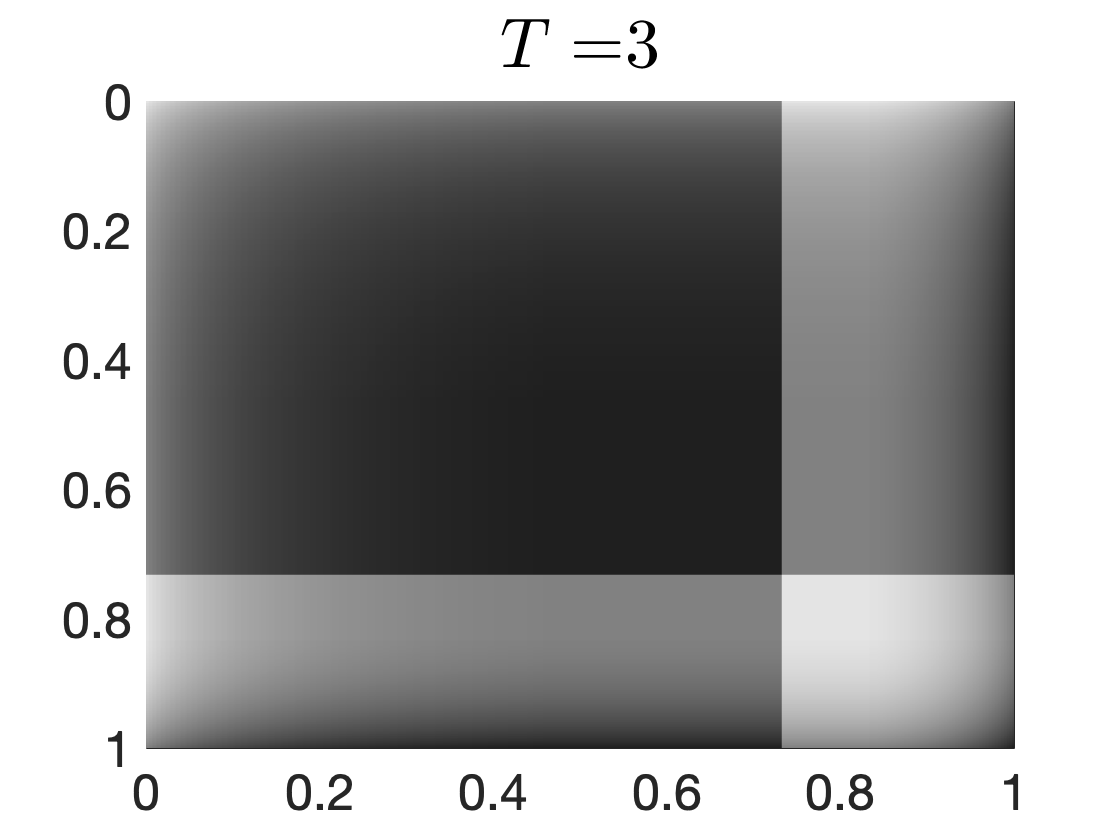}
\caption{\small The top row displays the empirical graphon when $n=100$ and $T=1,2,3$, the bottom row displays the corresponding functional law of large numbers. Simulations are based on a single run. A dot represents an edge. The labels of the vertices are updated dynamically so that they are ordered lexicographically, i.e., the vertices with opinion $+$ have lower labels than the vertices with opinion $-$, and then by increasing type.}
\label{fig-plot}
\end{figure}

\begin{figure}
\includegraphics[width=5cm]{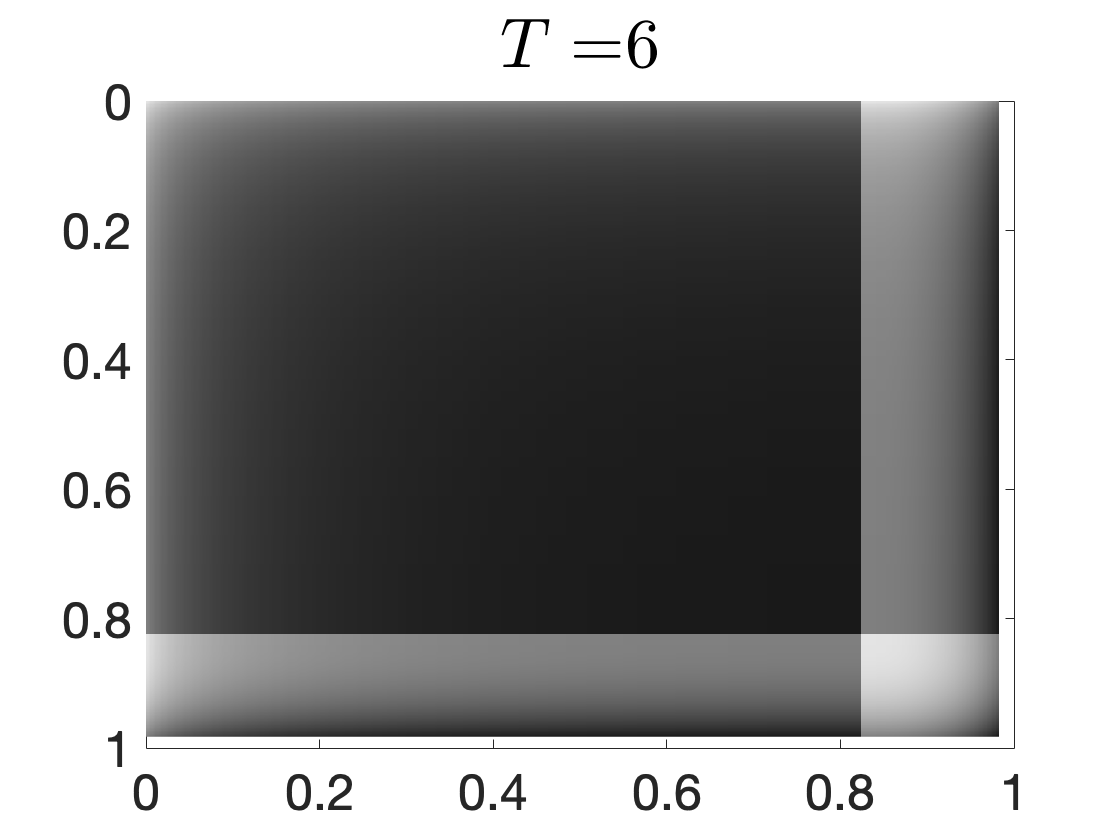}
\includegraphics[width=5cm]{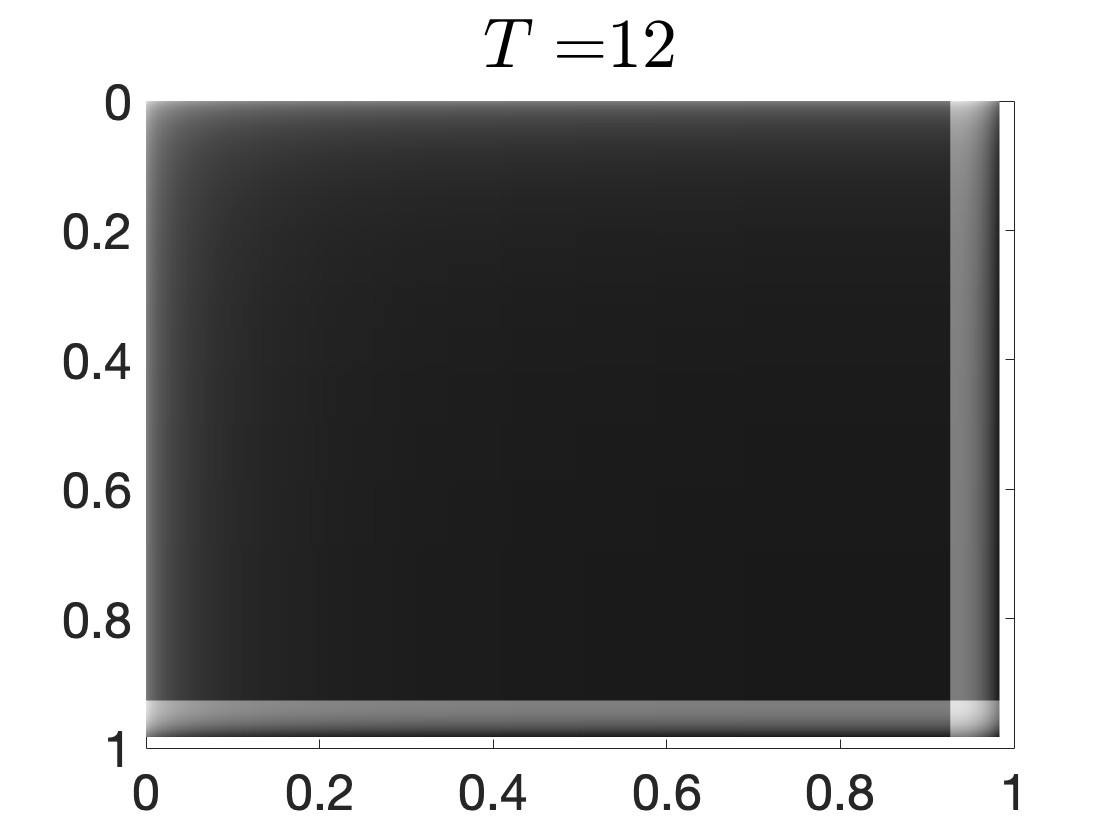}
\includegraphics[width=5cm]{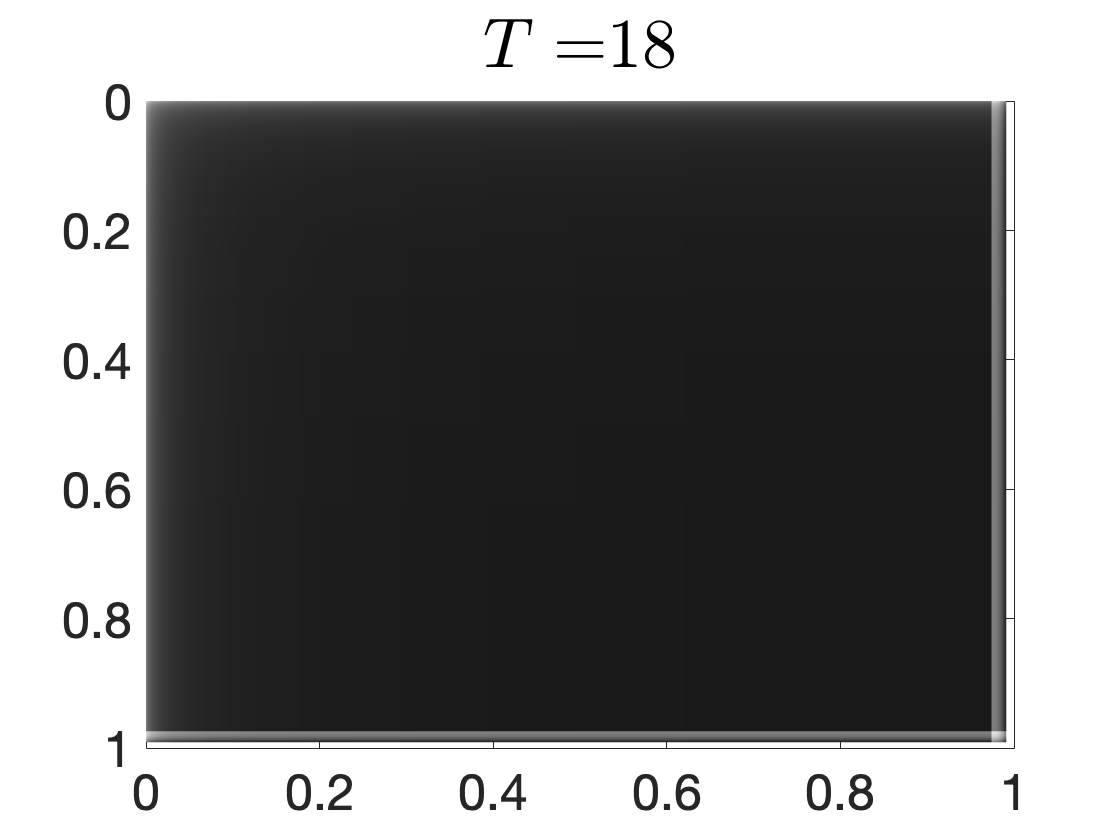}
\caption{\small Eventually all the vertices end up holding opinion $+$ in model 2.}
\label{fig-later}
\end{figure}

\begin{figure}
	\includegraphics[width=5.1cm]{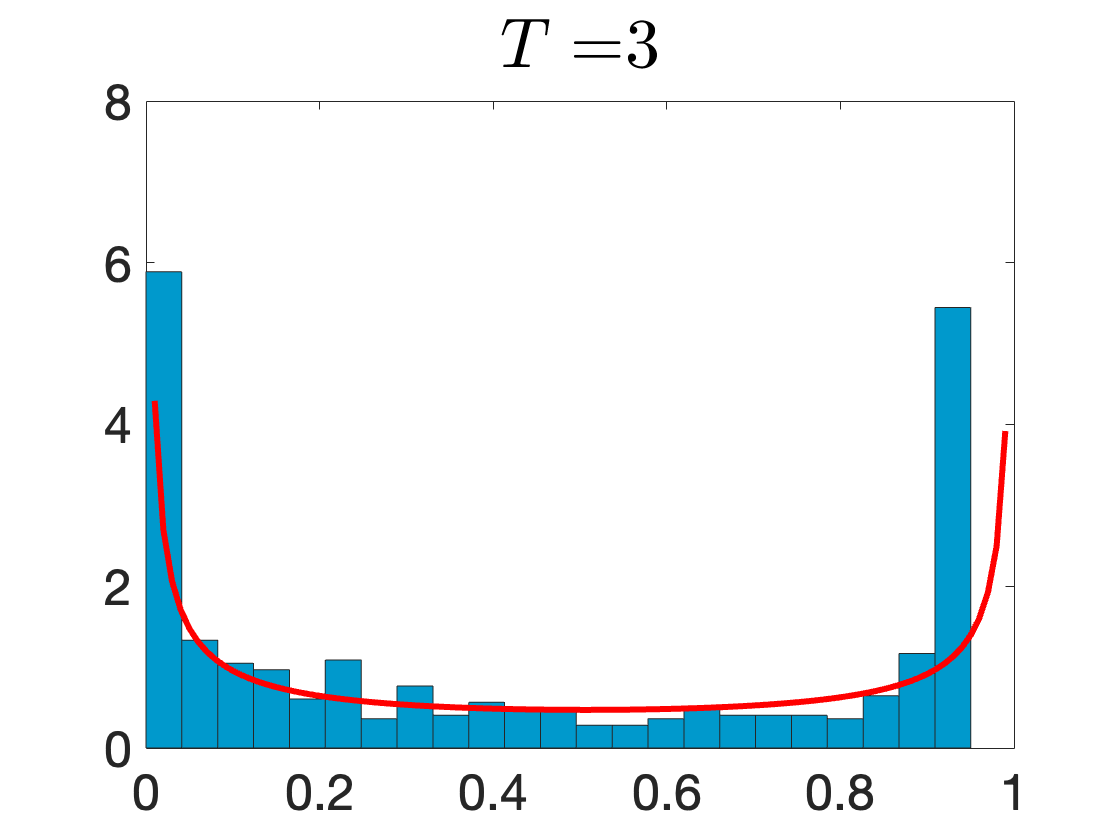}
	\includegraphics[width=5.1cm]{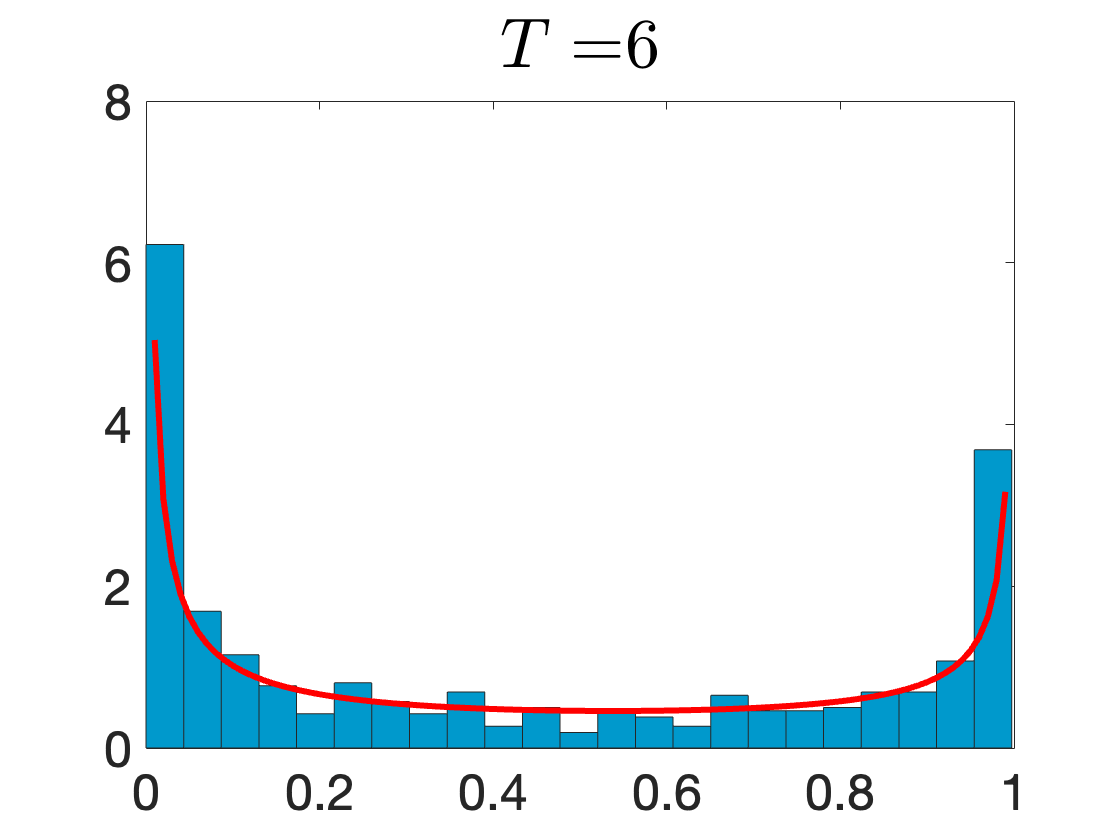}
	\includegraphics[width=5.1cm]{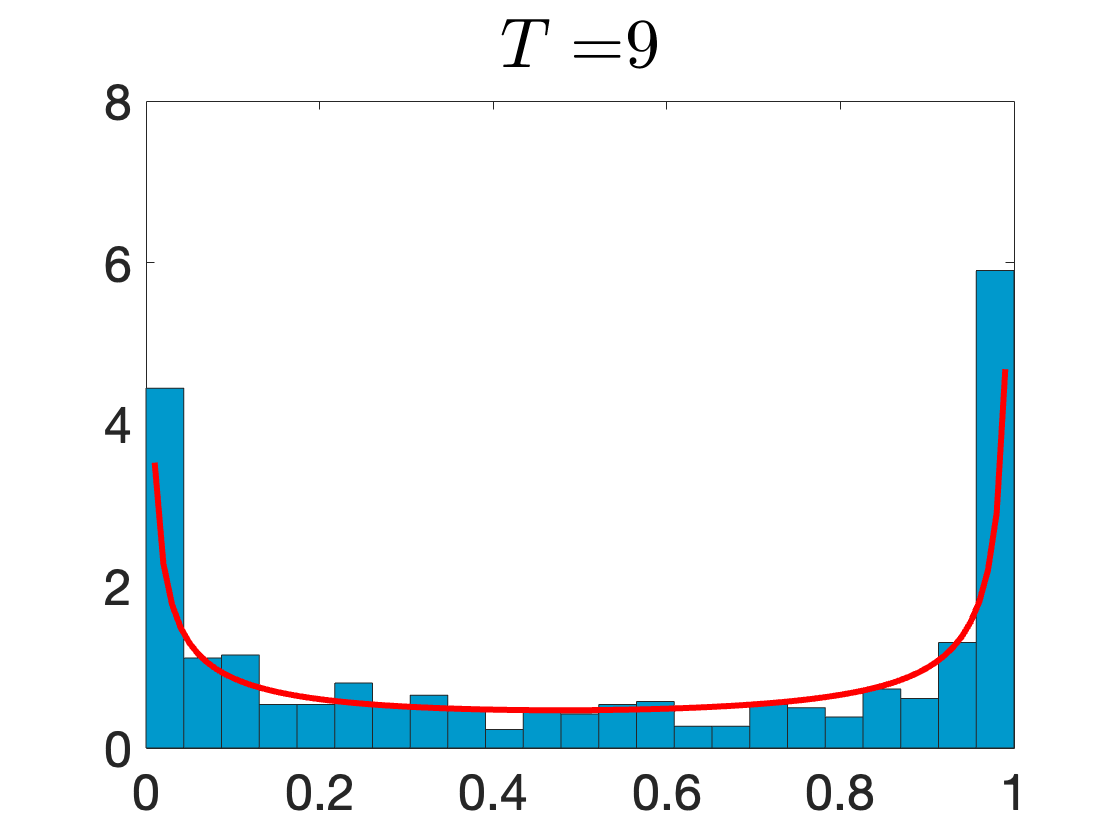}
	\caption{\small Empirical distribution of $(y_i(T))_{i\in [n]}$ when $n=600$ and $T=3,6,9$, for $\beta=0.66$, $\pi_+=\pi_-=0.6$. For any $i\in[n]$, $y_i(0)$ is chosen uniformly at random in $[0,1]$, independently of the other vertices. The red line indicates the corresponding $\hbox{Beta}(\beta p_+,\beta(1-p_+))$ density function, where $p_+$ denotes the proportion of vertices holding opinion $+$. Simulations are based on a single run.}
	\label{fig:histmod2}
\end{figure}


\subsection{Proof of the main theorems}
\label{mod2:pr}



\subsubsection{Coupling}
\label{sec:coupling}

To prove convergence of the process in the space of graphons, we follow Step 4 in Section \ref{sec:genstrategy}: we construct a mimicking process that satisfies Assumptions 1-3 in Section \ref{sec:genstrategy} and that is close to the original process in $L^1$. More precisely, we describe a coupling between the original and the mimicking process that is such that the number of differences in the vertex opinions that occur in a small time window $(t,t+\Delta)$, for a given $\Delta>0$, is $o(\Delta)$ as $n\to\infty$; similarly, the number of differences in the edges is $o(\Delta^2)$. Finally, we apply Hoeffding's inequality in the regime $\Delta\to 0$, in combination with the above bounds between the original and the mimicking process, to  establish the claim for the whole time interval $[0,T]$.

\medskip\noindent
{\it Mimicking process.} Suppose that the process $(G_n^*(t))_{t\in[0,T]}$ is characterised by the following dynamics:
\begin{itemize}
	\item $G_n^*(0)$ is an ERRG with connection probability $p_0$.
	\item Each vertex $i$ holds opinion $+$ or $-$, and is assigned a rate-$\beta$ Poisson clock. Each time when this clock rings vertex $i$ updates its opinion as follows: with probability $\alpha(t;y^*_i(t),\vec{v})$ it takes opinion $+$, where $y^*_i(t)$ denotes the type of vertex $i$ at time $t$ in $G_n^*(t)$. Otherwise, it takes opinion $-$.
	\item Each edge is assigned a rate-1 Poisson clock. Each time when this clock rings the edge $ij$ updates its state (either active or inactive): the edge is active with a probability that depends on the opinion of the selected vertices, namely, $\tfrac12(p(x^*_i(t))+p(x^*_j(t)))$, where $x^*_i(t)$ denotes the opinion of vertex $i$ at time $t$ in $G_n^*(t)$. Otherwise, the edge $ij$ is inactive.
\end{itemize}

\begin{lemma}
\label{lem:coupling}
There exists a coupling of $\{(G_n(t))_{t\in[0,T]}\}_{n\in\mathbb{N}}$ and $\{(G_n^*(t))_{t\in[0,T]}\}_{n\in\mathbb{N}}$ such that, for any $\epsilon>0$,
\begin{equation}
\lim_{n\to\infty} \dfrac{1}{n} \log \mathbb{P}\bigl(\|\tilde h^{G_n(t)}-\tilde h^{G_n^*(t)}\|_{L_1} > \epsilon
\hbox{ for some } t\in[0,T]\bigr) = 0.
\end{equation}
\end{lemma}

\begin{proof} The claim is proved in three steps.

\medskip\noindent
{\it Step I: description of the coupling.} Suppose that the outcomes of $(G_n(t))_{t\geq0}$ and $(G^*_n(t))_{t\geq0}$ are generated in the following manner.
\begin{itemize}
\item For each $i\in[n]$, vertex $i$ is assigned the same (coupled) rate-$\beta$ Poisson clock in both processes. When the clock associated to vertex $i$ rings, generate an outcome $u$ drawn from a $\hbox{Unif}(0,1)$ distribution.
\begin{itemize}
\item If $u\leq N_i^+(t)/N_i(t)$, then vertex $i$ takes opinion $+$, otherwise it takes opinion $-$, in $(G_n(t))_{t\geq0}$, where $N_i(t)$ and $N_i^+(t)$ denote the total number of neighbours of vertex $i$ and those having opinion $+$, respectively, in $(G_n(t))_{t\geq0}$.
\item If $u\leq\alpha(t;y^*_i(t),\vec{v})$, then vertex $i$ takes opinion $+$, otherwise it takes opinion $-$, in $(G^*_n(t))_{t\geq0}$.
\end{itemize}
\item Assign each edge the same (coupled) rate-$1$ Poisson clock in both processes. When the Poisson clock associated with edge $ij$ rings, generate an outcome $u$ drawn from a $\hbox{Unif}(0,1)$ distribution. 
\begin{itemize}
\item If $u\leq (p(x_i(t))+p(x_j(t)))/2$, then edge $ij$ is active, otherwise it is inactive, in $(G_n(t))_{t\geq0}$.
\item If $u\leq (p(x^*_i(t))+p(x^*_j(t)))/2$, then edge $ij$ is active, otherwise it is inactive, in $(G^*_n(t))_{t\geq0}$.
\end{itemize} 
\end{itemize}

\medskip\noindent
{\it Step II: majorization of the $L_1$ distance.} If vertex $i$ has the same opinion in both models and the clock associated with vertex $i$ rings, then a {\it difference} is formed (i.e., vertex $i$ has opinion $+$ in one process and opinion $-$ in the other) with probability 
\[
\left|\dfrac{N_i^+(t)}{N_i(t)}-\alpha(t;y^*_i(t),\vec{v})\right|.
\]
If edge $ij$ has the same state (active or inactive) in both models and the clock associated with edge $ij$ rings, then a difference is formed with probability 
\[
\dfrac{1}{2}\left|p(x_i(t))-p(x^*_i(t))+p(x_j(t))-p(x^*_j(t))\right|.
\]
Note that if the opinions of both vertices $i$ and $j$ are the same in the two processes, then the edge $ij$ has the same state in both $G_n(t)$ and $G_n^*(t)$. Thus, the two processes produce a different state for edges only if at their ends the vertices have different opinions.

\medskip\noindent
{\it Step III: bounding the dominating process.}

Our approach consists in showing that the number of differences that occur in some small time window $[t, t+\Delta)$, with $\Delta>0$, is stochastically dominated by a random variable with certain desirable properties. We subsequently consider the limit as $\Delta \to 0$ and use Hoeffding's inequality to establish the result.
	Let
	\begin{itemize}
		\item $t_\Delta = [t, t+ \Delta)$ be the time window,
		\item $N_E \sim{\rm Bin}({n \choose 2}, 1-e^{-\Delta})$ be the total number of edge clocks that ring in $t_\Delta$,
		\item $N_V \sim{\rm Bin}(n, 1- e^{-\beta \Delta})$ be the total number of vertex clocks that ring in $t_\Delta$,
		\item $d_E(t)$ be the total number of differences in the edges between $G_n(t)$ and $G^*_n(t)$,
		\item $d_V(t)$ be the total number of differences in the opinions of the vertices between $F_n(t;\cdot,\cdot)$ and $F^*_n(t;\cdot,\cdot)$, where $F_n$ is defined in \eqref{eq:emptype} and $F_n^*$ is defined similarly by referring to the graph $G^*_n(t)$. Here, referring to $F_n(t;\cdot,\cdot)$ and $F^*_n(t;\cdot,\cdot)$ means ``all information about the vertex types and opinions''.
	\end{itemize}
	Given $\Omega\equiv(N_E, N_V, d_E(t),d_V(t),F_n(t;\cdot,\cdot),F_n^*(t;\cdot,\cdot))$, we want to bound the probability $p_{V}$ that, when an edge clock (of an edge chosen uniformly at random) rings, a difference is formed during any time $s \in t_\Delta$. Note that if the clock of a vertex pair $ij$ rings, then a difference can only be formed when either $x_i(s) \neq x_i^*(s)$ or $x_j(s) \neq x_j^*(s)$. Observing that $d_V(t)=\sum_{i=1}^n \mathbf{1}\{ x_i(t) \neq x_i^*(t) \}$, considering the worst case scenario that every time a vertex clock rings during $t_\Delta$ a new difference is formed, and observing that each vertex $i$ is part of $n$ vertex pairs, we have
	\[
	p_{V} \leq \frac{n (d_V(t)+ N_V)}{{ n \choose 2}}.
	\]
Note that in the above inequality the denominator is the total number of vertex pairs and the numerator is a bound on the total number of vertex pairs that have at least one vertex with a difference of opinion between the two models. Because this bound is uniform in $t_\Delta$ and does not depend on $N_E$, given $\Omega$, we therefore have 
	\[
	d_E(t+\Delta) - d_E(t) \stackrel{{\rm st}}{\leq}{\rm Bin}\left(N_E, \frac{n (d_V(t)+ N_V)}{{ n \choose 2}}\right),
	\]
where $\stackrel{{\rm st}}{\leq}$ means ``stochastically dominated by''.

	Given $\Omega$, we next want to bound the probability that, when a vertex clock (of a vertex chosen uniformly at random) rings, a difference is formed during any time $s \in t_\Delta$, namely, we want to derive a bound for
	\[
	\frac{1}{n} \sum_{i=1}^n \left| \frac{N_i^+(s)}{N_i(s)} - \alpha(s;y^*_i(s),\vec{v}) \right|,
	\]
which applies for all $s \in t_\Delta$. This will be more involved than for the conditional probability that an edge difference is formed, which was derived above. Applying the triangle inequality, we have 
	\begin{align*}
		 \dfrac{1}{n} \displaystyle\sum_{i=1}^n \left| \frac{N_i^+(s)}{N_i(s)} - \alpha(s;y^*_i(s),\vec{v}) \right| 
		\leq \underbrace{\frac{1}{n} \sum_{i=1}^n \left| \frac{N_i^+(s)}{N_i(s)} - \frac{N_i^{*+}(s)}{N^*_i(s)} \right|}_{(i)} + \underbrace{\frac{1}{n} \sum_{i=1}^n \left| \frac{N_i^{*+}(s)}{N^*_i(s)} - \alpha(s;y^*_i(s),\vec{v}) \right|}_{{(ii)}},
	\end{align*}
 where $N_i^*(s)$ and $N_i^{*+}(s)$ denote the total number of neighbours of vertex $i$ and those having opinion $+$, respectively, in $G^*_n(s)$.
We first bound \emph{(i)}. For $\ell \in [0,1]$, let 
	\begin{equation}
		\label{eq:setN}
	\mathcal{N}_{\ell} = \{ i \in [n] : N_i(s) \leq \ell n \text{ or } N_i^*(s) \leq \ell n \text{ for some } s \in t_\Delta \}.
	\end{equation}
By Lemma \ref{lem:setN}, we know that it is possible to choose $\ell>0$ sufficiently small so that $|\mathcal{N}_{\ell}|=0$ with asymptotically high probability. In addition, we have
	\begin{align*}
		\frac{1}{n}\sum_{i=1}^n   \left| \frac{N^+_i(s)}{N_i(s)} - \frac{N^{*+}_i(s)}{N^*_i(s)} \right| 
		&= \frac{1}{n} \sum_{i \in [n] \backslash \mathcal{N}_{\ell}} \left| \frac{N^+_i(s)}{N_i(s)} - \frac{N^{*+}_i(s)}{N^*_i(s)} \right| + \frac{1}{n} \sum_{i \in \mathcal{N}_{\ell}}   \left| \frac{N^+_i(s)}{N_i(s)} - \frac{N^{*+}_i(s)}{N^*_i(s)} \right| \\
		&\leq  \frac{1}{n} \sum_{i \in [n] \backslash \mathcal{N}_{\ell}}   \left| \frac{N^+_i(s)}{N_i(s)} - \frac{N^{*+}_i(s)}{N_i(s)} \right|+ \frac{1}{n} \sum_{i \in [n] \backslash \mathcal{N}_{\ell}}   \left| \frac{N^{*+}_i(s)}{N_i(s)} - \frac{N^{*+}_i(s)}{N^*_i(s)} \right| +  \frac{1}{n} |\mathcal{N}_{\ell} | \\
		&\leq \frac{1}{n} \sum_{i \in [n] \backslash \mathcal{N}_{\ell}} \frac{1}{\ell n} \left| N^{+}_i(s) - N^{*+}_i(s) \right| + \sum_{i \in [n] \backslash \mathcal{N}_{\ell}}  \sup_{s \in t_\Delta} \left| \frac{N^{*+}_i(s) - N^{+}_i(s)}{N_i(s)N^*_i(s)} \right| +  \frac{1}{n} |\mathcal{N}_{\ell} | \\
		&\leq \frac{2}{\ell n^2} (d_E(t) + N_E) +  \frac{2}{\ell^2 n^2} (d_E(t) + N_E) + \frac{1}{n} |\mathcal{N}_{\ell}|.
	\end{align*} 
	
We next bound \emph{(ii)}. To do so we use the convergence of the mimicking process in the space of graphons. We therefore rewrite the expressions $N^{*+}_i(s)/N^*_i(s)$ and $\alpha(s;y^*_i(s),\vec{v})$ in terms of their empirical and limiting graphons, respectively. Recall that, as per Assumption 4, vertices are labelled dynamically, so that all opinion $+$ vertices have a lower label than opinion $-$ vertices. Let $r^+_n(s)$ denote the proportion of opinion $+$ vertices, and $r^+(s)$ the proportion of opinion $+$ vertices, both in the limiting graphon at time $s$. Let $g^*_n(s;u,v) = h^{G^*_n}(s;u,v)$ denote the empirical graphon at time $s$. We first note that 
	\[
	\dfrac{N^{*+}_i(s)}{N^*_i(s)} =  \dfrac{\int_{0}^{r^{*+}_n(s)} {\rm d}u \, g_n^{*}(s; \frac{i}{n},u)}{ \int_{0}^{1} {\rm d}u \, g^{*}_n(s; \frac{i}{n},u)},
	\]
which follows from definition of the empirical graphon. In addition, by Lemma \ref{lmm:repalpha}, 
	\begin{equation}\label{eqn:aExp}
		\alpha(s;y^*_i(s),\vec{v}) = \dfrac{{\int_0^{r^+(s)}} {\rm d}y \, g^{[F]}(s;y,F(s;y^*_i(s)))}{\int_0^1 {\rm d}y \, g^{[F]}(s;y,F(s; y^*_i(s)))}.
	\end{equation}
We therefore have 
	\begin{align*}
		\frac{1}{n}\sum_{i=1}^n  \left| \dfrac{N^{*+}_i(s)}{N^*_i(s)}- \alpha(s;y^*_i(s),\vec{v})\right| 
		&=    \frac{1}{n}\sum_{i=1}^n  \left| \dfrac{\int_{0}^{r^+_n(s)} {\rm d}u \, g_n^{*}(s; \frac{i}{n},u)}{ \int_{0}^{1} {\rm d}u \, g^{*}_n(s; \frac{i}{n},u)}- \dfrac{{\int_0^{r^+(s)}} {\rm d}y \, g^{[F]}(s;y,F(s;y^*_i(s)))}{\int_0^1 {\rm d}y \, g^{[F]}(s;y,F(s; y^*_i(s)))}\right| \\
		&= \int_0^1 {\rm d}v \left| \dfrac{\int_{0}^{r^+_n(s)} {\rm d}u \, g_n^{*}(s; v,u)}{ \int_{0}^{1} {\rm d}u \, g^{*}_n(s; v,u)}- \dfrac{{\int_0^{r^+(s)}} {\rm d}y \, g^{[F]}(s;y,F(s;\bar F^*_n(s;v)))}{\int_0^1 {\rm d}y \, g^{[F]}(s;y,F(s; \bar F^*_n(s;v)))}\right| \\
		&\leq \underbrace{\int_0^1 {\rm d}v \left| \dfrac{\int_{0}^{r^+_n(s)} {\rm d}u \, g_n^{*}(s; v,u)}{ \int_{0}^{1} {\rm d}u \, g^{*}_n(s; v,u)}- \dfrac{{\int_0^{r^+(s)}} {\rm d}y \, g^{[F]}(s;y,v)}{\int_0^1 {\rm d}y \, g^{[F]}(s;y,v)}\right|}_{\rm (a)} \\
		&\quad + \underbrace{\int_0^1 {\rm d}v \left| \dfrac{{\int_0^{r^+(s)}} {\rm d}y \, g^{[F]}(s;y,v)}{\int_0^1 {\rm d}y \, g^{[F]}(s;y,v)} - \dfrac{{\int_0^{r^+(s)}} {\rm d}y \, g^{[F]}(s;y,F(s;\bar F^*_n(s;v)))}{\int_0^1 {\rm d}y \, g^{[F]}(s;y,F(s; \bar F^*_n(s;v)))} \right|}_{\rm (b)}.
	\end{align*}
We bound (a) and (b) separately. We start with (a), where again by Lemma \ref{lem:setN} we have that $|\mathcal{N}_{\ell}|=0$ for some $\ell>0$ with high probability,
	\begin{align*}
		&\int_0^1 {\rm d}v \left| \dfrac{\int_{0}^{r^+_n(s)} {\rm d}u \, g_n^{*}(s; v,u)}{ \int_{0}^{1} {\rm d}u \, g^{*}_n(s; v,u)}- \dfrac{{\int_0^{r^+(s)}} {\rm d}y \, g^{[F]}(s;y,v)}{\int_0^1 {\rm d}y \, g^{[F]}(s;y,v)}\right| \\
		&\leq \int_0^1 {\rm d}v \left| \dfrac{\int_{0}^{r^+_n(s)} {\rm d}u \, g_n^{*}(s; v,u)}{ \int_{0}^{1} {\rm d}u \, g^{*}_n(s; v,u)}- \dfrac{{\int_0^{r^+(s)}} {\rm d}y \, g^{[F]}(s;y,v)}{\int_{0}^{1} {\rm d}u \, g^{*}_n(s; v,u)}\right| \\ 
		&\qquad +
		\int_0^1 {\rm d}v \left| \dfrac{{\int_0^{r^+(s)}} {\rm d}y \, g^{[F]}(s;y,v)}{ \int_{0}^{1} {\rm d}u \, g^{*}_n(s; v,u)}- \dfrac{{\int_0^{r^+(s)}} {\rm d}y \, g^{[F]}(s;y,v)}{\int_0^1 {\rm d}y \, g^{[F]}(s;y,v)}\right| \\
		&\leq \frac{1}{\ell} \int_0^1 {\rm d}v \left|\int_{0}^{r^+_n(s)} {\rm d}u \, g_n^{*}(s; v,u)  - {\int_0^{r^+(s)}} {\rm d}y \, g^{[F]}(s;y,v) \right| + \frac{1}{\ell^2} \int_0^1 {\rm d}v \left| \int_{0}^1 {\rm d}u \, [g^*_n(s;v,u) - g^{[F]}(s;u,v)] \right| \\
		&\leq \frac{1}{\ell}|r^+_n(s) - r^+(s)| + \frac{1}{\ell} \int_0^1 {\rm d}v \left|\int_{0}^{r^+(s)} {\rm d}u \, [g_n^{*}(s; v,u)  - g^{[F]}(s;v,u)] \right| \\
		&\qquad + \frac{1}{\ell^2} \int_0^1 {\rm d}v \left| \int_{0}^1 {\rm d}u \, [g^*_n(s;v,u) - g^{[F]}(s;u,v)] \right| \\
		&\leq \frac{1}{\ell}|r^+_n(s) - r^+(s)| + \frac{2}{\ell}  d_\square(g^*_n(s; \cdot), g^{[F]}(s;\cdot)) + \frac{2}{\ell^2} d_\square(g^*_n(s; \cdot), g^{[F]}(s;\cdot)).
	\end{align*}
Note that, by the functional law of large numbers, as $n \to \infty$ each of these terms tends to zero (uniformly in $s$) with high probability.

Next, we focus on (b). Since $F_n^*\Rightarrow F$, it follows from the continuous map theorem that 
	\[
	d_V(t+\Delta) - d_V(t) \stackrel{{\rm st}}{\leq}{\rm Bin}\left(N_V , \frac{2 }{\ell n^2}\left(1+ \frac{1}{\ell}\right)(d_E(t)+ N_E) +\frac{1}{n}|\mathcal{N}_{\ell}|+\varepsilon(n)\right),
	\]
where $\varepsilon(n)$ is a function that can be chosen arbitrarily small with high probability, and can therefore effectively be ignored in the limit $n\to\infty$. Again, as $n\to\infty$ we have that $|{\cal N}_{\ell}|\to0$ by Lemma \ref{lem:setN}.

Our plan now is to bound the number of differences formed in the vertices and edges between the two models by considering the time intervals $[0,\Delta]$, $[\Delta, 2 \Delta]$, $\dots$, $[T-\Delta,T]$ in sequence, and then use the union bound. We will use the results above which, for $n$ large enough, can be summarised as
	\begin{equation}
		\begin{aligned}\label{eqn:coB}
			N_E &\stackrel{{\rm st}}{\leq}{\rm Bin}\left( {n \choose 2}, \Delta\right), \\
			N_V &\stackrel{{\rm st}}{\leq}{\rm Bin}\left( n , \beta \Delta\right), \\
			d_E(t+\Delta) - d_E(t) &\stackrel{{\rm st}}{\leq}{\rm Bin}\left(N_E, \frac{n (d_V(t)+ N_V)}{{ n \choose 2}}\right), \\
			d_V(t+\Delta) - d_V(t) &\stackrel{{\rm st}}{\leq}{\rm Bin}\left(N_V , \frac{2 }{\ell n^2}\left(1+ \frac{1}{\ell}\right)(d_E(t)+ N_E)\right).
		\end{aligned}
	\end{equation}
For any $\Delta>0$, by Hoeffding's inequality (see \eqref{eq:Hoeffgen}), we have 
	\begin{equation}
		\begin{aligned}\label{eqn:HoefN}
			\mathbb{P}\left( N_E \leq 2 n^2 \Delta \right) &\leq \exp{(-2 n^2 \Delta^2)}, \\
			\mathbb{P}\left( N_V \leq 2 n \beta \Delta \right) &\leq \exp{(-2 n \beta^2 \Delta^2)}.
		\end{aligned}
	\end{equation}
Let $t_i=i\Delta$ for $i=0,\dots,T/\Delta$, and $t_\Delta^{(i)}=[t_{i-1},t_i]$ for $i=1,\dots,T/\Delta$. In addition, let $N^{(i)}_E$ (respectively, $N^{(i)}_V$) denote the number of edge (respectively, vertex) clocks that ring during the time interval $t_\Delta^{(i)}$. By \eqref{eqn:HoefN}, for any $\Delta >0$, we have 
	\begin{align*}
		&\limsup_{n\to\infty}\mathbb{P}\left( \exists i \in \{1, \dots, T/\Delta\} : N^{(i)}_E \geq n \Delta \right) \leq \lim_{n\to\infty}\frac{T}{\Delta} \exp(-2n^2\Delta) = 0, \\
		&\limsup_{n\to\infty}\mathbb{P}\left( \exists i \in \{1, \dots, T/\Delta\} : N^{(i)}_V \geq n \Delta \right) \leq \lim_{n\to\infty}\frac{T}{\Delta} \exp(-2n\beta^2\Delta^2) = 0.
	\end{align*}
For this reason it is justified to let $N^{(i)}_V \leq 2 n \beta \Delta$ and $N^{(i)}_E \leq n^2 \Delta$ for all $i$. Hence, by applying \eqref{eqn:coB} we obtain the following stochastic inequalities:
	\begin{align*}
		d_E(t_i)-d_E(t_{i-1}) &\stackrel{\rm st}{\leq} {\rm{\rm Bin}}\left( n^2 \Delta, \frac{n d_V(t_{i-1}) + 2n^2 \beta \Delta}{n^2/2} \right), \\
		d_V(t_i)-d_V(t_{i-1}) &\stackrel{\rm st}{\leq} {\rm Bin}\left( 2 n \beta \Delta, \frac{2}{\ell n^2}\left(1 + \frac{1}{\ell}\right)(d_E(t_{i-1})+n^2\Delta)\right).
	\end{align*}
Applying Hoeffding's inequality once more, and using the union bound, we obtain, for any $\Delta>0$,
	\begin{align*}
		\mathbb{P}&\left( \exists i \in \{1, \dots, T/\Delta\} : d_E(t_i)-d_E(t_{i-1}) \geq \frac{2 n^2 \Delta (nd_V(t_{i-1}) + 2n^2 \beta \Delta }{n^2/2} \right) \\
		&\leq \sum_{i=1}^{T/\Delta} \exp \left( - 2 n^2 \Delta \left( \frac{n d_V(t_{i-1}) + 2 n^2 \beta \Delta}{n^2 / 2} \right)^2 \right) \leq \frac{T}{\Delta} \exp (-2 n^2 \beta^2 \Delta^2) \to 0
	\end{align*}
as $n \to \infty$. Similarly,
	\begin{align*}
		\mathbb{P}&\left(  \exists i \in \{1, \dots, T/\Delta\} : d_V(t_i) -d_V(t_{i-1}) \geq 2 \cdot 2 n \beta \Delta \frac{2}{\ell n^2}(1 + \frac{1}{\ell})(d_E(t_{i-1}) + n^2 \Delta) \right) \\
		&\leq \sum_{i=1}^{T/\Delta} \exp \left( -2(2n\beta \Delta)^2 \frac{2}{\ell n^2} (1+\frac{1}{\ell})(d_E(t_{i-1})+n^2 \Delta) \right) \leq \frac{T}{\Delta} \exp \left( -2 (2 n \beta \Delta)^2 \frac{2}{\ell n^2}(1 + \frac{1}{\ell})n^2 \Delta \right) \to 0
	\end{align*}
as $n \to \infty$ for any $\Delta>0$. For any $i$, we can therefore restrict ourselves to
	\begin{equation}
		\label{eq:assumption}
		\begin{aligned}
			d_E(t_{i}) - d_{E}(t_{i-1}) &\leq  b_1 d_V(t) \Delta n + b_2 n^2 \Delta^2, \\
			d_V(t_i) - d_V(t_{i-1}) &\leq  \frac{c_1 d_E(t_{i-1}) \Delta}{n} + c_2 n \Delta^2,
		\end{aligned}
	\end{equation}
where $b_1,b_2,c_1,c_2 <\infty$ are fixed and do not depend on $\Delta$ or $n$. Let $a_1 = \max\{ b_1, c_1\}$ and $a_2 =\max\{ b_2, c_2\}$. By Lemma \eqref{lmm:induction}, we have
	\begin{align*}
		d_V(T) &\leq n a_2 \Delta^2 \sum_{k=1}^{T/\Delta} (a_1 \Delta)^{k-1} {i \choose k} \\ 
		&= n a_2 \Delta^2 \left[ (a_1 \Delta)^{-1} \left( \sum_{k=0}^{T/\Delta} (a_1\Delta)^k {T/\Delta \choose k} - 1 \right) \right]
		= n a_2 \Delta^2 \left[ (a_1 \Delta)^{-1} (1+a_1 \Delta)^{T/\Delta} -1 \right] \\
		&\to n a_2 \Delta^2 \left[ (a_1 \Delta)^{-1}\exp(a_1 T) -1 \right] = n\frac{a_2}{a_1} \Delta (\exp(a_1 T)-1)
	\end{align*}
as $n\to\infty$. Taking $\Delta \downarrow 0$ we conclude that $d_V(T) = o(n)$ as $n \to \infty$ with probability 1. We can similarly show that $d_E(T) = o(n^2)$ with probability 1. The desired result now follows.
\end{proof}

\begin{lemma}\label{lem:mimproc}
	The process $\{(G_n^*(t))_{t\geq0}\}_{n\in\mathbb{N}}$ satisfies the assumptions of \cite[Theorem 3.10]{BdHM22pr}.
\end{lemma}

\begin{proof}
	The dynamics of each vertex in the mimicking process is independent of the other vertices. We can therefore follow the same reasoning as for the first model to establish a functional law of large numbers for the empirical type process as $n\to \infty$. The other assumptions of \cite[Theorem 3.10]{BdHM22pr} also follow from similar arguments.
\end{proof}


\subsubsection{Proof of Theorem~\ref{thm:graphonconvalt}}

Theorem~\ref{thm:graphonconvalt} now follows from Lemmas~\ref{lem:coupling}-~\ref{lem:mimproc} and \cite[Theorem 3.10]{BdHM22pr}.


\subsubsection{Proof of Theorem~\ref{thm:systempdealt}}
\label{sec:pdemodel2}

Recall the definition of the matrix $M(u)$ given in\eqref{eq:defM}. The system of two coupled PDEs can be written as
\begin{equation}
\label{eq:system2}
\dfrac{\partial}{\partial t} \vec{v}(t,u) = M(u) \dfrac{\partial}{\partial u} \vec{v}(t,u) + N(t;u,\vec{v}) \vec{v}(t,u),
\end{equation}
where	
\begin{equation}
\label{eq:matrixN}
N(t;u,\vec{v})=  \left(\begin{matrix} 1-\beta(1-\alpha(t;u,\vec{v}))
& \beta \alpha(t;u,\vec{v}) \\ \beta(1-\alpha(t;u,\vec{v}))
& 1- \beta \alpha(t;u,\vec{v}) \end{matrix}\right).
\end{equation}
The statement directly follows from the Cauchy–Kovalevskaya theorem \cite{K1875} after noting that the function $\alpha(t;u,\vec{v})$ is analytic. Indeed, the function $H(t;u,v)$ depends exponentially on $t$, is linear in both $u$ and $v$, and therefore it is analytic. In addition, we are looking at analytic vectors $\vec{v}(t,u)$, so that $\alpha(t;u,\vec{v})$ is defined as the ratio of two analytic functions, and therefore is analytic.


\subsubsection{Proof of Theorem \ref{thm:systempdeasymp2}}

We start with case (i). The claimed Beta limit follows by checking that the densities $f_+(\infty,u)$ and $f_-(\infty,u)$ given in \eqref{eq:limdensities2} satisfy 
\begin{equation}
	\label{eq:mod2asymp}
	M(u)\dfrac{\partial}{\partial u}\vec{v}(\infty,u) + N(\infty;u,\vec{v}) \vec{v}(\infty,u) = 0.
\end{equation}
By a straightforward computation we deduce that this is true if and only if 
\[
\alpha(\infty;u,\vec{v}) = p_+,
\]
where
\begin{equation}\label{eq:alphaasymp}
\alpha(\infty;u,\vec{v})= \lim_{t\to\infty} \alpha(t;u,\vec{v})
= \dfrac{\int_0^1 {\rm d}y\,f_+(\infty,y) H(\infty;y,u)}{\int_0^1 {\rm d}y\, [f_+(\infty,y) + f_-(\infty,y)]\,H(\infty;y,u)}
\end{equation}
and
\[
H(\infty;y,u) = \pi_-+\Pi(y+u), \qquad \Pi = \dfrac{\pi_+-\pi_-}{2}.
\]
In this case we have  $\pi_+=\pi_-$, which corresponds to $\Pi=0$ and $H(\infty;y,u)=\pi_-$, 
\[
\alpha(\infty;u,\vec{v}) = \dfrac{\int_0^1 {\rm d}y\,f_+(\infty,y)}{\int_0^1 {\rm d}y\, [f_+(\infty,y) + f_-(\infty,y)]} = \int_0^1 {\rm d}y\,f_+(\infty,y) = p_+.
\]
To conclude the proof of (i), it remains to show that no other analytic densities satisfy \eqref{eq:mod2asymp}. For any $u\in[0,1]$ we may write
$$
f_+(\infty,u) = \sum_{k=0}^\infty f_{+,k} u^k, \qquad 
f_-(\infty,u) = \sum_{k=0}^\infty f_{-,k} u^k,
$$
so that 
$$
\int_0^1  \dd y f_+(\infty,y) H(\infty;y,u)= (\pi_- + \Pi u) m_{+,0} + \Pi m_{+,1},
$$
where
$$
m_{+,0} = \sum_{k=0}^\infty \dfrac{f_{+,k}}{k+1}, \qquad 
m_{+,1} = \sum_{k=0}^\infty \dfrac{f_{+,k}}{k+2}.
$$
Similarly, we set
$$
m_{-,0} = \sum_{k=0}^\infty \dfrac{f_{-,k}}{k+1}, \qquad 
m_{-,1} = \sum_{k=0}^\infty \dfrac{f_{-,k}}{k+2}.
$$
Note that $m_{+,0}=p_+$ and $m_{+,0}+m_{-,0}=1$. After setting $m_0=m_{+,0}+m_{-,0}$ and $m_1=m_{+,1}+m_{-,1}$, we see that \eqref{eq:alphaasymp} becomes
$$
\alpha(\infty;u,\vec{v})
= \dfrac{(\pi_-+\Pi u)m_{+,0} + \Pi m_{+,1}}{(\pi_-+\Pi u)m_0 + \Pi m_1}.
$$
This means that, after putting
$$
A_{+,0} = \Pi m_{+,0}, \qquad A_{+,1} = \pi_-m_{+,0} + \Pi m_{+,1}, \qquad
A_{-,0} = \Pi m_{-,0}, \qquad A_{-,1} = \pi_-m_{-,0} + \Pi m_{-,1},
$$
and $A_0=A_{+,0}+A_{-,0}$, $A_1=A_{+,1}+A_{-,1}$, we can write \eqref{eq:mod2asymp} as
$$
\left\{
\begin{aligned}
	& -\left[\beta (A_{-,1} + uA_{-,0}) - (A_{1} + uA_{0})\right] \sum_{k=0}^\infty f_{+,k} u^k + \beta(A_{+,1}+uA_{+,0}) \sum_{k=0}^\infty f_{-,k} u^k \\
	&\qquad = (1-u)(A_{1} + uA_{0}) \sum_{k=1}^\infty k f_{+,k} u^{k-1}, \\
	&- \left[\beta(A_{+,1} + uA_{+,0}) - (A_{1} + uA_{0})\right] \sum_{k=0}^\infty f_{-,k} u^k + \beta (A_{-,1} + uA_{-,0}) \sum_{k=0}^\infty f_{+,k} u^k \\
	&\qquad = -u(A_{1} + uA_{0}) \sum_{k=1}^\infty k f_{-,k} u^{k-1},
\end{aligned}
\right.
$$
which provides us with recursive relations between the coefficients. Concretely, equating the coefficients pertaining to both sides of the equations yields
\begin{equation}
	\label{eq:systmod2zeroas}
	A_1	f_{+,1} = -[ \beta A_{-,1} - A_1] f_{+,0} +\beta A_{+,1} f_{-,0}, \qquad
	0 = -[ \beta A_{+,1} - A_1] f_{-,0} +\beta A_{-,1} f_{+,0},
\end{equation}
and
\begin{equation}
	\label{eq:systmod2zeroasbis}
	\begin{aligned}
		-(\beta A_{-,1} - A_1) f_{+,1} - (\beta A_{-,0} - A_0) f_{+,0} + \beta A_{+,1} f_{-,1} + \beta A_{+,0} f_{-,0}
		&= 2 A_1	f_{+,2} + (A_0-A_1) f_{+,1}, \\
		-(\beta A_{+,1} - A_1) f_{-,1} - (\beta A_{+,0} - A_0) f_{-,0} + \beta A_{-,1} f_{+,1} + \beta A_{-,0} f_{+,0} 
		&= -A_1 f_{-,1},
	\end{aligned}
\end{equation}
and, for $k\geq2$,
\begin{equation}
	\label{eq:systmod2as}
	\begin{aligned}
		&-(\beta A_{-,1} - A_1) f_{+,k} - (\beta A_{-,0} - A_0) f_{+,k-1} + \beta A_{+,1} f_{-,k} + \beta A_{+,0} f_{-,k-1} \\	
		&\quad= (k+1)A_1 f_{+,k+1} + k(A_0-A_1) f_{+,k} - (k-1)A_0 f_{+,k-1}, \\
		&-(\beta A_{+,1}-A_1) f_{-,k} - (\beta A_{+,0} - A_0) f_{-,k-1} + \beta A_{-,1} f_{+,k} + \beta A_{-,0} f_{+,k-1} \\	
		&\quad= -kA_1 f_{-,k} - (k-1)A_0 f_{-,k-1}.
	\end{aligned}
\end{equation}

Since the condition $\pi_+=\pi_-$ implies that $A_{+,0}=A_{-,0}=0$, from the second equation in \eqref{eq:systmod2as} we can express $f_{-,k}$ in terms of $f_{+,k}$, so that the first equation in \eqref{eq:systmod2as} allows us to write $f_{+,k+1}$ in terms of $f_{+,k}$ only. Thus, the coefficients $f_{+,k}$ and $f_{-,k}$ are uniquely determined once $f_{+,0}$ is known. Since we already proved that the Taylor coefficients of the densities in \eqref{eq:limdensities2} satisfy these recursive relations, we deduce that no other solution is allowed. Note that the value of $f_{+,0}$ is uniquely determined by the fact that $f_+(\infty,u)+f_-(\infty,u)$ is a density. This concludes the proof of case (i).

Consider now case (ii). Suppose that the densities $f_+(\infty,\cdot)$ and $f_-(\infty,\cdot)$ are as in \eqref{eq:limdensities3}-\eqref{eq:limdensities4}, respectively. Note that in this case $\alpha(\infty;u,\vec{v})=1$ and $\alpha(\infty;u,\vec{v})=0$, respectively. The claim then follows by checking that \eqref{eq:mod2asymp} is satisfied by these densities and not satisfied by analytic functions but $f_+(\infty,\cdot)=f_-(\infty,\cdot)=0$, which however are not admissible densities. This can be done by using a similar argument as in case (i), because now the densities in \eqref{eq:limdensities2} coincide with those in \eqref{eq:limdensities3}-\eqref{eq:limdensities4} in the limiting cases $p_+=1$ and $p_+=0$, respectively. This concludes the proof.


\subsection{Nonlinear version of the second model}
\label{sec:nonlin}

In this section, inspired by the nonlinear version of the voter model in \cite{MM19}, \cite{LHAJS20}, \cite{M23}, \cite{JTSZH20}, we introduce a generalisation of the co-evolutionary model defined above that allows polarisation effects to appear. 
In this model the \emph{vertex dynamics} are the same as above, but we consider generalised \emph{edge dynamics}. We suppose that each edge is resampled at rate 1. (i)~If both adjacent vertices hold opinion $+$, then the edge is active with probability $\pi_+$; (ii)~if both adjacent vertices hold opinion $-$, then the edge is active with probability $\pi_-$; and (iii)~\emph{if one adjacent vertex holds opinion $+$ and the other opinion $-$, then the edge is active with probability $\left[(\pi_++\pi_-)/2\right]^q$, where $q>0$.} Observe that if $q=1$ (i.e., the linear case) then the model is the same as above. We therefore refer to this model as the {\it nonlinear version of the second model}.

The idea behind introducing the parameter $q$ is to enable the model to display polarisation, in that the system becomes increasingly polarised as $q$ grows. We refer to polarisation as the propensity of the population to split into two distinct communities connected by a small number of edges. Observe that as $q$ increases the probability to have agreeing edges does not change, whereas the probability to have disagreeing edges decreases. Consequently, increasing $q$ results in an increasingly more polarised network. This is illustrated in Fig.\ \ref{fig-mod2mod} which displays simulated outcomes of empirical graphons for $T=1,2,3$ when $q=12$.

The problem with considering this generalisation of the model is that establishing a functional law of large numbers in the space of graphons becomes infeasible. To explain this, consider the probability that edge $ij$ is active at time $p_{ij}(t)$. Following a reasoning similar to the one used in the derivation of \eqref{eqn:Hdef}, we find that
\begin{align*}
p_{ij}(t) &= \eee^{-t} p_0 + \int_{0}^t {\rm d}s \, \eee^{-s} \Bigg[ \pi_+ \mathbf{1}\{x_i(t-s)=+\} \mathbf{1}\{x_j(t-s)=+\}\\
&\qquad\qquad\qquad\qquad\qquad +\pi_- \mathbf{1}\{x_i(t-s)=-\} \mathbf{1}\{x_j(t-s)=-\} \\
&\quad + \left(\frac{\pi_+ + \pi_-}{2}\right)^q \Big( \mathbf{1}\{x_i(t-s)=+\} \mathbf{1}\{x_j(t-s)=-\} + \mathbf{1}\{x_i(t-s)=-\} \mathbf{1}\{x_j(t-s)=+\}\Big) \Bigg].
\end{align*}
When $q > 1$, this expression for $p_{ij}(t)$ no longer simplifies to an equation that can be written only in terms of the types of $i$ and $j$. Consequently, for those $q$ we are unable to define $H$ as required in Steps 2 and 3 of Section \ref{sec:genstrategy}. This issue appears to be fundamental in the sense that it cannot be solved by changing the definition of $y_i(t)$ and $y_j(t)$ (as given in \eqref{eqn:ydef}). Moreover, it is unclear how to guess what the limit in the space of graphons $g^{[F]}$ might be, even without working in the framework of Section \ref{sec:genstrategy}. In the next section we define a model that displays similar polarisation behavior but where our general framework can still be applied.

\begin{figure}
\begin{center}
\includegraphics[width=5cm]{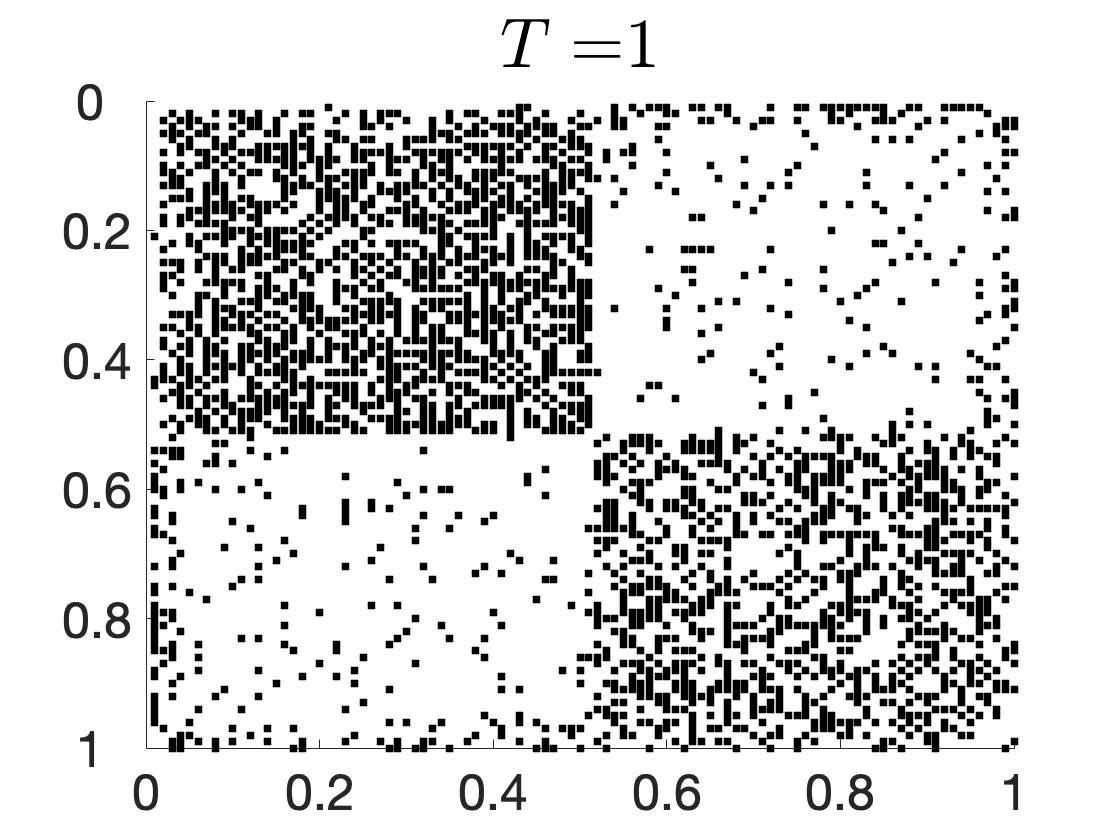}
\includegraphics[width=5cm]{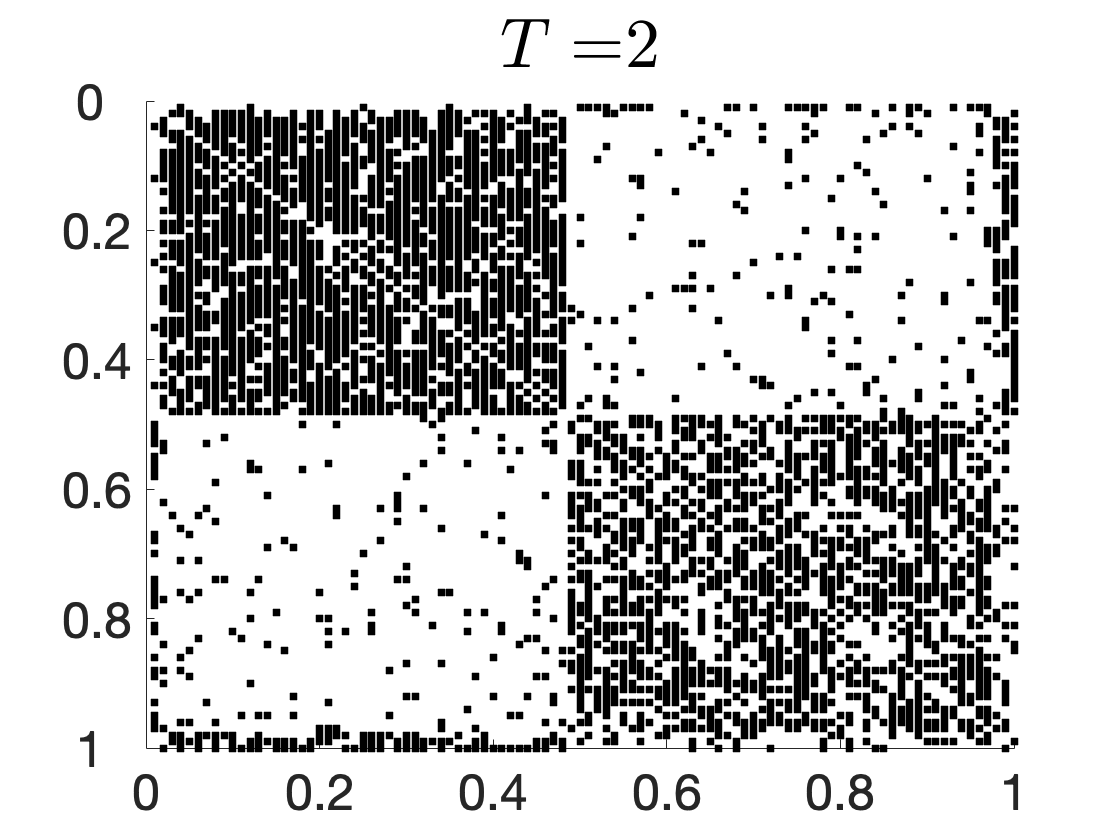}
\includegraphics[width=5cm]{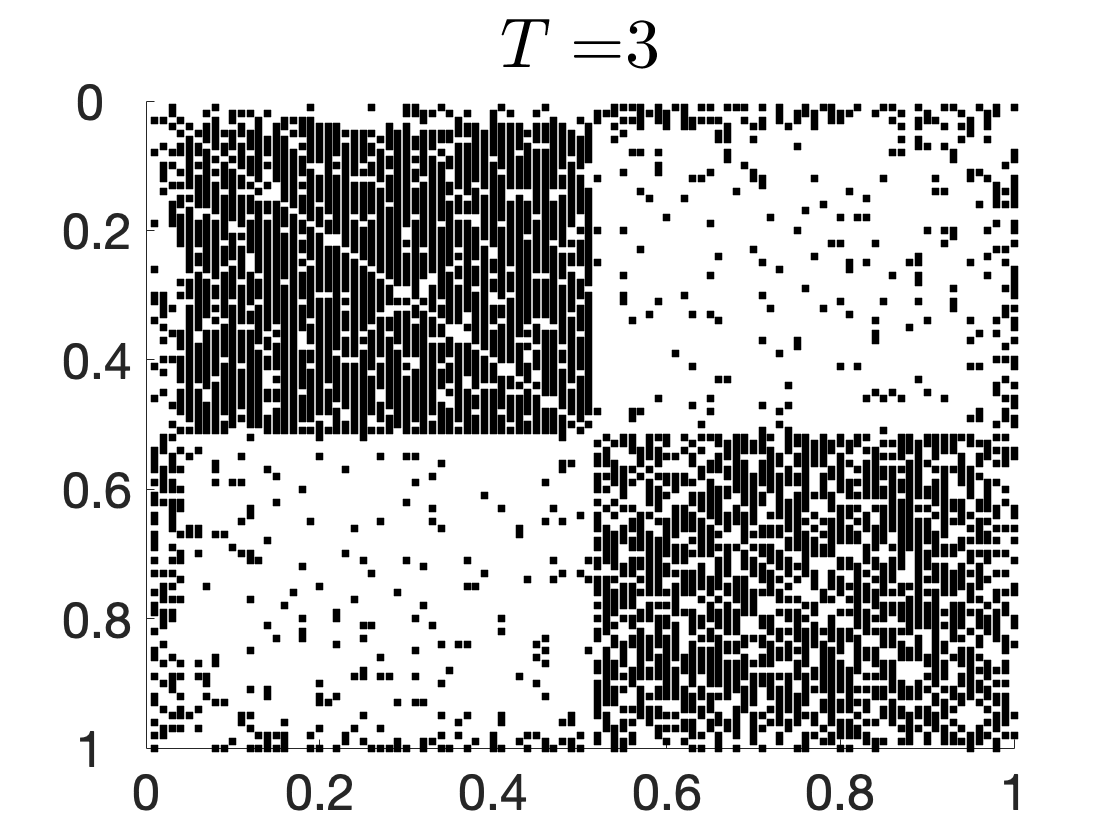}
\caption{\small Empirical graphons for the generalised second model with $n=100$, $p_0=0.05$, $\beta=0.66$, $\pi_+=0.9$, $\pi_-=0.7$ and $q=12$. Initially, half of the vertices hold opinion $+$ and half opinion $-$.}
\label{fig-mod2mod}
\end{center}
\end{figure}


\section{Third model: two-way feedback and polarisation}
\label{sec:twoway2}

While the second model is a co-evolutionary network, it does not capture the key feature of a voter model on a dynamic network we are after: \textit{vertices holding different opinions are less likely to be connected.} Below we provide one solution to this problem.

Consider $2q$ copies of the second model with $q\geq1$. Of these, $q$ are referred to as $g$-graphs and $q$ as $r$-graphs. All these graphs evolve in parallel, and the $g$-graphs and $r$-graphs have different parameters. The opinions of the vertices are the same in all the copies, but the states of the edges are different. The $2q$ graphs are conditionally independent of each other, given the opinions of the vertices. The true graph, called the {\it resulting graph}, is a function of the $g$-graphs and the $r$-graphs. To be precise, the dynamics of the process is as follows.
\begin{itemize}
\item
{\it Vertex dynamics.} Each vertex holds opinion $+$ or $-$, and is assigned an independent rate-$\beta$ Poisson clock. Each time the clock rings the vertex selects one of its neighbours \textit{in the resulting graph} uniformly at random and copies the opinion of that vertex.
\item
{\it Edge dynamics.} The edge dynamics for $g$-graphs, $r$-graphs and the resulting graph is as follows.

\textbf{$\blacktriangleright$ $g$-graphs and $r$-graphs:} 
For $k\in\{g, r\}$, in every $k$-graph each edge is re-sampled at rate $1$, i.e., attach a rate-$1$ Poisson clock to each edge, and each time the clock rings the edge is active with a probability that depends on the current opinion of the two connected vertices: with probability $\pi_+^k$ if the two vertices hold opinion $+$, with probability $\pi_-^k$ if the two vertices hold opinion $-$, and with probability $\tfrac12(\pi_+^k+\pi_-^k)$ otherwise.
\item[]
\textbf{$\blacktriangleright$ resulting graph:} The \emph{resulting graph} is a function of the $g$-graphs and the $r$-graphs:
\begin{itemize}
\item Edge $ij$ is active in the resulting graph if edge $ij$ is \emph{active in exactly all} the copies of the $k$-graph for one $k\in\{g, r\}$.
\item Edge $ij$ is inactive in the resulting graph in all the other cases.
\end{itemize}
\end{itemize}
The dynamics can be interpreted as follows. An active edge $ij$ in all the $g$-graphs can be interpreted as an individual (either $i$ or $j$) sharing a positive thought about opinion $+$ with the other individual, and an active edge $ij$ in all the $r$-graphs can be interpreted as an individual (either $i$ or $j$) sharing a positive thought about opinion $-$ with the other individual. If $i$ and $j$ share positive thoughts about only one opinion, then they are friends and are therefore connected in the resulting graph, whereas if they share positive thoughts about neither opinion (they do not contact each other) or both opinions (they argue), then they are not friends and are therefore not connected in the resulting graph.

As an illustration, suppose that $\pi_-^r=\pi_+^g=1$ and $\pi_+^r=\pi_-^g=0$. If vertices $i$ and $j$ have both held the opinion $+$ for an infinitely long period of time, then the probability that edge $ij$ is active is~$1$ (the same applies if they have held opinion $-$), whereas if vertex $i$ has held opinion $+$ and vertex~$j$ has held opinion $-$ for an infinitely long period of time, then the probability edge $ij$ is active is~$2\cdot0.5^q$. Consequently, vertices that share the same opinion (generally) have a higher probability of being connected.

By considering the number of copies as a function of the parameter $q$, we make it less likely that individuals holding different opinion are connected. As a result, we obtain different levels of polarisation depending on the value of $q$. In the limit $q\to\infty$, we expect to obtain {\it strong polarisation}, i.e., there are two distinct communities holding a different opinion without active edges in-between them.

Section \ref{sec:prep3} sets notation and steps through the framework in Section \ref{sec:genstrategy}, Section~\ref{mod3:mr} states the main results, Section~\ref{mod3:num} offers simulations, Section~\ref{mod3:pr} provides proofs.

\subsection{Preparations}\label{sec:prep3}

\medskip \noindent
\emph{\underline{Step 1:}} Let the type of vertex $i$ at time $t$ again be given by \eqref{eqn:typedef}, and let
\begin{equation}
\label{eq:alphatilde}
\widetilde\alpha(t;u,\vec{v}(t,\cdot)) = \frac{\int_0^1 {\rm d}y\,f_+(t,y) H_R(t;y,u)}{\int_0^1 {\rm d}y\, [f_+(t,y) + f_-(t,y)]\,H_R(t;y,u)},
\end{equation}
where $H_R(t;\cdot,\cdot)$ is defined in \eqref{eqn:Hthird} below, and $f_+(t,\cdot)$ and $f_-(t,y)$ are the (unique) solution to \eqref{eq:PDEsalt2} below. In what follows we simply write $\widetilde\alpha(t;u,\vec{v})$, instead of  $\widetilde\alpha(t;u,\vec{v}(t,\cdot))$, so as  to lighten the notation. Similar to what we have been doing in \eqref{eq:gen2}, we can define a mimicking process with one-way dependence whose generalised vertex types are independent Markov processes that are characterised via the generator
\begin{equation}
	\label{eq:gen3}
	({\cal L}_t f)(x,y) = \beta((1-\widetilde\alpha(t;y,\vec{v}))\mathbf{1}\{x=+\}+\widetilde\alpha(t;y,\vec{v})\mathbf{1}\{x=-\})[f(x',y)-f(x,y)]
	+ b(x,y) \dfrac{\partial}{\partial y} f(x,y),
\end{equation}
where $x'$ is the opinion of the selected vertex after switching its opinion $x$ at time $t$, and $b(x,\cdot)$ is the drift term when starting from $x\in\{-,+\}$, which has the same shape as for models 1 and 2 (see \eqref{eq:b+}-\eqref{eq:b-} below). 
It leads to the Kolmogorov forward equations
\begin{equation}
\label{eq:PDEsalt2}
\begin{aligned}
\frac{\partial}{\partial t} f_+(t,u) + (1-u) \frac{\partial}{\partial u} f_+(t,u)  
&= \beta f_-(t,u)\,\widetilde\alpha(t;u,\vec{v}) -(\beta(1-\widetilde\alpha(t;u,\vec{v}))-1) f_+(t,u),\\
\frac{\partial}{\partial t} f_-(t,u) - u \frac{\partial}{\partial u} f_-(t,u) 
&= \beta f_+(t,u)\,(1-\widetilde\alpha(t;u,\vec{v})) - (\beta \widetilde\alpha(t;u,\vec{v})-1)f_-(t,u),
\end{aligned}
\end{equation}
which characterise $f_-$ and $f_+$, and therefore $F$.  

\medskip \noindent
\emph{\underline{Steps 2 and 3:}} By the same arguments that led to \eqref{eqn:Hdef}, the probability that an edge between vertices of type $u$ and $v$ is present in a $g$-graph is 
\[
H_g(t;u,v) = p_0 {\rm e}^{-t}+ \tfrac{1}{2} \left[\pi_+^g u 
+  \pi_-^g(1-{\rm e}^{-t} - u) +  \pi_+^g v + \pi_-^g(1-{\rm e}^{-t} - v) \right],
\]
while the probability that an edge between vertices of type $u$ and $v$ is present in a $r$-graph is
\[
H_r(t;u,v) = p_0 {\rm e}^{-t}+ \tfrac{1}{2} \left[\pi_+^r u 
+  \pi_-^r(1-{\rm e}^{-t} - u) +  \pi_+^r v + \pi_-^r(1-{\rm e}^{-t} - v) \right].
\]
Due to conditional independence between the $g$-graph and the $r$-graph (i.e., conditional on the types of the vertices), the probability that there is an edge in the resulting graph between vertices of type $u$ and $v$ is
\begin{equation}
\label{eqn:Hthird}
H_R(t;u,v) = H_g(t;u,v)^q(1-H_r(t;u,v))^q + H_r(t;u,v)^q(1-H_g(t;u,v))^q.
\end{equation}

\medskip \noindent
\emph{\underline{Step 4:}} 
Given that is $F$ characterised by \eqref{eq:PDEsalt2}, $H_R$ is characterised by \eqref{eqn:Hthird}, and recalling \eqref{eqn:invdef}, we obtain the candidate limit
\begin{equation}
\label{eqn:empgrathird}
g^{[F]}(t;x,y) = H_R(t; \bar F(t;x), \bar F(t; y))
\end{equation}
for the empirical graphon of the mimicking process. Similar to how this was done in Section~\ref{sec:twoway1}, we are then to couple the mimicking process and the co-evolutionary model to establish the results below.


\subsection{Main results: Theorems~\ref{thm:graphonconvthird}--\ref{thm:systempdeasymp3}}
\label{mod3:mr}

The main theorems for the third model are the following.

\begin{theorem}
\label{thm:graphonconvthird}
$h^{G_n} \Rightarrow g^{[F]}$ as $n\to\infty$ in the space $D((\mathcal{W},d_{\square}),[0,T])$, where $h^{G_n}$ is the empirical graphon associated with $G_n$ defined in \eqref{eq:graphon}.
\end{theorem}

Recall the definition of the vector of densities $\vec{v}(t,u)$ given in \eqref{eq.vecf}. In line with what we did for model 2, in the next theorem we state the (local) existence of the densities $f_+(t,u)$ and $f_-(t,u)$, which satisfy the Kolmogorov forward equations for the mimicking process defined in Section \ref{sec:mim2} below. 

\begin{theorem}
\label{thm:systempdethird}
For every $t\in[0,T]$ and $u\in[0,1]$, there exists a unique analytic vector of densities $\vec{v}(t,u)$, i.e., for any $u\in(0,1)$ there exists an open set $O_u$ containing $u$ such that there is a unique analytic vector $\vec{v}(t,u)$ on $O_u$ satisfying $\vec{v}(0,u)=(f_+(0,u),f_-(0,u))^\top$ for any $u\in O_u$.
\end{theorem}

\begin{theorem}
\label{thm:systempdeasymp3}
For any $u\in[0,1]$, the following statements hold.
\begin{itemize}
	\item[(i)] Let $p_+$ be as in \eqref{eq:limprop}.
	If $\pi^g_+=\pi^g_-$ and $\pi^r_+=\pi^r_-$, then
	\begin{equation}\label{eq:limdensities3bis}
		\lim_{t\to\infty}\vec{v}(t,u) := \left(\begin{matrix} f_+(\infty,u) \\ f_-(\infty,u)\end{matrix}\right)
		= \left(\begin{matrix} u  \\ 1-u \end{matrix}\right) f_{\beta p_+, \beta (1-p_+)} (u),
	\end{equation}
	where $f_{\beta p_+, \beta (1-p_+)}$ is the density corresponding to a ${\rm Beta}(\beta p_+,\beta (1-p_+))$ random variable, i.e.,
	\[
	f_{\beta p_+, \beta (1-p_+)} (u) = u^{\beta p_+-1} (1-u)^{\beta (1-p_+)-1}\mathbf{1}\{u\in[0,1]\}.
	\]
	\item[(ii)] If both
	\begin{itemize}
		\item[(A)] $\pi^g_+\neq\pi^g_-$ or $\pi^r_+\neq\pi^r_-$
		\item[(B)] $\pi^g_+\neq\pi^r_-$ or $\pi^r_+\neq\pi^g_-$
	\end{itemize} 
    hold, then only consensus is admissible, i.e., either
	\begin{equation}\label{eq:mod3cons1}
		\lim_{t\to\infty}\vec{v}(t,u) := \left(\begin{matrix} f_+(\infty,u) \\ f_-(\infty,u)\end{matrix}\right)
		= \left(\begin{matrix} \delta_1(u)  \\ 0 \end{matrix}\right),
	\end{equation}
	or
	\begin{equation}\label{eq:mod3cons2}
		\lim_{t\to\infty}\vec{v}(t,u) := \left(\begin{matrix} f_+(\infty,u) \\ f_-(\infty,u)\end{matrix}\right)
		= \left(\begin{matrix} 0  \\ \delta_0(u) \end{matrix}\right),
	\end{equation}
	where $\delta_{x_0}(\cdot)$ is the point--mass distribution at $x_0$. 
	\end{itemize}
	\end{theorem}
Theorems \ref{thm:graphonconvthird}-\ref{thm:systempdeasymp3} are the analogues of Theorems \ref{thm:graphonconvalt}-\ref{thm:systempdeasymp2}, and show once again the challenges we are facing in the co-evolutionary setting. Our result provides insight into the structure of the evolution for large $t$, even though it remains a major challenge to give a complete picture of the limiting behaviour, as shown by the following conjecture.

\begin{conjecture}\label{conjecture}
	For any $u\in[0,1]$, the following statements hold for choices of the parameters not covered by Theorem \ref{thm:systempdeasymp3}\it{(i)}.
	\begin{itemize}
	\item[(i)] If $\pi^g_+=\pi^r_-$ and $\pi^r_+=\pi^g_-$, then for any $q\in \mathbb{N}$ there exist unique limiting densities $f_+$ and $f_-$ corresponding to consensus in \eqref{eq:mod3cons1}-\eqref{eq:mod3cons2}, or there are different levels of polarisation depending on the value of $q$ (see Figures \ref{fig1-mod3}-\ref{fig3-mod3}), i.e., 
	\[
	\int_0^1 \dd u\, f_+(\infty,u) = \int_0^1 \dd u \,f_-(\infty,u) = \dfrac{1}{2},
	\]
	and the density of edges connecting vertices having different opinions is a non-zero monotone decreasing function in $q$.
	\item[(ii)] If $\pi^g_+=\pi^r_-$ and $\pi^r_+=\pi^g_-$, then in the limit $q\to\infty$ only consensus occurs as in \eqref{eq:mod3cons1}-\eqref{eq:mod3cons2}, or strong polarisation is admissible, i.e.,
	\[
\lim_{q\to\infty}\lim_{t\to\infty}\vec{v}(t,u)
	= \dfrac{1}{2} \left(\begin{matrix} \delta_1(u)  \\ \delta_0(u) \end{matrix}\right),
	\]
	and the density of edges connecting vertices having different opinions is zero. Moreover, if $\pi^g_+,\pi^r_-,\pi^r_+,\pi^g_-\in(0,1)$, then also the density of edges connecting vertices having the same opinion tends to zero as $q\to\infty$, but slower than the density of disagreeing edges. 
\end{itemize}
\end{conjecture}

\begin{figure}
	\includegraphics[width=5cm]{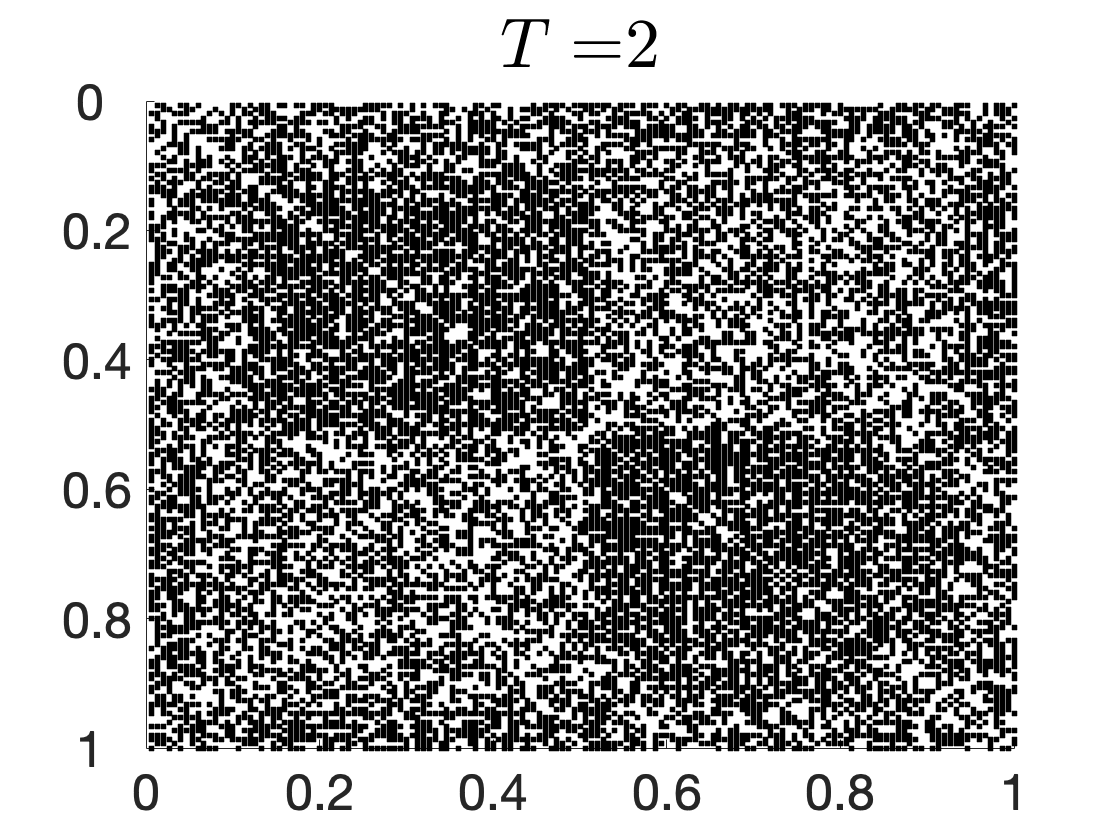}
	\includegraphics[width=5cm]{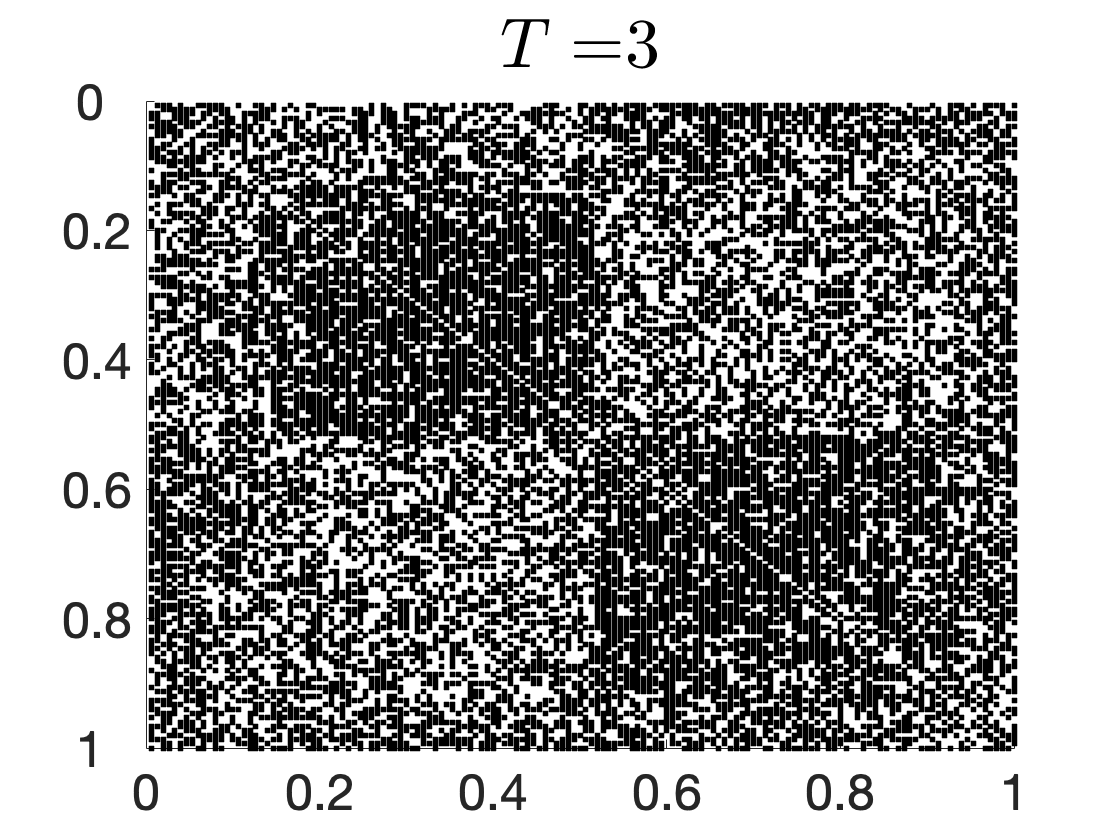}
	\includegraphics[width=5cm]{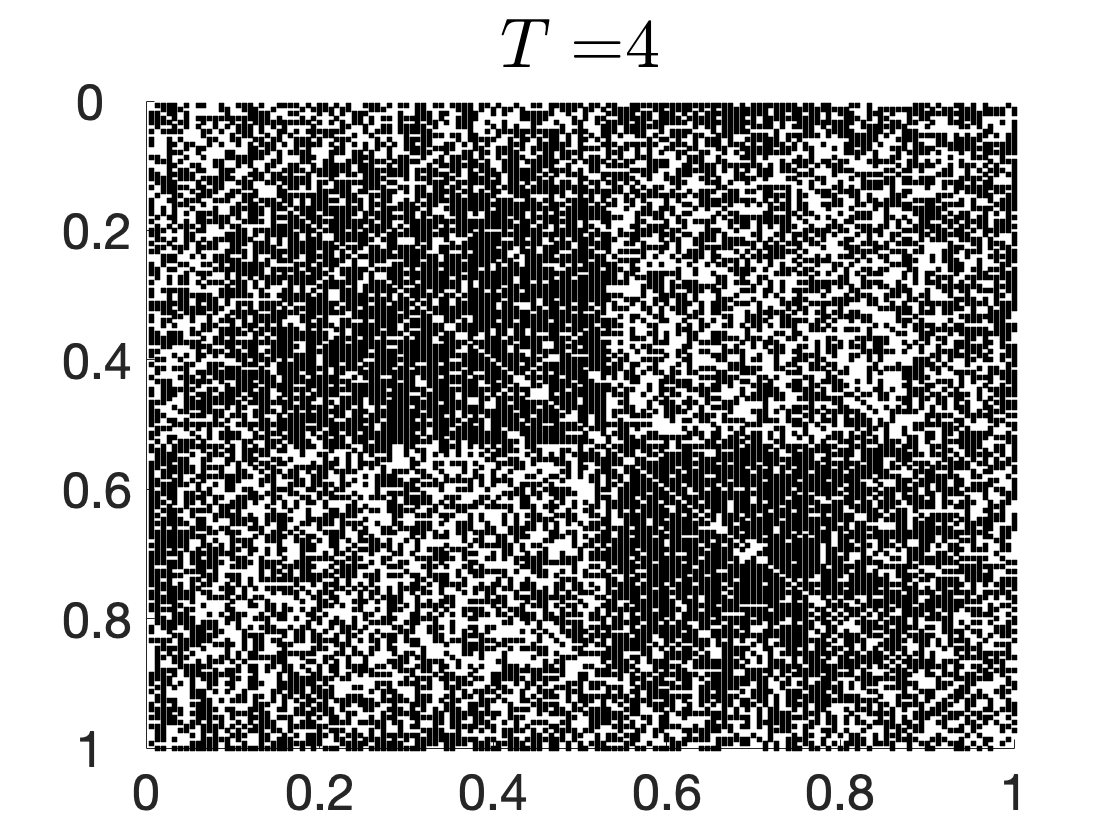}

	\includegraphics[width=5cm]{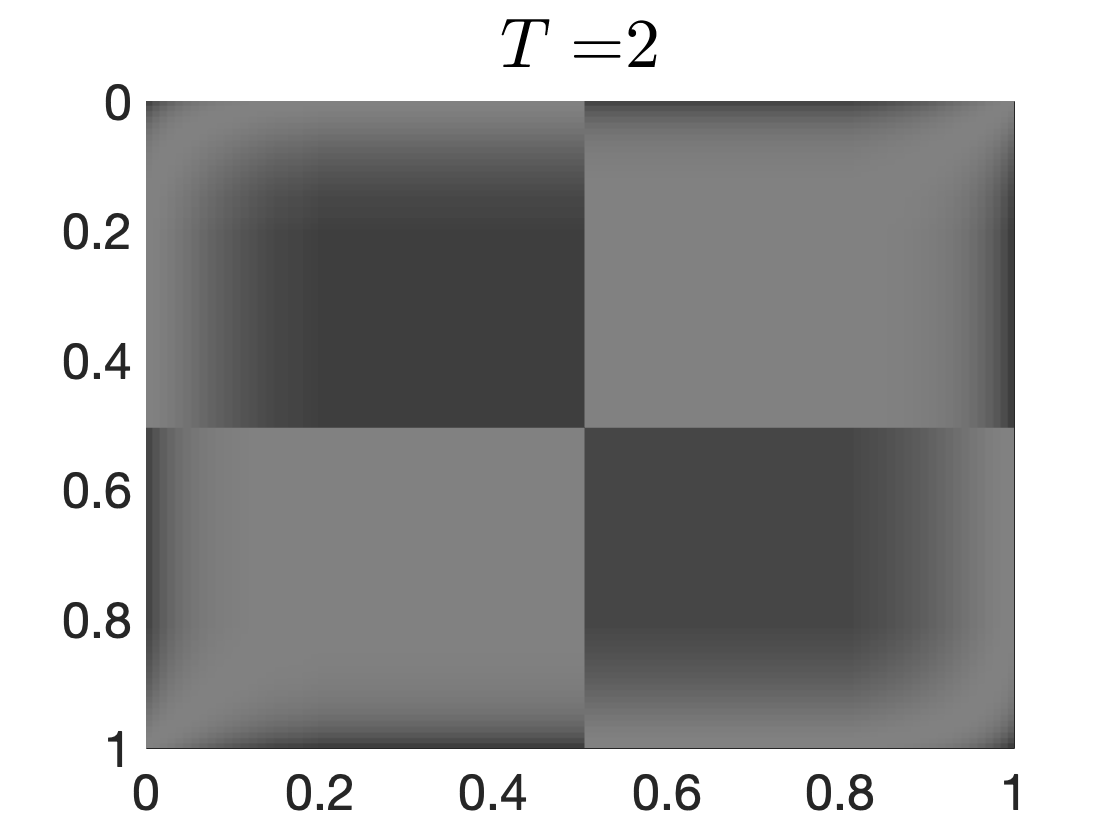}
	\includegraphics[width=5cm]{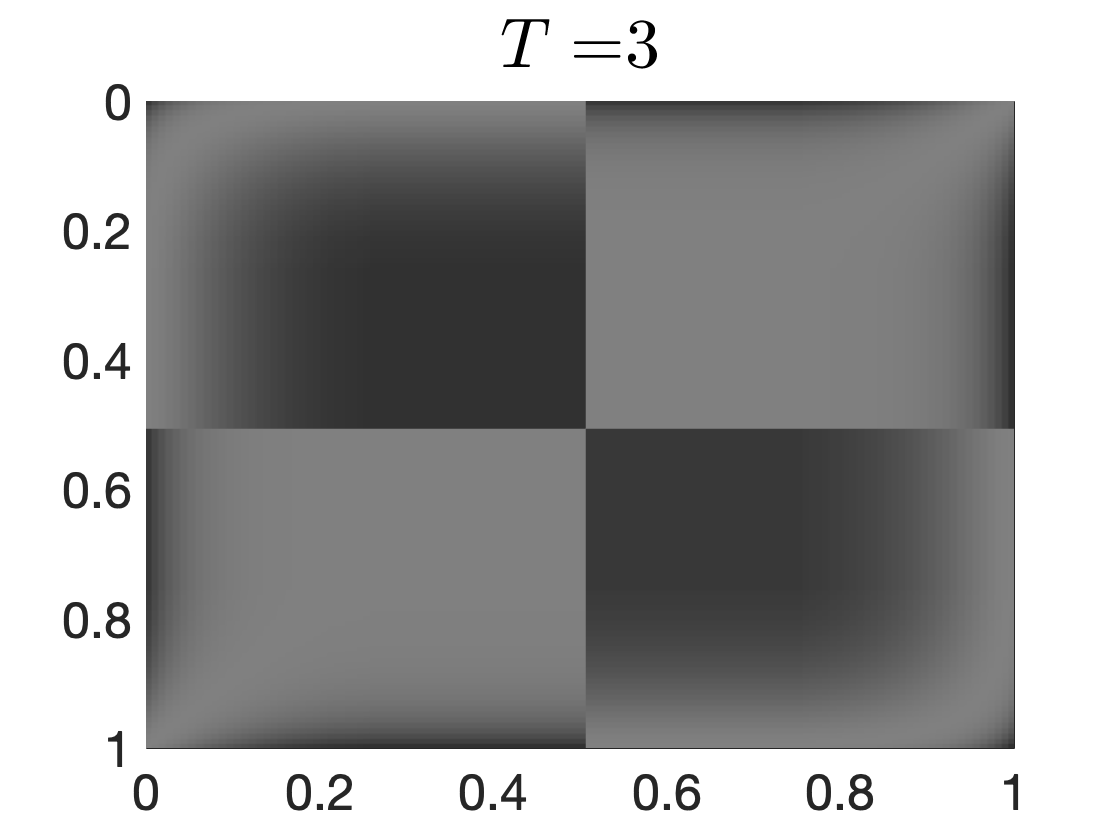}
	\includegraphics[width=5cm]{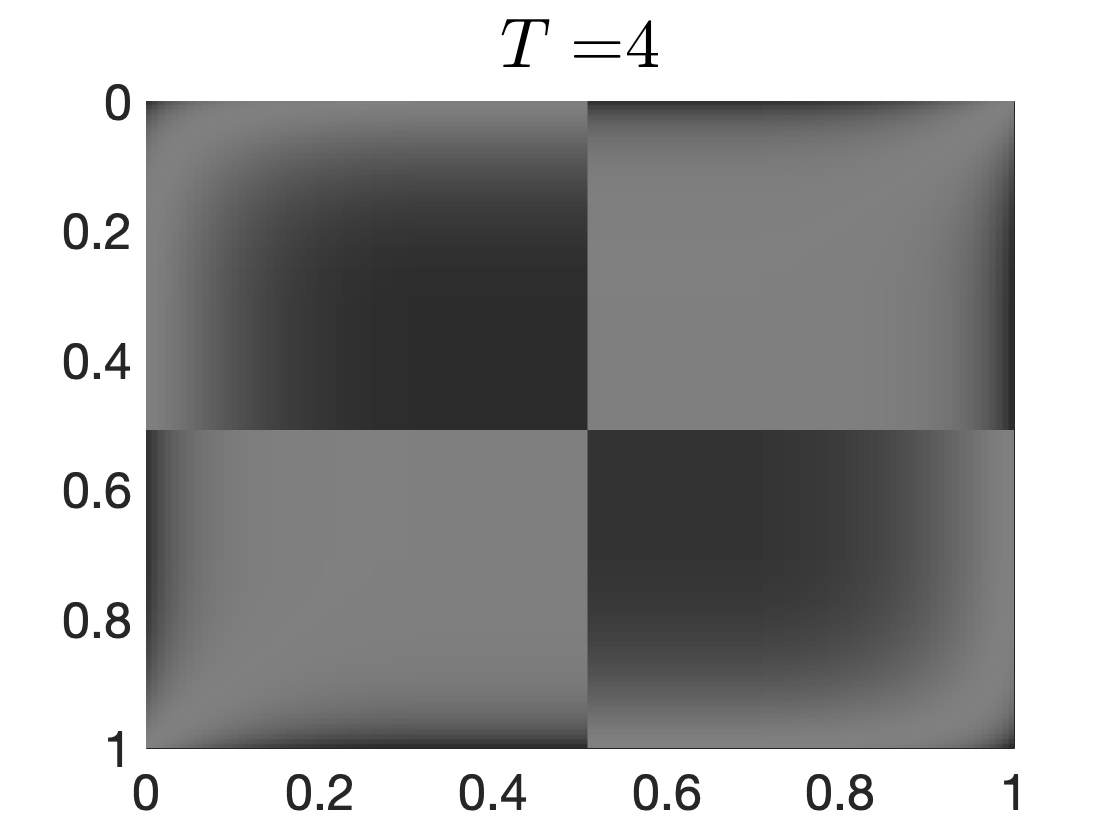}
	\caption{\small The top row displays the empirical graphon when $n=150$, $q=1$ and $T=2,3,4$, the bottom row displays the corresponding functional law of large numbers. Simulations are based on a single run. A dot represents an edge. The labels of the vertices are updated dynamically so that they are ordered lexicographically, i.e., the vertices with opinion $+$ have lower labels than the vertices with opinion $-$, and then by increasing type.}
	\label{fig1-mod3}
\end{figure}

\begin{figure}[h!]
	\includegraphics[width=5cm]{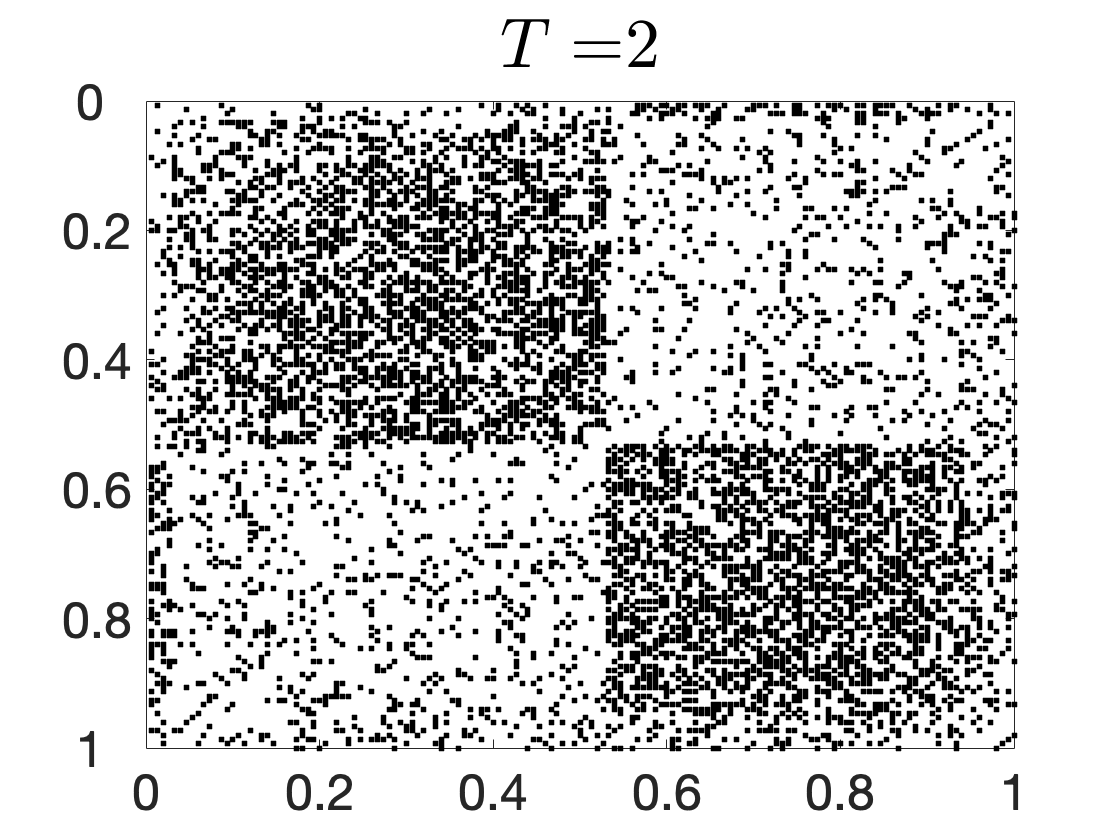}
	\includegraphics[width=5cm]{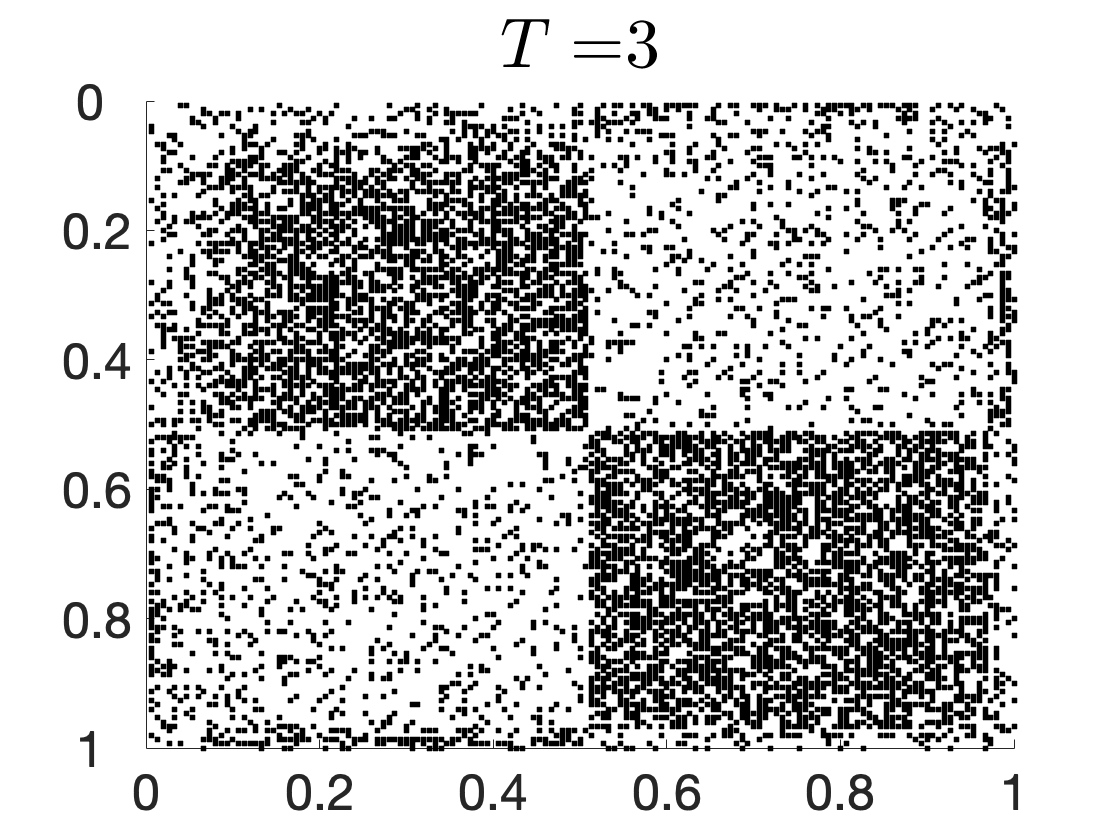}
	\includegraphics[width=5cm]{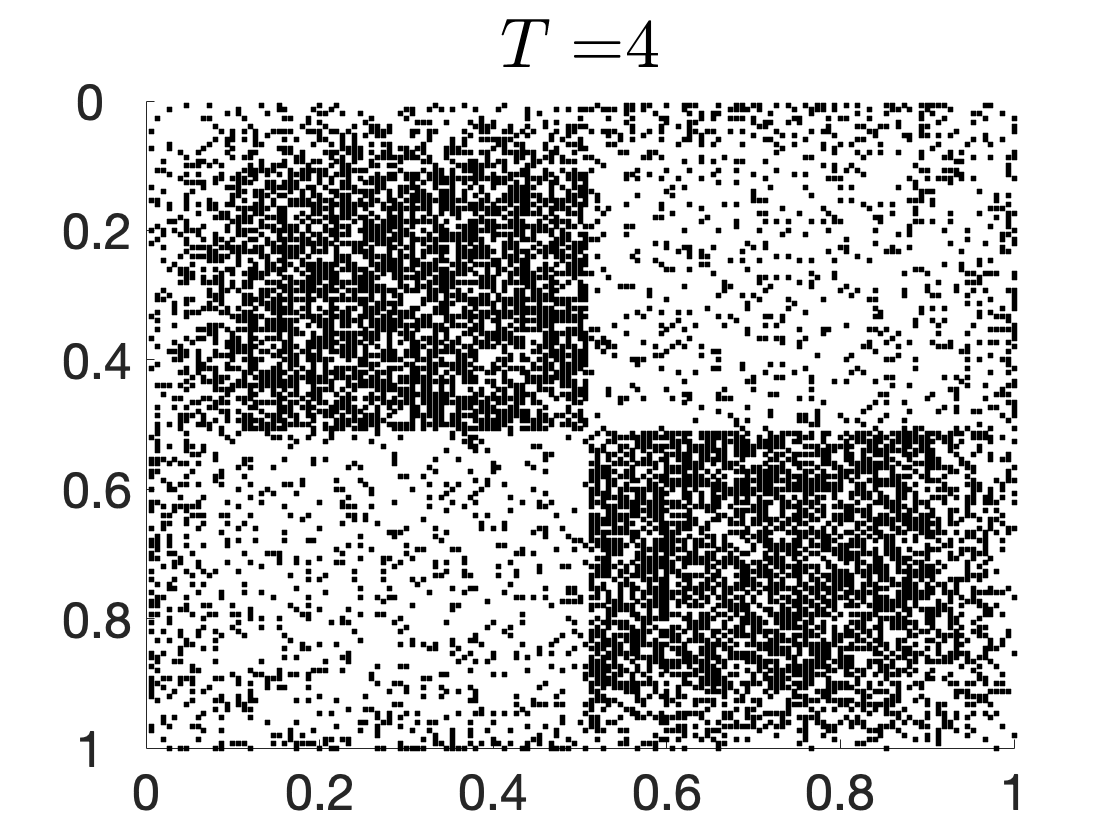}
	
	\includegraphics[width=5cm]{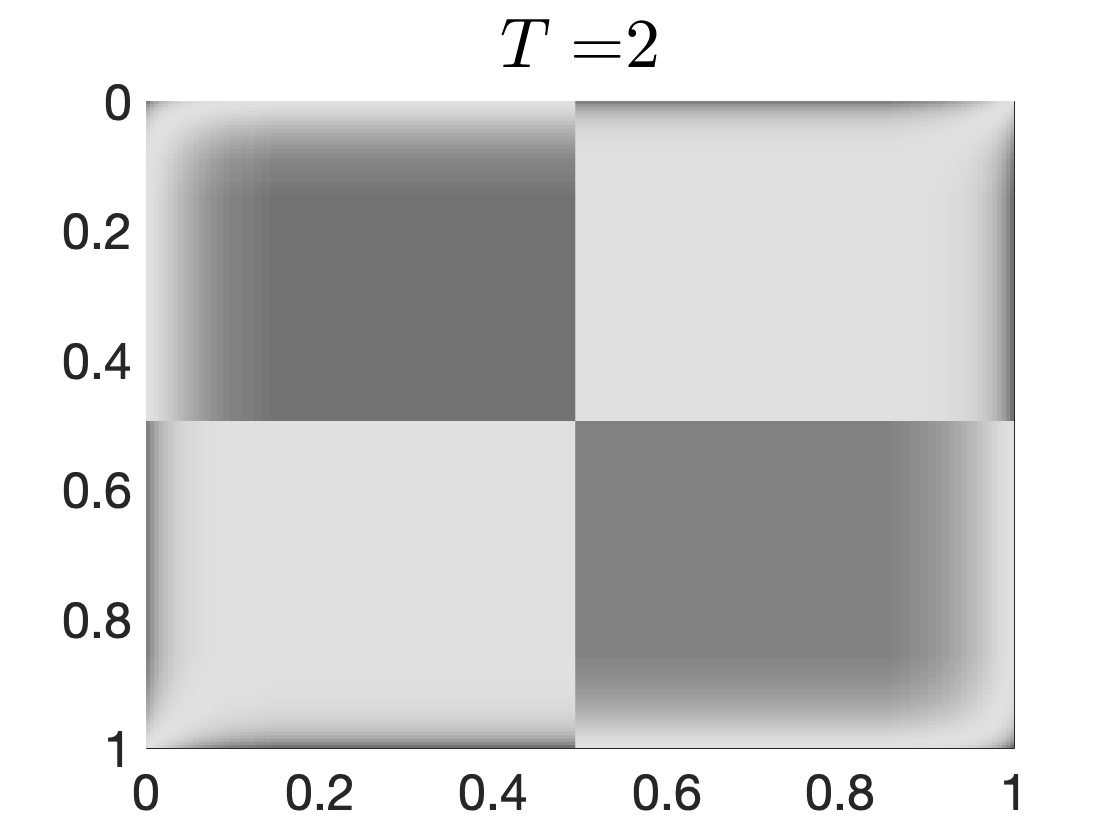}
	\includegraphics[width=5cm]{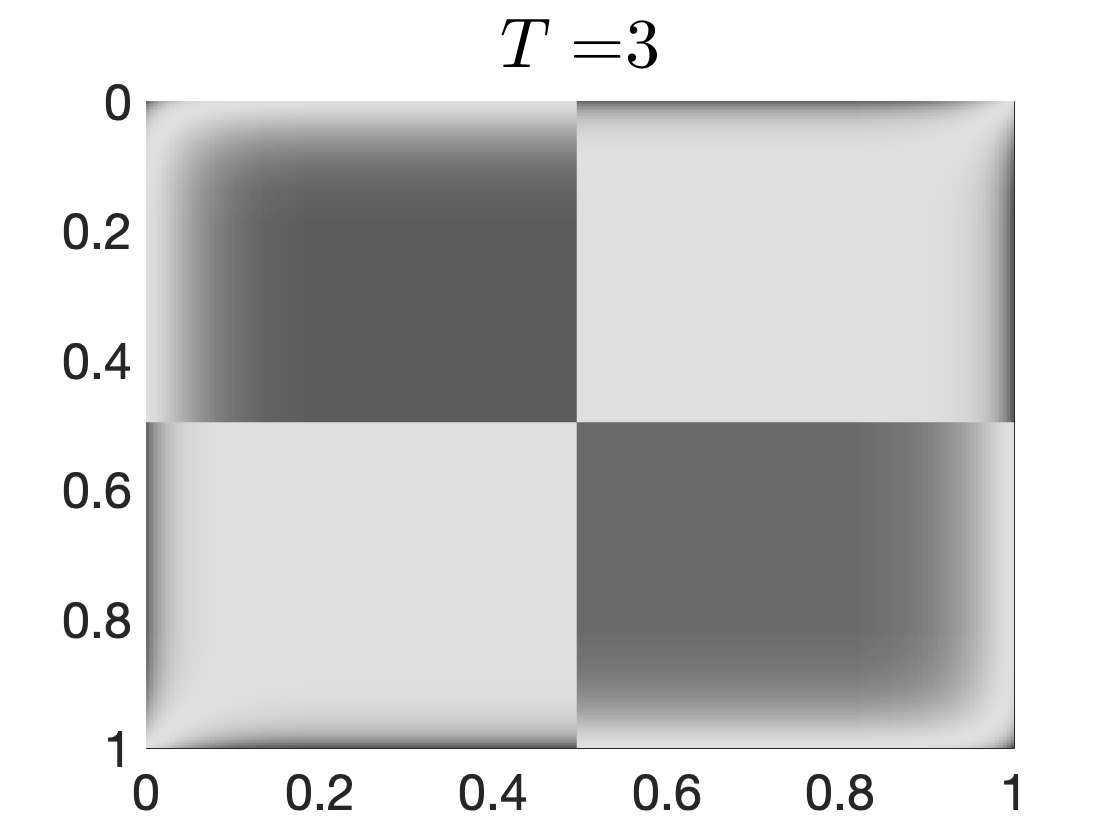}
	\includegraphics[width=5cm]{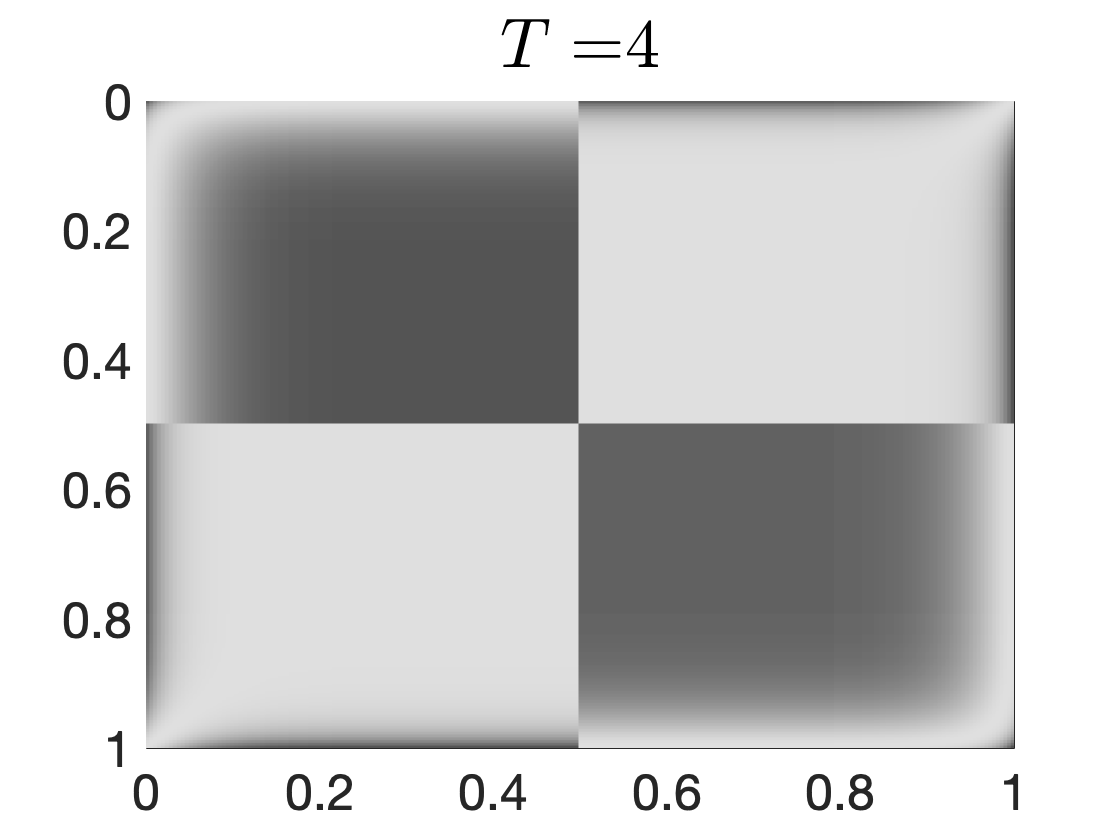}
	\caption{\small The top row displays the empirical graphon when $n=150$, $q=2$ and $T=2,3,4$, the bottom row displays the corresponding functional law of large numbers. Simulations are based on a single run. A dot represents an edge. The labels of the vertices are updated dynamically so that they are ordered lexicographically, i.e., the vertices with opinion $+$ have lower labels than the vertices with opinion $-$, and then by increasing type.}
	\label{fig2-mod3}
\end{figure}

\begin{figure}[htbp]
	\includegraphics[width=5cm]{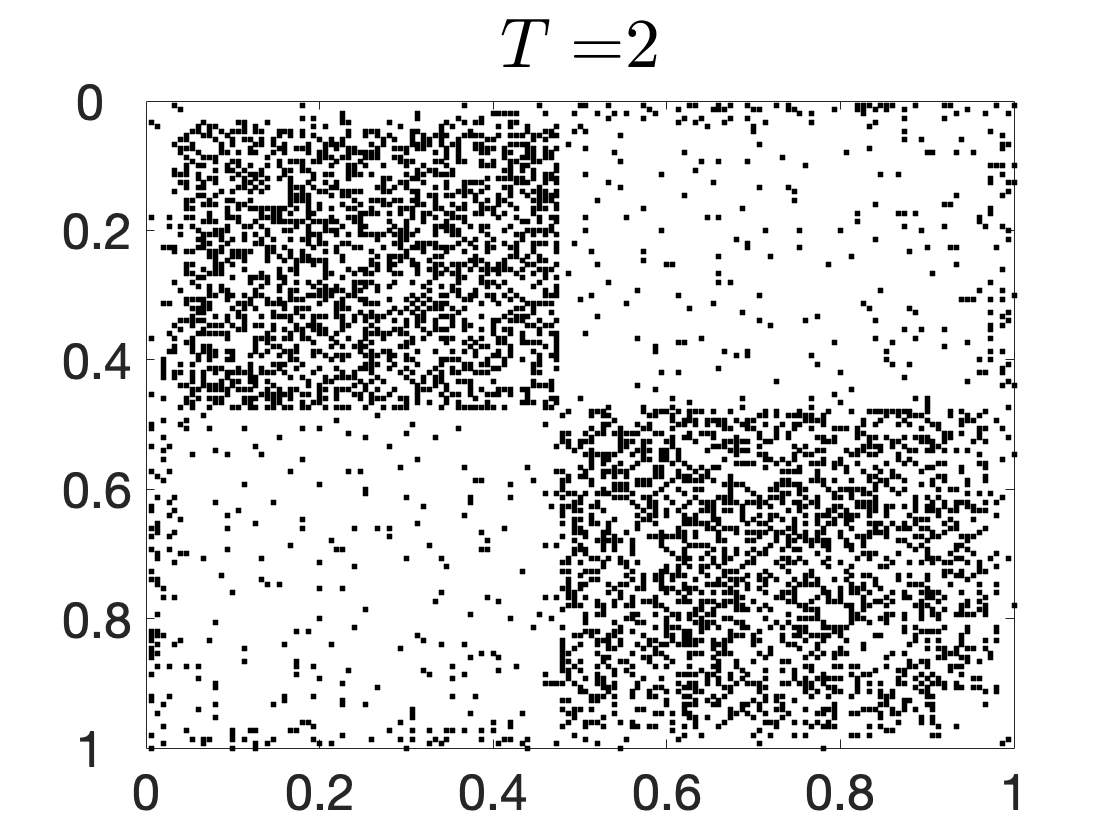}
	\includegraphics[width=5cm]{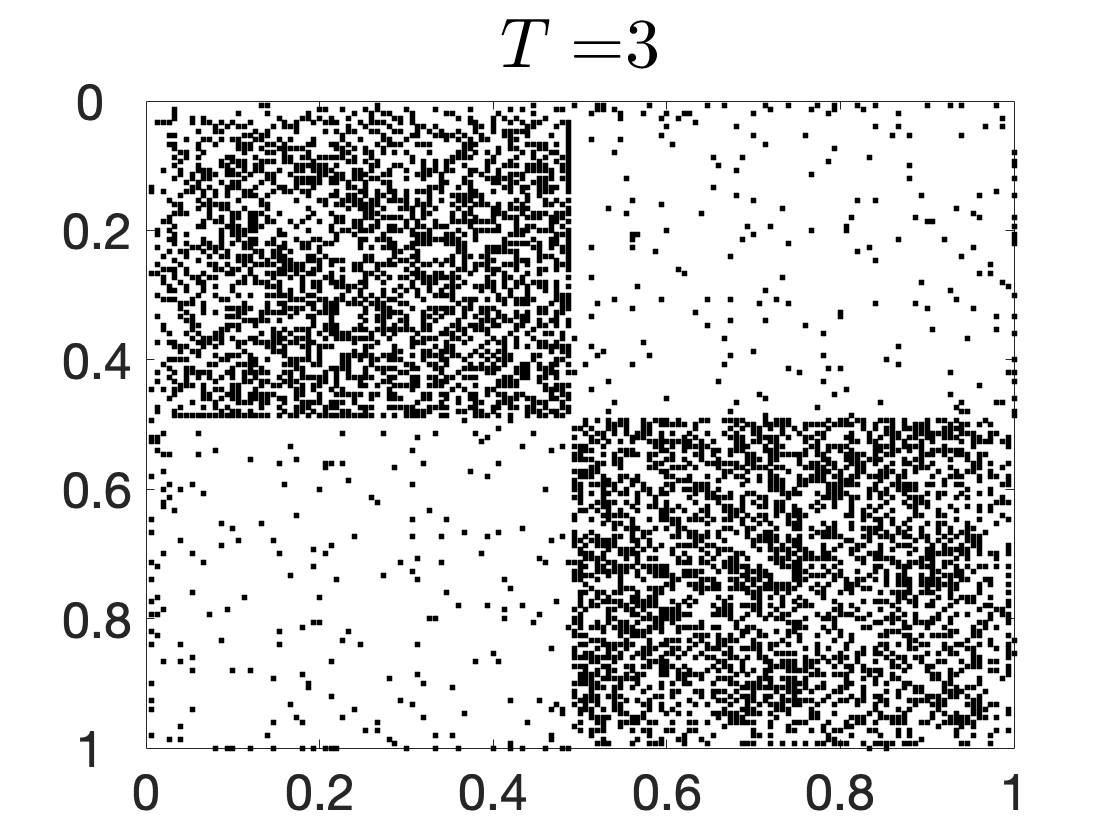}
    \includegraphics[width=5cm]{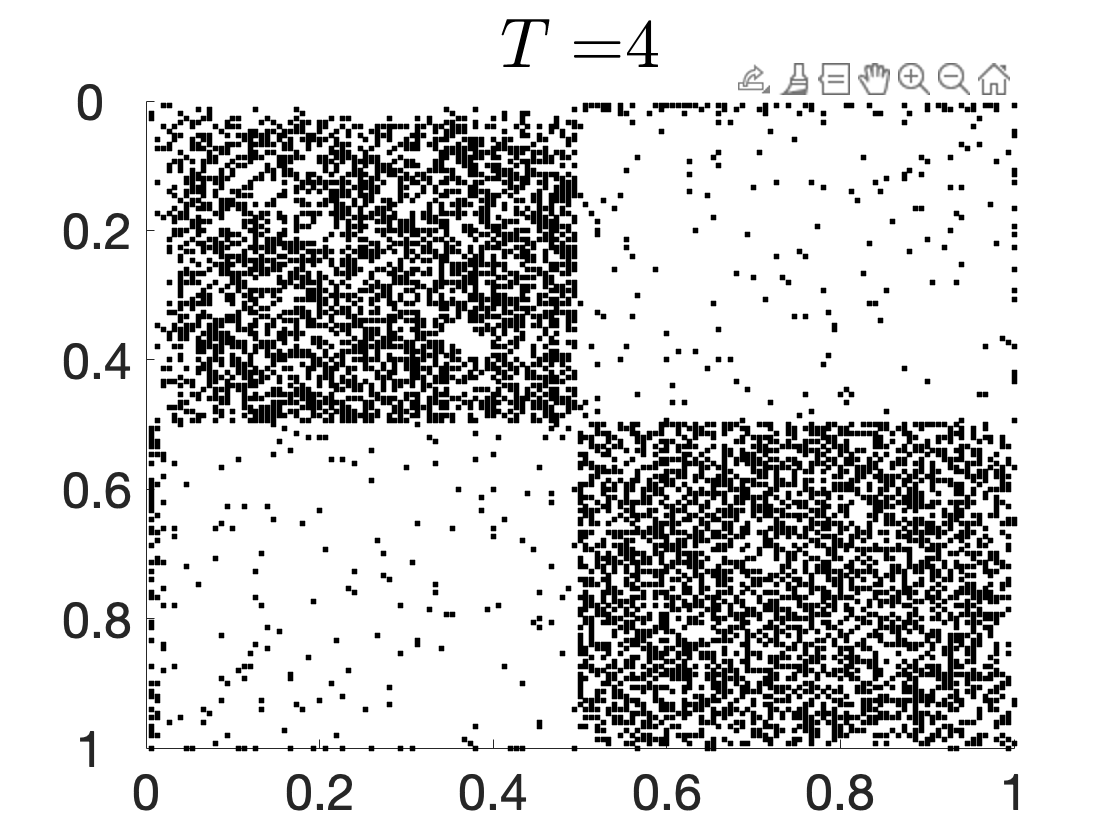}
	
	\includegraphics[width=5cm]{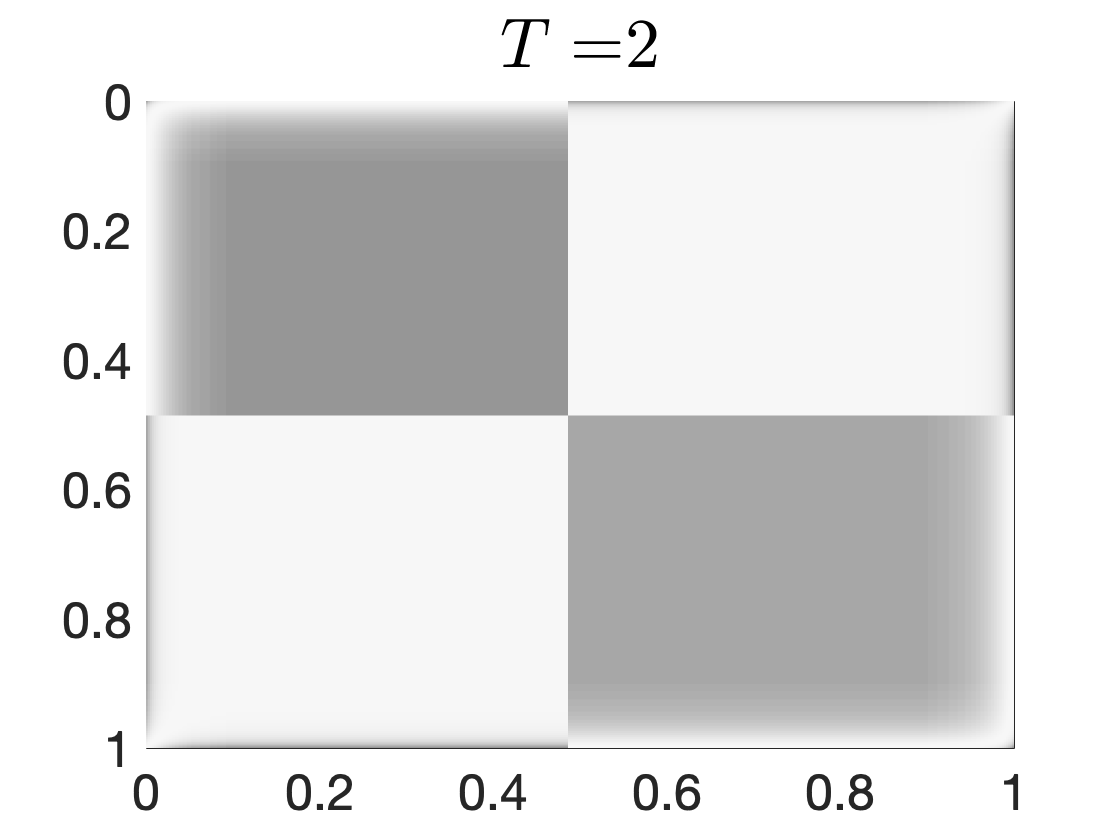}
	\includegraphics[width=5cm]{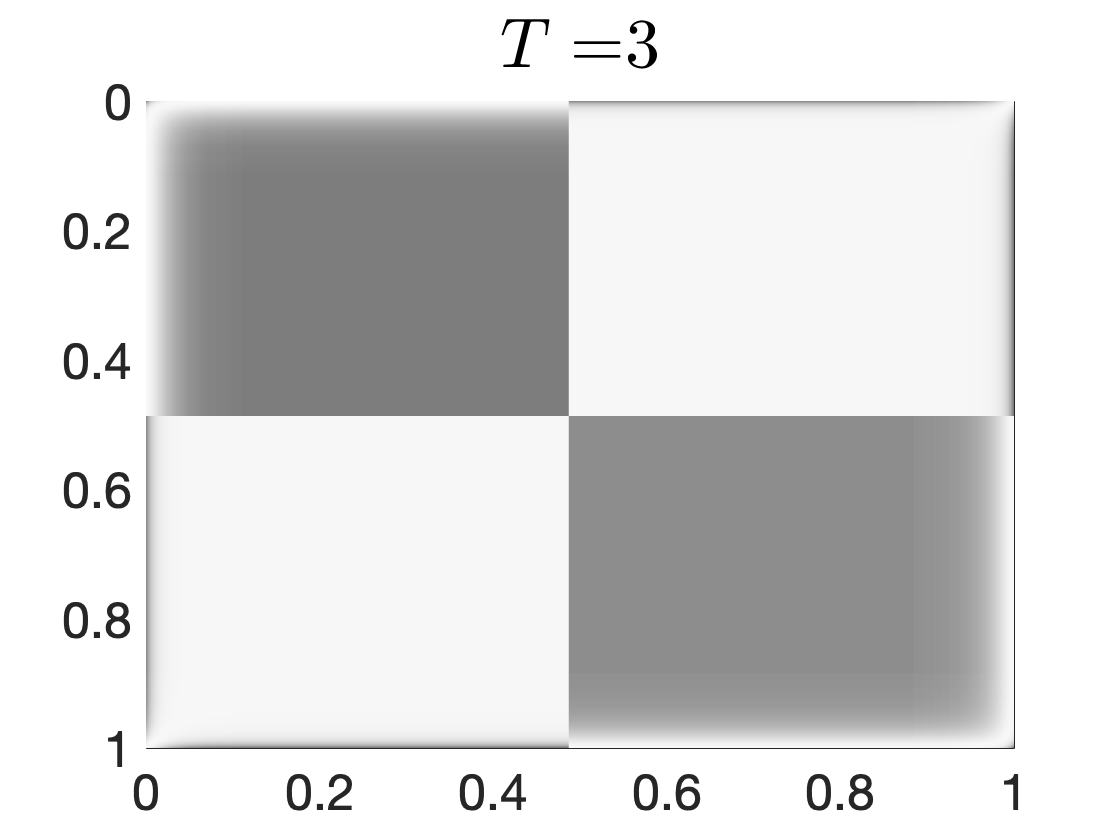}
    \includegraphics[width=5cm]{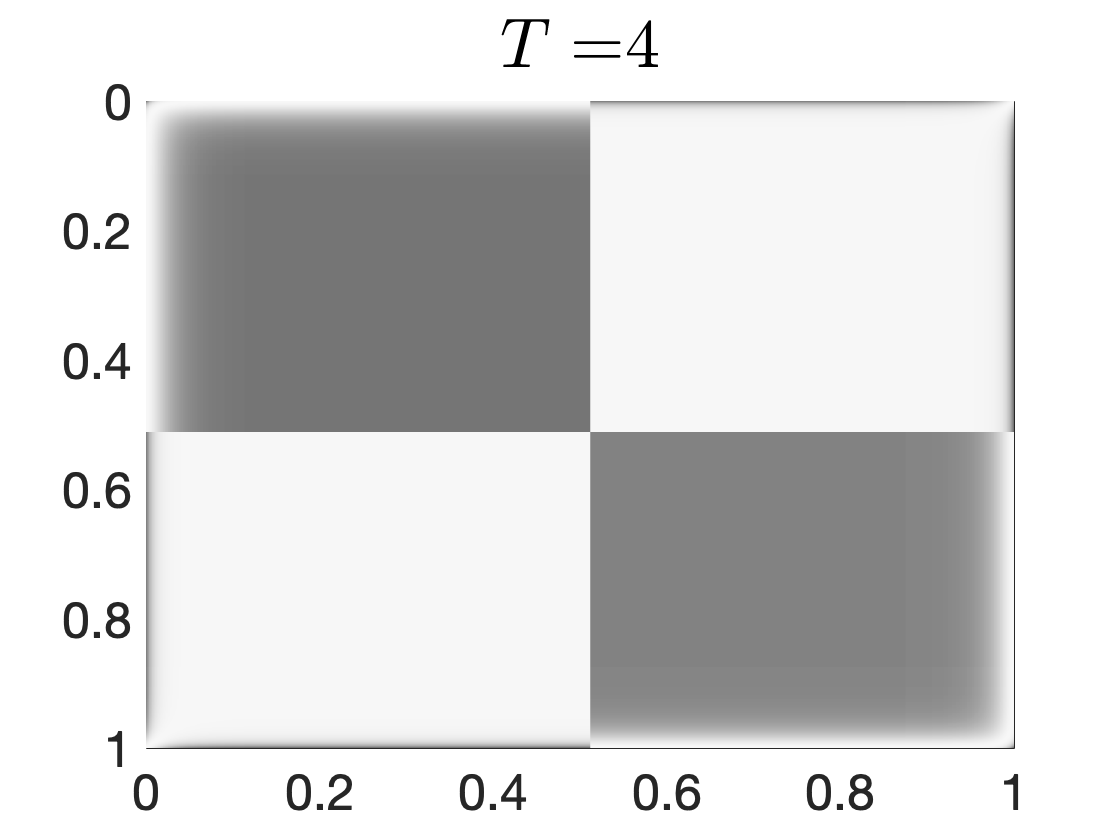}
	\caption{\small The top row displays the empirical graphon when $n=150$, $q=3$ and $T=2,3,4$, the bottom row displays the corresponding functional law of large numbers. Simulations are based on a single run. A dot represents an edge. The labels of the vertices are updated dynamically so that they are ordered lexicographically, i.e., the vertices with opinion $+$ have lower labels than the vertices with opinion $-$, and then by increasing type.}
	\label{fig3-mod3}
\end{figure}

\begin{figure}[h!]
	\includegraphics[width=5.1cm]{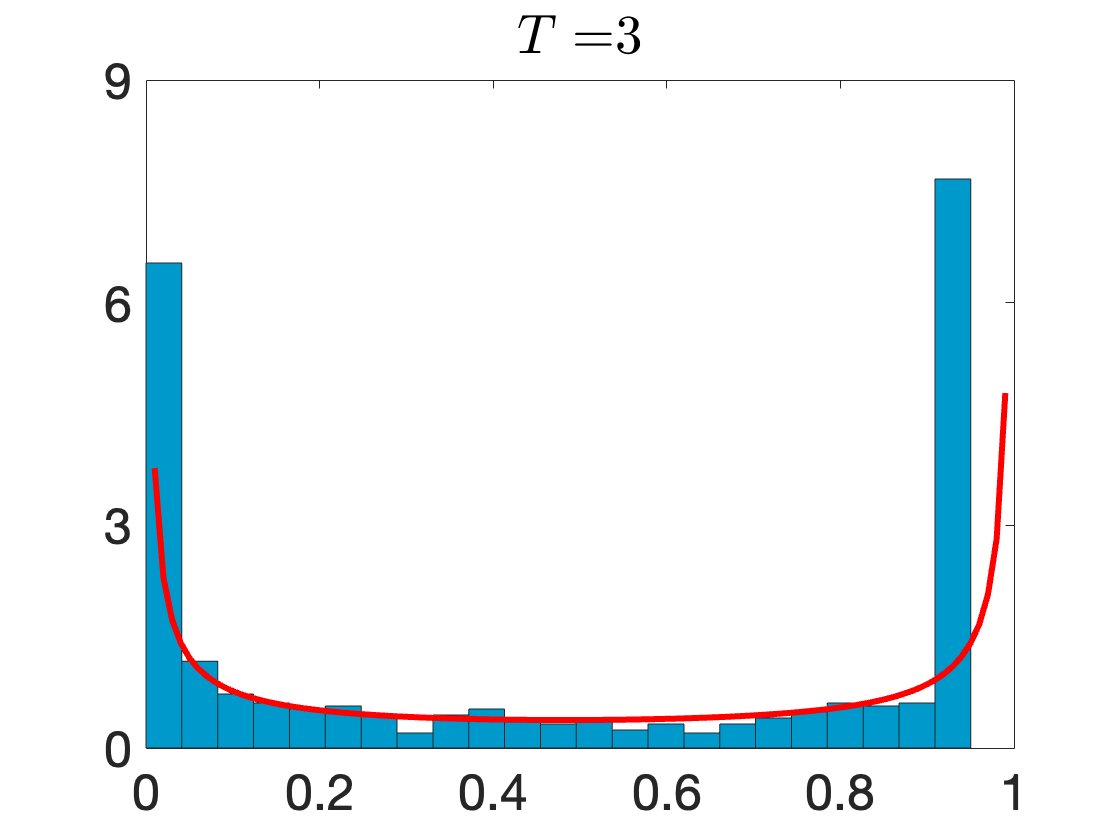}
	\includegraphics[width=5.1cm]{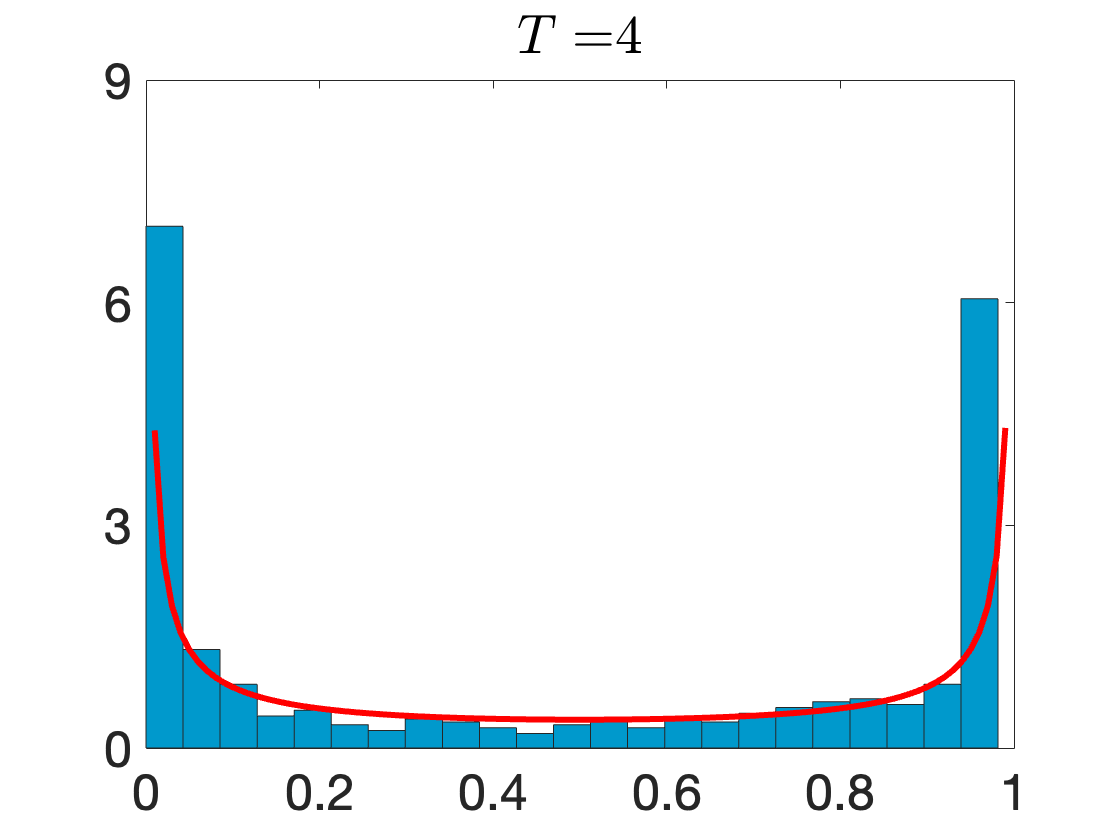}
	\includegraphics[width=5.1cm]{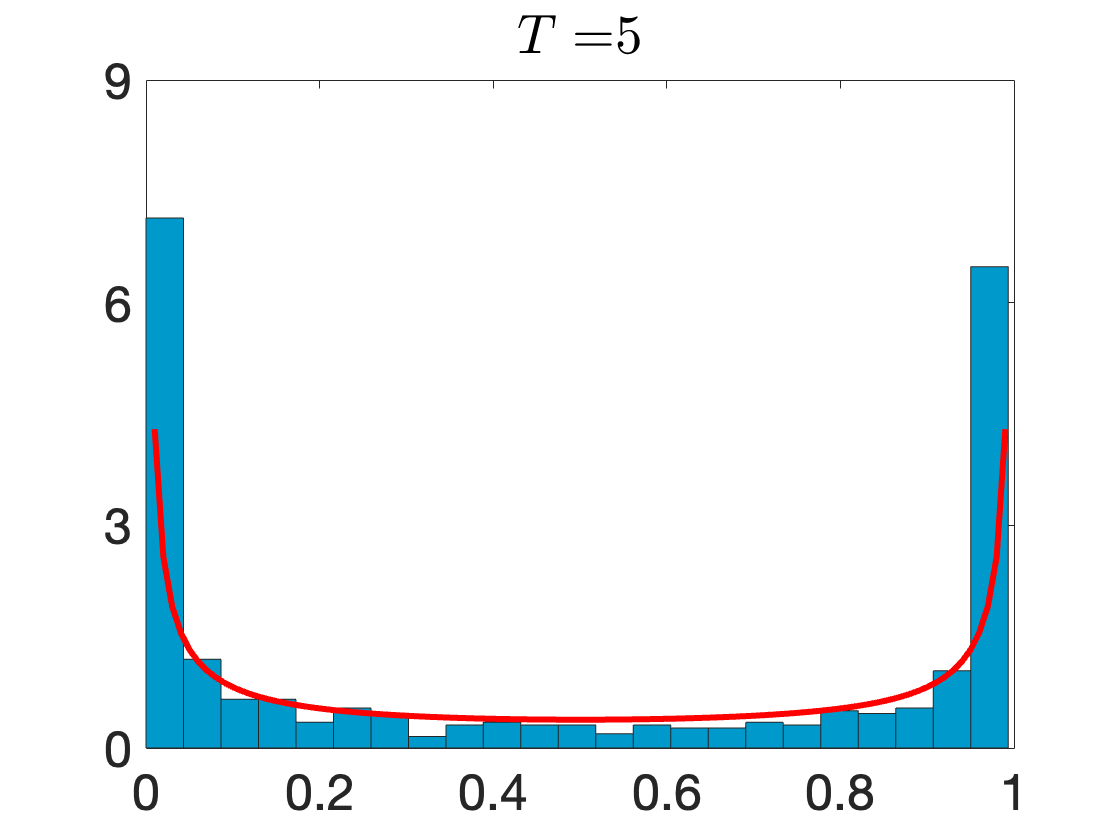}
	\caption{\small Empirical distribution of $(y_i(T))_{i\in [n]}$ when $n=600$ and $T=3,4,5$, for $q=1$, $\beta=0.5$, $\pi^g_+=\pi^g_-=0.7$ and $\pi^r_+=\pi^r_-=0.3$. For any $i\in[n]$, $y_i(0)$ is chosen uniformly at random in $[0,1]$, independently of the other vertices. The red line indicates the corresponding $\hbox{Beta}(\beta p_+,\beta(1-p_+))$ density function, where $p_+$ denotes the proportion of vertices holding opinion $+$. Simulations are based on a single run.}
	\label{fig:histmod3}
\end{figure}


\subsection{Numerical examples}
\label{mod3:num}

Suppose that $p_0=0.05$, all vertices are initially assigned opinion $+$ or $-$ with equal probability, and $y_i(0)$ is uniformly distributed in $[0,1]$ for any $i\in[n]$. We take $\beta=0.5$, $\pi_-^r=\pi_+^b=0.9$, and $\pi_+^r=\pi_-^b=0.1$. In Figs.\ \ref{fig1-mod3},\ref{fig2-mod3},\ref{fig3-mod3} we let $q=1,2,3$ respectively, and plot the empirical graphon when $n=150$ for $T=2,3,4$ (top row), and the corresponding functional law of large numbers implied by Theorem \ref{thm:graphonconvthird} (bottom row). By comparing these figures we observe that as $q$ increases the model becomes increasingly polarised, as intended. In Fig.~\ref{fig:histmod3} we let $q=1$ and compare the empirical distribution $(y_i(T))_{i\in[n]}$ when $n=600$ and $T=3,4,5$ to the limiting Beta distribution implied by Theorem \ref{thm:systempdeasymp3}. In this figure we observe that the empirical type distribution is well approximated by the limiting Beta distribution even for relatively small values of $T$.


\subsection{Proof of the main theorems}
\label{mod3:pr}


\subsubsection{Proof of Theorem \ref{thm:graphonconvthird}}
\label{sec:mim2}

To prove convergence of the process in the space of graphons, we follow the line of argument that we relied on in Section \ref{sec:coupling}, namely, we construct a mimicking process that satisfies Assumptions 1-3 in Section \ref{sec:genstrategy}, and that is sufficiently close to the original process in $L^1$. Since similar arguments apply, we only report the definition of the mimicking process and the description of the coupling.

\medskip\noindent
{\it Mimicking process.} Suppose that the process $(G_n^*(t))_{t\in[0,T]}$ is characterised by the following dynamics:
\begin{itemize}
	\item $G_n^*(0)$ is an ERRG with connection probability $p_0$.
	\item Each vertex $i$ holds opinion $+$ or $-$, and is assigned a rate-$\beta$ Poisson clock. Each time this clock rings, vertex $i$ updates its opinion as follows: with probability $\widetilde\alpha(t;y^*_i(t),\vec{v})$ it takes opinion $+$, where $y^*_i(t)$ denotes the type of vertex $i$ at time $t$ in $G_n^*(t)$. Otherwise, it takes opinion $-$.
	\item Each edge is assigned a rate-1 Poisson clock. Each time this clock rings, the edge $ij$ updates its state (either active or inactive): the edge is active with a probability that depends on the opinion of the selected vertices, namely, $2\left[\tfrac14(p(x^*_i(t))+p(x^*_j(t)))(2-p(x^*_i(t))-p(x^*_j(t)))\right]^q$, where $x^*_i(t)$ denotes the opinion of vertex $i$ at time $t$ in $G_n^*(t)$. Otherwise, edge $ij$ is inactive.
\end{itemize}

\medskip\noindent
{\it Description of the coupling.} Suppose that the outcomes of $(G_n(t))_{t\geq0}$ and $(G^*_n(t))_{t\geq0}$ are generated in the following manner.
\begin{itemize}
	\item For each $i\in[n]$, vertex $i$ is assigned the same (coupled) rate-$\beta$ Poisson clock in both processes. When the clock associated to vertex $i$ rings, generate an outcome $u$ of a $\hbox{Unif}(0,1)$ distribution.
	\begin{itemize}
		\item If $u\leq N_i^+(t)/N_i(t)$, then vertex $i$ takes opinion $+$, otherwise it takes opinion $-$, in $(G_n(t))_{t\geq0}$, where $N_i(t)$ and $N_i^+(t)$ denote the total number of neighbours of vertex $i$ and those having opinion $+$, respectively, in $(G_n(t))_{t\geq0}$.
		\item If $u\leq\widetilde\alpha(t;y_i^*(t),\vec{v})$, then vertex $i$ takes opinion $+$, otherwise it takes opinion $-$, in $(G^*_n(t))_{t\geq0}$.
	\end{itemize}
	\item Assign each edge the same (coupled) rate-$1$ Poisson clock in both processes. When the Poisson clock associated with edge $ij$ rings, generate an outcome $u$ of a $\hbox{Unif}(0,1)$ distribution. 
	\begin{itemize}
		\item If $u\leq 2\left[\tfrac14(p(x_i(t))+p(x_j(t)))(2-p(x_i(t))-p(x_j(t)))\right]^q$, then the edge $ij$ is active, otherwise it is inactive, in $(G_n(t))_{t\geq0}$.
		\item If $u\leq 2\left[\tfrac14(p(x^*_i(t))+p(x^*_j(t)))(2-p(x^*_i(t))-p(x^*_j(t)))\right]^q$, then the edge $ij$ is active, otherwise it is inactive, in $(G^*_n(t))_{t\geq0}$.\end{itemize} 
\end{itemize}


\subsubsection{Proof of Theorem~\ref{thm:systempdethird}}

In view of the discussion in Section \ref{sec:prep3}, the claim follows with similar arguments as those developed in Section \ref{sec:pdemodel2}.


\subsubsection{Proof of Theorem \ref{thm:systempdeasymp3}}

We start with case (i). The claimed Beta limit in \eqref{eq:limdensities3bis} follows by checking that
$\widetilde\alpha(\infty;u,\vec{v}) = p_+$,
where
\[
\widetilde\alpha(\infty;u,\vec{v})= \lim_{t\to\infty} \widetilde\alpha(t;u,\vec{v})
= \dfrac{\int_0^1 {\rm d}y\,f_+(\infty,y) \displaystyle \lim_{t\to\infty} H_R(t;y,u)}{\int_0^1 {\rm d}y\, [f_+(\infty,y) + f_-(\infty,y)]\, \displaystyle\lim_{t\to\infty}H_R(t;y,u)}.
\]
This turns out to be true when $\pi^g_+=\pi^g_-$ and $\pi^r_+=\pi^r_-$. To show that there are no other admissible analytic densities occurring in the limit, we argue as in the proof of Theorem \ref{thm:systempdeasymp2}\emph{(i)}. 

Finally, the proof of (ii) follows by arguing as in the proof of Theorem \ref{thm:systempdeasymp2}\emph{(ii)}. 


\section{Conclusions}
\label{sec:conclusions}

In this paper we proved functional laws of large numbers for three examples of the voter model on a dynamic random graph, both with one-way and two-way feedback. These functional laws of large numbers were used to visualise the evolution of the process, determine when both opinions are present for an extended period time (co-existence) or one opinion quickly becomes dominant (consensus), and to characterise the limiting densities of the vertex types in terms of Beta-distributions. To the best of our knowledge, this paper is the first to investigate co-evolutionary networks with the help of graphons. As such, several questions remain open:
\begin{itemize}
    \item As explained in Section \ref{sec:genstrategy}, the approach we develop seems robust and can possibly be applied to modifications of the voter model such as the contact process. Indeed, we believe that reducing a two-way feedback process to a one-way feedback mimicking process rests on a broader principle, which merits future research. 
    \item In Remark \ref{rem:diff}, we conjectured a \emph{diffusion limit} for the co-evolutionary network in Section \ref{sec:twoway1}. It remains to be determined whether an approach similar to Section \ref{sec:genstrategy} can be used to establish a diffusion limit for the model in Section \ref{sec:twoway1}, and more broadly for co-evolutionary models in general. We believe that the fact that diffusion limits require a different time scale (see Remark \ref{rem:diff}) will lead to difficulties in the equivalent of Lemma \ref{lem:coupling} and a different approach may be required.
    \item In Section \ref{sec:nonlin}, we described a co-evolutionary network where it was unclear how to write down a functional law of large numbers. It remains to be determined whether functional laws of large numbers can be established in such cases, and if so, what mathematical framework is required.
    \item It is natural to ask what happens for the {\it Potts} version of the  voter model, i.e., each individual can have an opinion in the set $\{1,...,m\}$ with $m \geq 3$. We believe that functional laws of large numbers for the densities of the $m$ opinions still hold and that the proof should be a slight generalisation of that presented in the present paper for $m=2$. Possibly there is a {\it phase transition} as a function of $m$. In particular, it is worth understanding under which assumptions there is coexistence of $m$ (or less) opinions, or consensus is attained, or polarisation occurs.
\end{itemize}


\appendix


\section{Preliminaries: graphons and Hoeffding's inequality}
\label{appA}

\par{\bf Graphons.}
Let $\cW$ be the space of functions $h\colon\,[0,1]^2 \to [0,1]$ such that $h(x,y) = h(y,x)$ for all $(x,y) \in [0,1]^2$, endowed with the {\it cut distance}
\begin{equation}
\label{eq:cutdist}
d_{\square}(h_1,h_2) = \sup_{S,T\subseteq[0,1]}\left|\int_{S \times T} \dx \dy\, [h_1(x,y)-h_2(x,y)]\right|, \quad h_1,h_2 \in \cW.
\end{equation}
On $\cW$, called the space of graphons, there is a natural equivalence relation $\sim$. Let $\Sigma$ be the space of measure-preserving bijections $\sigma\colon\, [0,1] \to [0,1]$. Then $h_1(x,y) \sim h_2(x,y)$ if $\delta_{\square}(h_1,h_2)=0$, where $\delta_{\square}$ is the \emph{cut metric} defined by 
\begin{equation}
\label{deltam}
\delta_{\square}(\tilde{h}_1,\tilde{h}_2) 
= \inf _{\sigma_1,\sigma_2 \in \Sigma} d_{\square}(h_1^{\sigma_1}, h_2^{\sigma_2}),
\qquad \tilde{h}_1,\tilde{h}_2 \in \widetilde\cW,
\end{equation}
with $h^\sigma(x,y)=h(\sigma x,\sigma y)$. This equivalence relation yields the quotient space $(\widetilde\cW,\delta_{\square})$, which is compact.

A finite simple undirected graph $G$ on $n$ vertices can be represented as a graphon $h^G\in\cW$ by setting
\begin{equation}
\label{eq:graphon}
h^G(x,y) := \left\{
\begin{array}{ll}
1 &\hbox{if there is an edge between vertex } \lceil nx\rceil \hbox{ and vertex } \lceil ny \rceil, \\
0 &\hbox{otherwise},
\end{array}
\right.
\end{equation}
which referred to as the {\it empirical graphon} associated with $G$, and has a block structure.

\medskip\noindent
{\bf Hoeffding's inequality.} If $X \sim{\rm Bin}(n,p)$, then, for any $a >0$,
\begin{equation}
	\label{eq:Hoeffgen}
\mathbb{P}(X \geq n(p+a)) \leq \exp{(-2na^2)}.
\end{equation}


\section{Proof of some lemmas of Section \ref{sec:oneway}}
\label{appB}

\begin{proof}[Proof of Lemma \ref{lmm:PDE1}]
The statement follows after applying the general strategy explained in Section \ref{sec:genstrategy} to derive the Kolmogorov forward equations starting from the generator \eqref{eq:gen1}. We need only to explicitly characterised the drift term. Suppose that a vertex has opinion $+$, then the corresponding drift term is 
\begin{equation}\label{eq:b+}
b(+, y) = 1 - y.
\end{equation}
Indeed, if vertex $i$ has opinion $+$ during the whole time interval $(t,t+\dd t)$, then 
\[
\begin{array}{ll}
	y_i(t+\dd t) &= \eee^{-(t+\dd t)}y_i(0) + \displaystyle\int_0^{t+\dd t} \dd s \, \eee^{-s} \mathbf{1}\{x_i(t+\dd t-s)=+\} \\
	&= \eee^{-(t+\dd t)}y_i(0) + \eee^{-(t+\dd t)} \displaystyle\int_0^{t+\dd t} \dd s \, \eee^s \mathbf{1}\{x_i(s)=+\} \approx y_i(t) - y_i(t)\dd t + \dd t,
\end{array}
\]
where $\approx$ means that the left and right hand sides are the same in the limit $\dd t \to0$. Thus, we conclude that $y_i'(t)=1-y_i(t)$. Similarly, we deduce that 
\begin{equation}\label{eq:b-}
b(-, y) = -y.
\end{equation}
As a consequence,
$$
\frac{\partial}{\partial t} f_+(t,u) - (1-u) \frac{\partial}{\partial u} f_+(t,u) = -\gamma_{+-} f_+(t,u) + \gamma_{-+} f_-(t,u),
$$
and 
$$
\frac{\partial}{\partial t} f_-(t,u) +u \frac{\partial}{\partial u} f_-(t,u) = -\gamma_{-+} f_-(t,u) +\gamma_{+-} f_+(t,u).
$$
which establish the claim. 
\end{proof}

\begin{proof}[Proof of Lemma \ref{Npower}]
The proof is by induction on $k$. The case $k=1$ is trivial. Suppose that the claim holds for $k$. For $k+1$, we find, using the induction hypothesis,
$$
\begin{array}{ll}
{\bar N}^{k+1}
&= (\gamma_{+-} + \gamma_{-+})^{k-1}(\gamma_{+-} A + \gamma_{-+} B)(\gamma_{+-} A^k + \gamma_{-+} B^k) \\[0.2cm]
&= (\gamma_{+-} + \gamma_{-+})^{k-1}(\gamma_{+-}^2 A^{k+1} + \gamma_{+-}\gamma_{-+} ABB^{k-1}
+ \gamma_{+-}\gamma_{-+} BAA^{k-1} + \gamma_{-+}^2 B^{k+1}) \\[0.2cm]
&= (\gamma_{+-} + \gamma_{-+})^{k-1}(\gamma_{+-}^2 A^{k+1} + \gamma_{+-}\gamma_{-+} B^{k+1}
+ \gamma_{+-}\gamma_{-+} A^{k+1} + \gamma_{-+}^2 B^{k+1}) \\[0.2cm]
&= (\gamma_{+-} + \gamma_{-+})^{k}(\gamma_{+-} A^{k+1} + \gamma_{-+} B^{k+1}),
\end{array}
$$
where the third equality follows from $AB^2=B$, $BA^2=A$ and \eqref{Apower}--\eqref{Bpower}.
\end{proof}

\begin{proof}[Proof of Lemma \ref{Nexp}]
By Lemma \ref{Npower}, we can write
$$
\begin{aligned}
{\rm e}^{\bar{N}t}
&= \mathbb{I}_2+\sum_{k\in\mathbb{N}}
\dfrac{(\gamma_{+-} + \gamma_{-+})^{k-1}t^k}{k!}(\gamma_{+-} A^k + \gamma_{-+} B^k) \\
&= \mathbb{I}_2 + \dfrac{\gamma_{+-}}{\gamma_{+-} + \gamma_{-+}} \sum_{k\in\mathbb{N}}
\dfrac{(\gamma_{+-} + \gamma_{-+})^{k}t^k}{k!} A^k
+ \dfrac{\gamma_{-+}}{\gamma_{+-} + \gamma_{-+}} \sum_{k\in\mathbb{N}}
\dfrac{(\gamma_{+-} + \gamma_{-+})^{k}t^k}{k!} B^k \\[0.4cm]
&= \mathbb{I}_2 + \dfrac{\gamma_{+-}}{\gamma_{+-} + \gamma_{-+}}({\rm e}^{t(\gamma_{+-} + \gamma_{-+}) A}-\mathbb{I}_2)
+ \dfrac{\gamma_{-+}}{\gamma_{+-}+\gamma_{-+}}({\rm e}^{t(\gamma_{+-} + \gamma_{-+}) B}-\mathbb{I}_2) \\[0.4cm]
&= \dfrac{1}{\gamma_{+-} + \gamma_{-+}}(\gamma_{+-}{\rm e}^{t(\gamma_{+-}+\gamma_{-+}) A}
+ \gamma_{-+}{\rm e}^{t(\gamma_{+-}+\gamma_{-+}) B}),
\end{aligned}
$$
which settles the claim. 
\end{proof}

\begin{proof}[Proof of Lemma \ref{expansion}]
Split the exponential series into even and odd terms, and use \eqref{Apower}-\eqref{Bpower}.
\end{proof}


\section{Additional lemmas for Section \ref{sec:twoway1}}
\label{appC}

\begin{lemma}
	\label{lem:setN}
	There exists $\ell>0$ such that $|{\cal N_{\ell}}|=0$ with asymptotically high probability, where the set ${\cal N}_{\ell}$ is defined in \eqref{eq:setN}.
\end{lemma}
\begin{proof}
	The proof follows after iteratively applying Hoeffding's inequality as done for $d_E(\cdot)$ and $d_V(\cdot)$ in the proof of Lemma \ref{lem:coupling} to get \eqref{eq:assumption}.
\end{proof}

\begin{lemma}\label{lmm:repalpha}
For any $s\in t_{\Delta}$,
	\[
	\alpha(s;y^*_i(s),\vec{v}) = \dfrac{{\int_0^{r^+(s)}} {\rm d}y \, g^{[F]}(s;y,F(s;y^*_i(s)))}{\int_0^1 {\rm d}y \, g^{[F]}(s;y,F(s; y^*_i(s)))}.
	\]
\end{lemma}
\begin{proof}
	By definition 
	\begin{equation}\label{eq:den1}
		\alpha(s;y^*_i(s),\vec{v}) = \dfrac{\int_0^1 {\rm d}y \, f_+(s,y) H(s;y,y^*_i(s))}{\int_0^1 {\rm d}y \, [f_+(s,y) + f_-(s,y)] H(s;y,y^*_i(s))}.
	\end{equation}
	The proof now follows by (1) equating the numerator of both expression and the denominator of both expressions separately, (2) using the fact that $H(t;u,v)$ is linear in $u$, and (3) using the equivalent expressions for expectation of a random variable $\mathbb{E}(X)=\int x f(x) {\rm d}x = \int_0^1 \bar F(u) {\rm d}u$.
	More precisely, since the function $H(t;\cdot,\cdot)$ is linear in both arguments, we can write
	\[
	H(s;y,y^*_i(s)) = \dfrac{y}{2}(\pi_+-\pi_-) + \widetilde H(s;y^*_i(s)),
	\]
	where the function $\widetilde H(s;y^*_i(s))$ does not depend on $y$. Thus, we get
	\[
	\begin{array}{ll}
		\displaystyle\int_0^1 {\rm d}y f_+(s,y) H(s;y,y^*_i(s)) &= \dfrac{\pi_+-\pi_-}{2} \displaystyle\int_0^1 {\rm d}y \, y f_+(s,y) + \widetilde H(s;y^*_i(s)) r^+(s) \\
		&= \displaystyle \int_0^{r^+(s)} {\rm d}y \left[ \dfrac{\pi_+-\pi_-}{2} \bar F(s;y) + \widetilde H(s;y^*_i(s)) \right] \\
		&= \displaystyle {\int_0^{r^+(s)}} {\rm d}y \, g^{[F]}(s;y,F(s;y^*_i(s))),
	\end{array}
	\]
	where the first equality follows from $\int_0^1 {\rm d}y \, f_+(s,y)=r^+(s)$ and the second equality follows from \[\int_0^1 {\rm d}y \, y f_+(s,y) = \int_0^{r^+(s)} {\rm d}y \, \bar F(s;y).\] The claim follows by applying similar computations to the denominator in \eqref{eq:den1}.
\end{proof}

\begin{lemma}
	\label{lmm:induction}
	Suppose that \eqref{eq:assumption} holds with $b_1,b_2,c_1,c_2 <\infty$ fixed and independent of $\Delta$ and $n$. Let $a_1 = \max\{ b_1, c_1\}$ and $a_2 =\max\{ b_2, c_2\}$. For any $i =0, \dots, T/\Delta$,
	\begin{equation}
		\begin{aligned}\label{eqn:At}
			d_V(t_i) &\leq n a_2 \Delta^2 \sum_{k=1}^i (a_1 \Delta)^{k-1} {i \choose k}, \\
			d_E(t_i) &\leq n^2 a_2 \Delta^2 \sum_{k=1}^i (a_1 \Delta)^{k-1} {i \choose k}.
		\end{aligned}
	\end{equation}
\end{lemma}
\begin{proof}
	We prove the statement by induction over $i$. Because $d_V(0)=d_E(t)=0$ the base case is immediate.
	We illustrate the induction step for $d_V(\cdot)$. If the above inequality holds, then we have 
	\begin{align*}
		d_V(t_{i+1})-d_V(t_i) &\leq \frac{\Delta}{n} a_1 \left[ n^2 a_2 \Delta^2 \sum_{k=1}^i (a_1 \Delta)^{k-1} {i \choose k} \right] + \Delta^2 a_2 n \\
		&= na_2 \Delta^2 \left[ \sum_{k=1}^i (a_1 \Delta)^k {i \choose k} +1 \right]
	\end{align*}
	This implies that 
	\begin{align*}
		d_V(t_{i+1}) &= d_V(t_i) + d_V(t_{i+1}-t_i) \\
		&\leq  \Delta^2 \sum_{k=1}^i (a_1 \Delta)^{k-1} {i \choose k} + na_2 \Delta^2 \left[ \sum_{k=1}^i (a_1 \Delta)^k {i \choose k} +1 \right] \\
		&= n a_2 \Delta^2 \left[ \sum^{i+1}_{k=1} (a_1 \Delta)^{k-1} {i+1 \choose k} \right].
	\end{align*}
	The induction step for $d_E(\cdot)$ is similar.
\end{proof}




\begin{thebibliography}{20}
	
\bibitem{ACAHFM22}
H.F.\ de Arruda, F.M.\ Cardoso, G.F.\ de Arruda, A.R.\ Hern\'{a}ndez, L.\ da Fontoura Costa, Y.\ Moreno.
\newblock Modelling how social network algorithms can influence opinion polarization.
\newblock {\em Information Sciences} 588 (2022) 265--278.

\bibitem{AdHR24}
S.\ Athreya, F.\ den Hollander, A.\ R\"ollin.
\newblock Interplay of vertex and edge dynamics for dense random graphs.
\newblock Work in progress.

\bibitem{AGHdH18}
L.\ Avena, H.\ G{\"u}lda{\c s}, R.\ van der Hofstad, F.\ den Hollander.
\newblock	Mixing times of random walks on dynamic configuration models.
\newblock	{\em The Annals of Applied Probability} 28 (2018) 1977--2002.

\bibitem{AGHdH19}
L.\ Avena, H.\ G{\"u}lda{\c s}, R.\ van der Hofstad, F.\ den Hollander.
\newblock	Random walks on dynamic configuration models: A trichotomy.
\newblock {\em Stochastic Processes and their Applications} 129 (2019) 3360--3375.

\bibitem{ABHdHQ24}
L.\ Avena, R.\ Baldasso, R.S.\ Hazra, F. den Hollander, M.\ Quattropani.
\newblock Discordant edges for the voter model on regular random graphs.
\newblock {\em ALEA, Latin American Journal of Probability and Mathematical Statistics} 21 (2024) 431--464.

\bibitem{ABHdHQ24pr}
L.\ Avena, R.\ Baldasso, R.S.\ Hazra, F. den Hollander, M.\ Quattropani.
\newblock The voter model on random regular graphs with random rewiring.
\newblock Work in progress.

\bibitem{ACHQ23pr}
L.\ Avena, F.\ Capannoli, R.S.\ Hazra, M.\ Quattropani.
\newblock Meeting, coalescence and consensus time on random directed graphs.
\newblock {\em The Annals of Applied Probability} 34(5) (2024) 4940--4997.

\bibitem{BS17}
R.\ Basu, A.\ Sly. 
\newblock Evolving voter model on dense random graphs. 
\newblock {\em The Annals of Applied Probability} 27 (2017) 1235--1288.

\bibitem{BVP24}
H.M.\ Borges, V.V.\ Vasconcelos, F.L.\ Pinheiro.
\newblock How social rewiring preferences bridge polarized communities.
\newblock {\em Chaos, Solitons and Fractals} 180 (2024) 114594.

\bibitem{BdHM22pr}
P.\ Braunsteins, F. den Hollander, M.\ Mandjes.
\newblock Graphon-valued processes with vertex-level fluctuations.
\newblock Preprint at arXiv:2209.01544 (2022).

\bibitem{Cap2024}
F.\ Capannoli.
\newblock Evolution of discordant edges in the voter model on random sparse digraphs.
\newblock Preprint at arXiv:2407.06318 (2024).

\bibitem{CK20} 
\v{C}ern\'y, J., and Klimovsky, A.
Markovian dynamics of exchangeable arrays.
In: \emph{Genealogies of Interacting Particle Systems}.
Lecture Notes Series, Institute for Mathematical Sciences, National University of Singapore (2020) 209--228.

\bibitem{CSW22}
S.\ Chatterjee, D.\ Sivakoff, M.\ Wascher.
\newblock The effect of avoiding known infected neighbors on the persistence of a recurring infection process.
\newblock {\em Electronic Journal of Probability} 27.109 (2022) 1--40.

\bibitem{CCC16}
Y.-T.\ Chen, J.\ Choi, J.T.\ Cox.
\newblock On the convergence of densities of finite voter models to the Wright-Fisher diffusion.
\newblock {\em Annales de l'Institut Henri Poincar\'e} 52 (2016) 286--322.

\bibitem{CDFR18}
C.\ Cooper, M.\ Dyer, A.\ Frieze, N.\ Rivera.
\newblock Discordant voting processes on finite graphs.
\newblock {\em SIAM Journal on Discrete Mathematics} 32.4 (2018).

\bibitem{CR18}
C.\ Cooper, N.\ Rivera.
\newblock Threshold behaviour of discordant voting on the complete graph.
\newblock {\em Journal of Discrete Algorithms} 50 (2018) 10--22.

\bibitem{C16}
H.\ Crane.
Dynamic random networks and their graph limits. 
\emph{The Annals of Applied Probability,} 26 (2016), 691--721.

\bibitem{DSCSQ17}
M.\ Del Vicario, A.\ Scala, G.\ Caldarelli, H.E.\ Stanley, W.\ Quattrociocchi.
\newblock Modeling confirmation bias and polarization.
\newblock {\em Scientific Reports} 7 (2017) 40391.
 
\bibitem{DGLV12}
R.\ Durrett, J.P. Gleeson,  A.L.\ Lloyd, P.J.\ Mucha, F.\ Shi, D.\ Sivakoff, J.E.S.\ Socolar, C.\ Varghese.
\newblock Graph fission in an evolving voter model. 
\newblock {\em Proceedings of the National Academy of Sciences} 109 (2012) 3682--3687.

\bibitem{EK86}
S.N.\ Ethier, T.G.\ Kurtz.
\newblock Markov Processes: Characterization and Convergence.
\newblock Wiley Series in Probability and Statistics, 1986.

\bibitem{G22}
Garbe, F., Hladk\'y, J., Sileikis, M., and Skerman, F. 
From flip processes to dynamical systems on graphons. 
\newblock Preprint at arXiv:2201.12272  (2022).

\bibitem{GZ06}
S.\ Gil, D.H.\ Zanette. 
\newblock Coevolution of agents and networks: Opinion spreading and community disconnection. 
\newblock {\em Physical Review Letters A} 356 (2006) 89--94. 

\bibitem{HPZ11} 
A.D.\ Henry, P.\ Pralat, C.Q.\ Zhang. 
\newblock Emergence of segregation in evolving social networks. 
\newblock {\em Proceedings of the National Academy of Sciences} 108 (2011) 8605--8610.

\bibitem{HS20}
J.\ Hermon, P. Sousi.
\newblock	A comparison principle for random walk on dynamical percolation.
\newblock {\em The Annals of Probability} 48 (2020) 2952–2987.

\bibitem{HN06} 
P.\ Holme, M.\ Newman. 
\newblock Nonequilibrium phase transition in the coevolution of networks and opinions.
\newblock {\em Physical Review E} 74 (2006) 056108.

\bibitem{HL75}
A.\ Holley. T.\. Liggett. 
\newblock Ergodic theorems for weakly interacting infinite systems and the voter model. 
\newblock {\em The Annals of Probability} 3 (1975) 643–-663.

\bibitem{IKKB09}
G.\ Iniguez, J.\ Kert\'esz, K.K.\ Kaksi, R.A.\ Barrio. 
\newblock Opinion and community formation in coevolving networks. 
\newblock {\em Physical Review E} 80 (2009) 066119.

\bibitem{JM17}
E.\ Jacob, P.\ M\"{o}rters.
\newblock The contact process on scale--free networks evolving by vertex updating.
\newblock {\em Royal Society Open Science} 4 (2017) 170081.

\bibitem{JTSZH20}
A.\ J\k{e}drzejewski, J.\ Toruniewska, K.\ Suchecki, O.\ Zaikin, J.A.\ Ho\l{}yst.
\newblock Spontaneous symmetry breaking of active phase in coevolving nonlinear voter model.
\newblock {\em Physical Review E} 102 (2020) 042313.

\bibitem{KB08a}
B.\ Kozma, A.\ Barrat.
\newblock Consensus formation on adaptive networks. 
\newblock {\em Physical Review E} 77 (2008) 016102.

\bibitem{KB08b}
B.\ Kozma, A.\ Barrat. 
\newblock Consensus formation on coevolving networks: Groups’ formation and structure. 
\newblock {\em Journal of Physics A Mathematics General} 41 (2008) 224020.

\bibitem{K1875}
S.\ von Kowalevsky.
\newblock Zur Theorie der partiellen Differentialgleichung.
\newblock {\em Journal f\"{u}r die reine und angewandte Mathematik} 80 (1875) 1--32.

\bibitem{L85}
T.\ Liggett. 
{\it Interacting Particle Systems}. Springer, New York, USA (1985).

\bibitem{LHAJS20}
J.\ Liu, S.\ Huang, N.M.\ Aden, N.F.\ Johnson, C.\ Song.
\newblock Emergence of polarization in coevolving networks.
\newblock {\em Physical Review Letters} 130 (2020) 037401.

\bibitem{MKF03}
M.W.\ Macy, J.A.\ Kitts, A.\ Flache, S.\ Benard.
\newblock Polarization in Dynamic Networks: A Hopfield model of Emergent Structure.
\newblock In: Dynamic Social Network Modeling and Analysis, pp.\ 162–173. 
\newblock Washington, DC: National Academic Press 2003.

\bibitem{M23}
B.\ Min.
\newblock Coevolutionary dynamics of group interactions: coevolving nonlinear voter model.
\newblock {\em Frontiers in Complex Systems} 1 (2023) 1298265.

\bibitem{MM19}
B.\ Min, M.S.\ Miguel.
\newblock Multilayer coevolution dynamics of the nonlinear voter model.
\newblock {\em New Journal of Physics} 21 (2019) 035004.

\bibitem{PTN06}
J.M.\ Pacheco, A.\ Traulsen, M.A.\ Nowak. 
\newblock Coevolution of strategy and structure in complex networks with dynamical linking. 
\newblock {\em Physical Review Letters} 97 (2006) 258103.

\bibitem{PSS15}
Y.\ Peres, A.\ Stauffer, J.E.\ Steif.
\newblock Random walks on dynamical percolation: mixing times, mean squared displacement and hitting times
\newblock {\em Probability Theory and Related Fields} 162 (2015) 487--530.

\bibitem{PSS20}
Y.\ Peres, P.\ Sousi, J.E.\ Steif.
\newblock Mixing time for random walk on supercritical dynamical percolation.
\newblock {\em Probability Theory and Related Fields} 176 (2020) 809--849.

\bibitem{RG17}
T.\ Raducha, T.\ Gubiec.
\newblock Coevolving complex networks in the model of social interactions.
\newblock {\em Physica A} 471 (2017) 427--435.

\bibitem{RM20}
T.\ Raducha, M.S.\ Miguel.
\newblock Emergence of complex structures from nonlinear interactions and noise in coevolving networks.
\newblock {\em Scientific Reports} 10 (2020) 15660.

\bibitem{R}
R\'ath, B. 
Time evolution of dense multigraph limits under edge-conservative preferential attachment dynamics.
\emph{Random Structures \& Algorithms} {41} (2012) 365--390.

\bibitem{RZ23}
A.\ R\"{o}llin, Z.-S.\ Zhang.
\newblock Dense multigraphon--valued stochastic processes and edge--changing dynamics in the configuration model.
\newblock {\em The Annals of Applied Probability} 33 (2023) 3207--3239.

\bibitem{SCPCFM21}
K.\ Sasahara, W.\ Chen, H.\ Peng, G.L.\ Ciampaglia, A.\ Flammini, F.\ Menczer.
\newblock Social influence and unfollowing accelerate the emergence of echo chambers.
\newblock {\em Journal of Computational Social Science} 4 (2021) 381--402.

\bibitem{ST20}
P.\ Sousi, S.\ Thomas.
\newblock Cutoff for random walk on dynamical {E}rd{\H o}s--{R}{\'e}nyi graph.
\newblock {\em Annales Institut Henri Poincar{\'e}, Probability and Statistics} 56 (2020) 2745--2773.
	
\bibitem{WW23pr}
X.\ Wang, B.\ Wu.
\newblock A robust way to speed up consensus via adaptive social networks.
\newblock Preprint at arXiv:2312.05041 (2023).

\bibitem{ZG06}
D.H.\ Zanette, S.\ Gil. 
\newblock Opinion spreading and agent segregation on evolving networks. 
\newblock {\em Physica D} 224 (2006) 156--165.

\end{thebibliography}
\end{document}